\documentclass{memo-l}
\usepackage[utf8]{inputenc}
\usepackage[T1]{fontenc}
\usepackage{amsmath,amssymb,amsthm}
\usepackage{graphicx}
\usepackage[usenames,dvipsnames]{color}
\usepackage{tikz}
\usetikzlibrary{positioning}
\usetikzlibrary{calc}
\usepackage[pdfusetitle]{hyperref}
\hypersetup{colorlinks,linkcolor=blue,citecolor=cyan}
\usepackage[nameinlink, capitalize]{cleveref}
\crefname{subsection}{subsection}{subsections}
\usepackage{tikz-cd}
\usepackage{adjustbox}
\usepackage{enumitem}
\setlist{font=\normalfont}
\usepackage{multirow}
\usepackage{physics}

\newtheorem{thm}{Theorem}[chapter]
\newtheorem{cor}[thm]{Corollary}
\newtheorem{lem}[thm]{Lemma}
\newtheorem{prob}[thm]{Problem}
\newtheorem{prop}[thm]{Proposition}
	 	
\theoremstyle{definition}
\newtheorem{defi}[thm]{Definition}

\theoremstyle{remark}
\newtheorem{remark}[thm]{Remark}

\numberwithin{section}{chapter}
\numberwithin{equation}{chapter}

\renewcommand{\epsilon}{\varepsilon}
\renewcommand{\phi}{\varphi}
\renewcommand{\restriction}[2]{#1 \upharpoonright #2}

\newcommand{\A}{\mathcal{A}}
\newcommand{\ac}{\mathbf}
\newcommand{\al}{\mathcal{A}}
\newcommand{\amen}{\mathrm{amen}}
\newcommand{\ap}{\mathcal{AE}}
\newcommand{\aper}{\mathrm{aper}}
\newcommand{\aph}{\mathcal{AH}}
\newcommand{\apt}{\mathcal{AT}}
\newcommand{\B}{\mathcal{B}}
\newcommand{\baire}{\mathbb{N}^{\mathbb{N}}}
\newcommand{\bs}{\boldsymbol}
\newcommand{\btu}{\bigtriangleup}
\newcommand{\C}{\mathbb{C}}
\newcommand{\cantor}{2^{\mathbb{N}}}
\newcommand{\car}{\curvearrowright}
\newcommand{\CBER}{\textup{CBER}}
\newcommand{\comp}{\mathrm{comp}}
\newcommand{\concat}{\string^}
\newcommand{\e}{\mathcal{E}}
\newcommand{\F}{\mathbb{F}}
\newcommand{\fin}{\mathrm{fin}}

\newcommand{\free}{\mathrm{free}}
\newcommand{\freeHyp}{\mathrm{freeHyp}}
\newcommand{\freeMeasHyp}{\mathrm{freeMeasHyp}}

\newcommand{\hra}{\hookrightarrow}
\newcommand{\hyp}{\mathrm{hyp}}
\newcommand{\I}{\mathbb{I}}
\newcommand{\id}{\mathrm{id}}

\newcommand{\inc}{\subseteq_B}

\newcommand{\mc}{\mathcal}
\newcommand{\measAmen}{\mathrm{measAmen}}
\newcommand{\measHyp}{\mathrm{measHyp}}
\newcommand{\mr}{\mathrel}
\newcommand{\N}{\mathbb{N}}
\newcommand{\ol}{\overline}
\newcommand{\PI}{\mathbf{\Pi}}
\newcommand{\pinc}{\subset_B}
\newcommand{\Q}{\mathbb{Q}}
\newcommand{\R}{\mathbb{R}}
\newcommand{\SIGMA}{\mathbf{\Sigma}}
\newcommand{\sm}{\mathrm{sm}}
\newcommand{\sq}{\sqsubseteq}
\newcommand{\uhr}{{\upharpoonright}}
\newcommand{\T}{\mathbb{T}}
\newcommand{\thra}{\twoheadrightarrow}
\newcommand{\tree}{\mathrm{Tr}}

\newcommand{\WF}{\mathrm{WF}}
\newcommand{\Z}{\mathbb{Z}}

\DeclareMathOperator{\Act}{Act}
\DeclareMathOperator{\Ap}{Ap}
\DeclareMathOperator{\Aut}{Aut}
\DeclareMathOperator{\CEINV}{CEINV}
\DeclareMathOperator{\CInd}{CInd}
\DeclareMathOperator{\Conv}{Conv}
\DeclareMathOperator{\EINV}{EINV}
\DeclareMathOperator{\Fin}{Fin}
\DeclareMathOperator{\Fr}{Fr}
\DeclareMathOperator{\grph}{graph}
\DeclareMathOperator{\Hom}{Hom}
\DeclareMathOperator{\Homeo}{Homeo}
\DeclareMathOperator{\im}{im}
\DeclareMathOperator{\Imagpart}{Im}
\DeclareMathOperator{\INV}{INV}
\DeclareMathOperator{\MGR}{MGR}
\DeclareMathOperator{\Prob}{Prob}
\DeclareMathOperator{\Realpart}{Re}
\DeclareMathOperator{\Sh}{Sh}
\DeclareMathOperator{\SL}{SL}
\DeclareMathOperator{\St}{St}
\DeclareMathOperator{\Stab}{Stab}
\DeclareMathOperator{\Supp}{Supp}

\makeindex

\begin{document}

\frontmatter

\title{Realizations of countable Borel equivalence relations}

\author{J. R. Frisch}
\address{Department of Mathematics, University of California, San Diego, La Jolla, CA 92093, USA}
\curraddr{}
\email{joshfrisch@gmail.com}
\thanks{JRF was partially supported by NSF Grants DMS-1464475, DMS-1950475 and DMS-2102838.}

\author{A. S. Kechris}
\address{Department of Mathematics, California Institute of Technology, Pasadena, CA 91125, USA}
\curraddr{}
\email{kechris@caltech.edu}
\thanks{ASK was partially supported by NSF Grants DMS-1464475 and DMS-1950475.}

\author{F. Shinko}
\address{Department of Mathematics, University of California, Berkeley, Berkeley, CA 94720, USA}
\curraddr{}
\email{forteshinko@berkeley.edu}
\thanks{FS was partially supported by NSF Grants DMS-1464475 and DMS-1950475.}

\author{Z. Vidny\'anszky}
\address{E\"otv\"os Lor\'and University, Institute of Mathematics, P\'azm\'any P\'eter stny. 1/C}
\curraddr{}
\email{zoltan.vidnyanszky@ttk.elte.hu}
\thanks{ZV was supported by NKFIH grants 113047 and 129211.}

\date{}



\dedicatory{
    We would like to thank
    Scot Adams,
    Ronnie Chen,
    Clinton Conley,
    Jingyin Huang,
    Andrew Marks,
    Petr Naryshkin,
    Severin Mejak,
    Aristotelis Panagiotopoulos,
    Brandon Seward,
    Michael Wolman and
    Andy Zucker
    for many helpful suggestions.
}

\begin{abstract}
    We study topological realizations of countable Borel equivalence relations,
    including realizations by continuous actions of countable groups,
    with additional desirable properties.
    Some examples include minimal realizations on any perfect Polish space,
    realizations as $K_\sigma$ relations,
    and realizations by continuous actions on the Baire space.
    We also consider questions related to
    realizations of specific important equivalence relations,
    like Turing and arithmetical equivalence.
    We focus in particular on the problem of realization
    by continuous actions on compact spaces and more specifically subshifts.
    This leads to the study of properties of subshifts,
    including universality of minimal subshifts,
    and a characterization of amenability
    of a countable group in terms of subshifts.
    Moreover we consider a natural universal space for actions and equivalence relations
    and study the descriptive and topological properties
    in this universal space of various properties,
    like, e.g.,
    compressibility, amenability or hyperfiniteness.
\end{abstract}

\maketitle

\tableofcontents

\mainmatter

\chapter{Introduction}\label{intro}
\section{Topological and continuous action realizations}\label{intro-top}
This paper is a contribution to the theory of countable Borel equivalence relations (CBER),
a recent survey of which can be found in \cite{Kec25}.
One of our main concerns is the subject of well-behaved, in some sense, realizations of CBER.
Given CBER $E, F$ on standard Borel spaces $X,Y$, resp.,
a Borel isomorphism of $E$ with $F$ is a Borel bijection $f \colon X \to Y$
which takes $E$ to $F$.
If such $f$ exists, we say that $E,F$ are \textbf{Borel isomorphic}, in symbols $E \cong_B F$.
A realization of a CBER $E$ is any CBER $F \cong_B E$.
We will be looking for realizations that have various desirable properties. 

To start with, a \textbf{topological realization} of $E$ is
an equivalence relation $F$ on a \textit{Polish space} $Y$ such that $E \cong_B F$,
in which case we say that $F$ is a topological realization of $E$ in the space $Y$.
It is clear that every $E$ admits a topological realization in some Polish space,
but we will look at topological realizations that have additional properties. 

Recall here the Feldman-Moore theorem that asserts that
every CBER is induced by a Borel action of a countable group
(see \Cref{theorem-feldman-moore}).
By \cite[13.11]{Kec95},
there is a Polish topology with the same Borel structure in which this action is continuous.
Thus every CBER admits a topological realization in some Polish space,
which is induced by a continuous action of some countable (discrete) group.
We will look again at such \textbf{continuous action realizations}
for which the space and the action have additional properties.

\textit{
    To avoid uninteresting situations,
    unless it is otherwise explicitly stated or clear from the context,
    all the standard Borel or Polish spaces below will be uncountable
    and all {\CBER} will be \textup{\textbf{aperiodic}},
    i.e., have infinite classes.
    We will denote by $\boldsymbol{\ap}$
    the class of all aperiodic {\CBER} on uncountable standard Borel spaces.
}

Concerning topological realizations, we first show the following (in \cref{3.1}),
where a minimal topological realization is one
in which all equivalence classes are dense:
\begin{thm}
    For every equivalence relation $E \in \ap$ and every perfect Polish space $Y$,
    there is a minimal topological realization of $E$ in $Y$.
\end{thm}
This has in particular as a consequence a stronger new version of a \textit{marker lemma}
(for the original form of the Marker Lemma see, e.g., \cite[Theorem 3.15]{Kec25}).
Let $E$ be a CBER on a standard Borel space $X$.
A \textbf{Lusin marker scheme} for $E$ is a family
$\{A_s\}_{s \in \N^{<\N}}$ of Borel sets such that
\begin{enumerate}[label=(\roman*)]
    \item
        $A_\emptyset = X$;
    \item \label{marker-scheme-union-condition}
        $\{A_{s \concat n}\}_n$ are pairwise disjoint and
        $\bigsqcup_n A_{s \concat n} \subseteq A_s$;
    \item
        Each $A_s$ is a complete section for $E$ (i.e., it meets every $E$-class).
\end{enumerate}

We have two types of Lusin marker schemes:
\begin{enumerate}[label=(\arabic*)]
    \item
        The Lusin marker scheme $\{A_s\}_{s \in \N^{<\N}}$ for $E$
        is of \textbf{type I} if in \ref{marker-scheme-union-condition} above,
        we actually have that $\bigsqcup_n A_{s \concat n} = A_s$
        and moreover the following holds:
        \begin{enumerate}[label=(\roman*), start=4]
            \item \label{item-markerschemei}
                For each $x \in \baire$,
                $\bigcap_n A_{x \uhr n}$ is a singleton.
        \end{enumerate}
        (Then in this case,
        for each $x \in \baire$,
        $A^x_n = A_{x \uhr n} \setminus \bigcap_n A_{x \uhr n}$
        is a vanishing sequence of markers
        (i.e., $\bigcap_n A^x_n = \emptyset$).) 
    \item
        The Lusin marker scheme $\{A_s\}_{s \in \N^{<\N}}$
        for $E$ is of \textbf{type II} if it satisfies the following:
        \begin{enumerate}[label=(\roman*)', start=4]
            \item
                If for each $n$,
                $B_n = \bigsqcup\{A_s : s\in \N^n\}$,
                then $\{B_n\}$ is a vanishing sequence of markers.
        \end{enumerate}
\end{enumerate}

We now have (see \cref{3.3}):
\begin{thm}
    Every $E \in \ap$ admits
    a Lusin marker scheme of type \textup{I} and
    a Lusin marker scheme of type \textup{II}.
\end{thm}

We next look at continuous action realizations.
One such realization of $E \in \ap$ would be
a realization $F$ on a compact Polish space,
where $F$ is generated by a continuous action of a countable (discrete) group.
We call these \textbf{compact action realizations}.
A \textbf{minimal, compact action realization} is a compact action realization
in which the group acts minimally, i.e., all the orbits are dense.
Finally, for each countable group $\Gamma$ and topological space $X$,
consider the shift action of $\Gamma$ on $X^\Gamma$
(see \Cref{prelim-action}).
The restriction of this action to a nonempty invariant closed set
is called a \textbf{subshift} of $X^\Gamma$.
We often identify a subshift with the underlying closed set.
Finally $\F_n$, $1 \le n \le \infty$, is the free group with $n$ generators.

Excluding the case of smooth CBERs
(see \Cref{prelim})
for which such a realization is impossible,
we show the following (in \cref{3.9}).
\begin{thm}\label{1.3}
    Every non-smooth hyperfinite equivalence relation in $\ap$ has a minimal, compact action realization.
    In fact this realization can be taken to be a subshift of $2^{\F_2}$
    if the equivalence relation is compressible and a subshift of $2^\Z$ otherwise.
\end{thm}

We discuss other cases of CBER which admit such realizations in \cref{realizations-cpt}.

For each countable group $\Gamma$ and standard Borel space $X$,
let $E(\Gamma, X)$ be the equivalence relation
induced by the shift action of $\Gamma$ on $X^\Gamma$.
Let $\Ap(X^\Gamma)$ be the \textbf{aperiodic part} of $X^\Gamma$,
i.e., the set of points $x\in X^\Gamma$ with infinite orbit,
and let $E^{\mathrm{ap}}(\Gamma, X)$ be the restriction of $E(\Gamma, X)$ to $\Ap(X^\Gamma)$.
Let also $\Fr(X^\Gamma)$ be its \textbf{free part},
i.e., the set of points $x\in X^\Gamma$ such that $\gamma\cdot x \neq x$
for all $\gamma \in \Gamma \setminus \{1\}$.
Denote by $F(\Gamma, X)$ the restriction of $E(\Gamma, X)$ to $\Fr(X^\Gamma)$.
Every aperiodic (i.e., having infinite orbits) Borel action of $\Gamma$ on a standard Borel space
is Borel isomorphic to the restriction of the shift action to an invariant Borel subset of $\Ap((\cantor)^\Gamma)$,
and similarly every free Borel action of $\Gamma$ on a standard Borel space
is Borel isomorphic to the restriction of the shift action to an invariant Borel subset of $\Fr((\cantor)^\Gamma)$

Also a CBER is \textbf{universal} if every CBER can be Borel reduced to it.
As opposed to \cref{1.3}, the next results (see \cref{3166} and \cref{3333})
show that some very complex CBER have compact action realizations.

\begin{thm}\label{1.4}
    \leavevmode
    \begin{enumerate}[label=(\arabic*)]
        \item
            For every infinite countable group $\Gamma$,
            $F(\Gamma, \cantor)$ admits a compact action realization.
            If $\Gamma$ is also finitely generated,
            then $E^{\mathrm{ap}}(\Gamma, \cantor)$ admits a compact action realization.
            In fact in both cases,
            such a realization can be taken to be a subshift of $(\cantor)^\Gamma$.
        \item
            Every compressible, universal {\CBER} admits a compact action realization.
            In fact such a realization can be taken to be a minimal subshift of $2^{\F_4}$.
    \end{enumerate}
\end{thm}

In particular,
it follows that arithmetical equivalence $\equiv_A$
on $\cantor$ has a compact action realization,
but it is unknown if Turing equivalence $\equiv_T$ has such a realization.
More generally,
we do not know whether \textit{every} non-smooth CBER
has a compact action realization.
We also do not know if \textit{every} non-smooth CBER
even admits some other kinds of realizations,
for example transitive
(i.e., having at least one dense orbit)
continuous action realizations
on arbitrary or special types of Polish spaces.
These problems as well as the situation with smooth CBER
in such realizations are discussed in \cref{realizations-cts}.

In \Cref{realizations-compress} we discuss some special properties of continuous actions
of countable groups on compact Polish spaces,
related to compressibility and paradoxical decompositions,
that may be relevant to compact action realizations.

Returning to Turing equivalence, in \cref{realizations-turing},
we discuss topological realizations of Turing equivalence $\equiv_T$
and show that it admits a Baire class $2$ isomorphism
to an equivalence relation given by a continuous group action
on the Baire space $\baire$.
We do not know if this can be improved to Baire class $1$
but we also show that no such isomorphism can be below
the identity on a cone of Turing degrees.

\section{Subshifts}\label{intro-subshifts}

Related to \cref{1.4},
we call a countable group $\Gamma$ \textbf{minimal subshift universal}
if there is a minimal subshift of $2^\Gamma$ on which
the restriction of the shift equivalence relation is universal.
Then we have, see \cref{univFlow} and \cref{3333}:
\begin{thm}
    Let $\Gamma$ and $\Lambda$ be infinite groups,
    where $\Lambda$ admits a Borel action on a standard Borel space
    whose induced equivalence relation is universal
    (e.g., any group containing $\F_2$).
    Then the wreath product $\Gamma \wr \Lambda$ is minimal subshift universal.
    In particular, $\F_3$ is minimal subshift universal.
\end{thm}
We do not know if $\F_2$ is minimal subshift universal.

It is well known that a countable group $\Gamma$ is \textbf{amenable}
iff every continuous action of $\Gamma$ on a compact space
admits an invariant Borel probability measure.
Call a class $\mathcal F$ of such actions
a \textbf{test for amenability} for $\Gamma$
if $\Gamma$ is amenable provided that every action in $\mathcal F$
admits an invariant Borel probability measure.
In \cite{GH97} it is shown that the class of actions on $\cantor$
is a test for amenability for any group.
Equivalently this says that the class of all subshifts of $(\cantor)^\Gamma$
is a test of amenability for $\Gamma$.
It turns out that the strongest result along these lines is actually true,
namely that the class of all subshifts of $2^\Gamma$
is a test of amenability for $\Gamma$,
see \cref{compressibleSubshift}.
This gives another characterization of amenability.

\begin{thm}\label{amenChar}
    Let $\Gamma$ be a countable group.
    Then $\Gamma$ is amenable iff every subshift of $2^\Gamma$
    admits an invariant Borel probability measure.
\end{thm}

Our proof of \Cref{amenChar} is based on an explicit construction
of a subshift of $2^\Gamma$ with no invariant Borel probability measure
for every non-amenable group $\Gamma$.
Andy Zucker communicated subsequently to us a different proof
based on ideas of abstract topological dynamics,
especially the concept of strongly proximal flows.

We study in \cref{subshifts-space} a universal space for actions and equivalence relations
and the descriptive or topological properties of various subclasses.

Fix a countable group $\Gamma$.
For any Polish space $X$,
let $F(X^\Gamma)$ be the Effros Borel space
of all closed nonempty subsets of $X^\Gamma$
and define the standard Borel space of subshifts of $X^\Gamma$ as follows:
\[
    \Sh(\Gamma, X)
    = \{F \in F(X^\Gamma) : \textup{$F$ is $\Gamma$-invariant}\}.
\]
If $X$ is compact, we endow $F(X^\Gamma)$ with the compact Vietoris topology
and then $\Sh(\Gamma, X)$ is a compact Polish space in this topology.

Consider the Hilbert cube $\I^\N$.
Every compact Polish space is
(up to homeomorphism)
a closed subspace of $\I^\N$,
and thus every $\mathbf\Gamma$-\textbf{flow}
(i.e., a continuous action of $\Gamma$ on a compact Polish space)
is (topologically) isomorphic to a subshift of $(\I^\N)^\Gamma$,
i.e., there is homeomorphism of the underlying space of the flow
with the space of a subshift preserving the action.
We can thus consider the compact Polish space $\Sh(\Gamma, \I^\N)$
as the universal space of $\Gamma$-flows.

Similarly consider the product space $\R^\N$.
Every Polish space is
(up to homeomorphism)
a closed subspace of $\R^\N$,
and thus every continuous $\Gamma$-action on a Polish space is
(topologically) isomorphic to a subshift of $(\R^\N)^\Gamma$.
We can thus consider the standard Borel space $\Sh(\Gamma, \R^\N)$
as the universal space of continuous $\Gamma$-actions.

In particular taking $\Gamma = \F_\infty$,
the free group with a countably infinite set of generators,
we see that every CBER is Borel isomorphic to the equivalence relation $E_F$
induced on some subshift $F$ of $(\R^\N)^{\F_\infty}$
and so we can view $\Sh(\F_\infty, \R^\N)$ also as the universal space of CBER
and study the complexity of various classes of CBER
(like, e.g., smooth, aperiodic, hyperfinite, etc.)
as subsets of this universal space.
Similarly, we can view $\Sh(\F_\infty, \I^\N)$ as the universal space of
CBER that admit a compact action realization.
In this case we can also consider complexity questions
as well as generic questions of various classes. 

Let $\Phi$ be a property of continuous $\Gamma$-actions on Polish spaces
which is invariant under (topological) isomorphism.
Let 
\[
    \Sh_\Phi(\Gamma, X)
    = \{F \in \Sh(\Gamma, X) : F \models \Phi \},
\]
where we write $F \models \Phi$ to mean that $F$ has the property $\Phi$.

We will consider below the following $\Phi$,
where for the definition of the concepts in 7)--10)
see \ref{subshifts-space-props} of \cref{subshifts-space}.
When a property is stated in terms of an equivalence relation,
it is understood that this refers to the equivalence reaction induced by the action.
For example, for 1)
this is the class of all actions whose induced equivalence relation is finite.
\newcommand{\summarySubshiftProps}{
    \begin{enumerate}[label=\arabic*)]
        \item
            $\fin$: finite equivalence relation
        \item
            $\sm$: smooth equivalence relation
        \item
            $\free$: free action
        \item
            $\aper$: aperiodic equivalence relation
        \item
            $\comp$: compressible equivalence relation
        \item
            $\hyp$: hyperfinite equivalence relation
        \item
            $\amen$: amenable equivalence relation
        \item
            $\measHyp$: measure-hyperfinite equivalence relation
        \item
            $\measAmen$: measure-amenable action
        \item
            $\freeMeasHyp$: free action $+$ measure-hyperfinite equivalence relation
        \item
            $\freeHyp$: free action $+$ hyperfinite equivalence relation
    \end{enumerate}
    
    We summarize in the following table what we know
    concerning the descriptive or generic properties of the $\Phi$ above. Various parts of this table are true for certain classes of infinite groups $\Gamma$ and we explain that in detail in the paragraph following the table. We would like however to point out that they all hold for the free non-abelian groups $\F_n$, $2 \le n \le \infty$.
    
    \begin{center}
        \renewcommand{\arraystretch}{1.1}
        \begin{tabular}{|l|c|c|c|}
            \hline
            \multicolumn{1}{|c|}{$\Phi$}
                & \multicolumn{2}{c|}{$\Sh_\Phi(\Gamma, \I^\N)$}
                & $\Sh_\Phi(\Gamma, \R^\N)$ \\
            \hline \hline
                fin & \multirow{2}{*}{meager}
                & \multicolumn{2}{c|}{\multirow{2}{*}{$\PI^1_1$-complete}} \\ \cline{1-1}
            sm & & \multicolumn{2}{c|}{} \\ \hline
            free & \multirow{9}{*}{comeager}
                & \multirow{2}{*}{$G_\delta$}
                & \multirow{3}{*}{$\PI^1_1$-complete} \\ \cline{1-1}
            aper & & & \\ \cline{1-1}\cline{3-3}
            comp & & open & \\ \cline{1-1}\cline{3-4}
            hyp & & \multicolumn{2}{c|}{\multirow{2}{*}{$\PI^1_1$-hard, $\SIGMA^1_2$}} \\ \cline{1-1}
            amen & & \multicolumn{2}{c|}{} \\ \cline{1-1}\cline{3-4}
            measHyp & & \multicolumn{2}{c|}{$\PI^1_1$-complete} \\ \cline{1-1}\cline{3-4}
            measAmen & & \multirow{2}{*}{$G_\delta$} &
                \multirow{2}{*}{$\PI^1_1$-complete} \\ \cline{1-1}
            freeMeasHyp & & & \\ \cline{1-1}\cline{3-4}
            freeHyp & & $\SIGMA^1_2$ & $\PI^1_1$-hard, $\SIGMA^1_2$ \\ \hline
        \end{tabular}
    \end{center}

    In this table,
    $\Gamma$ is a countably infinite group,
    $\Gamma$ is residually finite in the ``$\boldsymbol{\Pi^1_1}$-complete'' entry of the first two rows,
    $\Gamma$ is non-amenable in the ``comeager'' entry of the fifth row,
    $\Gamma$ is non-amenable in the 
    ``$\boldsymbol{\Pi^1_1}$-hard'' and ``$\boldsymbol{\Pi^1_1}$-complete'' entries of the last six rows for $\R^\N$ and $\Gamma$ is also  
    residually finite in the
    ``$\boldsymbol{\Pi^1_1}$-hard'' and
    ``$\boldsymbol{\Pi^1_1}$-complete'' entries in rows 6--8 for $\I^\N$.
    $\Gamma$ is exact in the ``comeager'' entry of rows 7--10
    and $\Gamma$ is a free group (in fact any group of finite asymptotic dimension, e.g., any hyperbolic group) in the ``comeager'' entry
    of the ``hyp'' and ``freeHyp'' rows, as it was recently proved in \cite{IS25}, answering a question previously left open in this work.
    We show in \Cref{hypDense} that hyperfiniteness is dense in
    $\Sh(\Gamma, \I^\N)$ for any subgroup $\Gamma$ of a hyperbolic group,
    but we do not know if it is generic in general.
    We also do not know the exact descriptive complexity of hyperfiniteness.
}
\summarySubshiftProps

\section{$K_\sigma$ realizations}
Clinton Conley also raised the question of whether
every $E \in \ap$ admits a $K_\sigma$ realization in a Polish space.
We show in \cref{transitiveKSigma}
that one can even obtain a transitive $K_\sigma$ realization on $\cantor$,
where an equivalence relation is \textbf{transitive}
if it has at least one dense class.
This raises the related question of whether every $E\in \ap$
admits a minimal $K_\sigma$ or even $F_\sigma$ realization in a Polish space,
where an equivalence relation is called \textbf{minimal} if all its classes are dense.
We show in \cref{t:fsminimal} that the answer is positive for $F_\sigma$.
In view of \cref{1.3},
every non-smooth hyperfinite equivalence relation in $\ap$
has a minimal $K_\sigma$ realization on a compact Polish space
and Solecki in \cite{Sol02} had shown that this fails for smooth relations,
but this is basically the extent of our knowledge concerning $K_\sigma$.
Call a CBER on a compact Polish space $X$ \textbf{compactly graphable}
if there is a compact graphing of $E$,
i.e., a compact graph (irreflexive, symmetric relation) $K\subseteq E$
so that the $E$-classes are the connected components of $K$.
Clearly every such $E$ is $K_\sigma$.
We also show in \cref{3.9.5} that every hyperfinite and every compressible CBER
in $\ap$ has a \textbf{compactly graphable realization},
i.e., is Borel isomorphic to an equivalence relation
on a (necessarily) compact Polish space that is compactly graphable.
We do not know if this is true for every $E\in \ap$.
Finally in \cref{ksigma-ideal} we study a $\sigma$-ideal associated with a $K_\sigma$ CBER.

\section{The Borel inclusion order}
In connection with these realization problems,
we were also led to consider the following quasi-order on CBER,
which we call the \textbf{Borel inclusion order}.
Given CBER $E,F$ on standard Borel spaces,
we put $E\subseteq_B F$ if there is $E'\cong_B E$ with $E'\subseteq F$.

\textit{
    Below,
    unless otherwise explicitly stated or understood from the context,
    by a \textbf{measure} on a standard Borel space
    we will always mean a Borel probability measure.
}
 
We now have the following result (see \cref{2.3}, \cref{2.8} and \cref{2.12}),
where $\boldsymbol{\aph}$ is the class of hyperfinite relations in $\ap$.

\begin{thm}
    \leavevmode
    \begin{enumerate}[label=(\arabic*)]
        \item
            If $E\subseteq_B F$ are in $\ap$,
            then $|\EINV_E|\ge |\EINV_F|$ and if $E,F\in \aph$ are not smooth,
            then $E \subseteq_B F \iff |\EINV_E|\ge |\EINV_F|$.
        \item
            For any $E \in \ap$, there is $F\in \aph$ with $F \subseteq E$
            such that moreover $\EINV_E = \EINV_F$.
    \end{enumerate}
\end{thm}

Using this and the classification theorem for hyperfinite CBER from \cite[9.1]{DJK94},
one can then prove the next result
(see \cref{2.10} and \cref{2.13a}),
where we use the following terminology and notation:

For each CBER $E$ and standard Borel space $S$,
$SE$ is the direct sum of ``$S$'' copies of $E$
(see \Cref{prelim}).
We let $E_0$ be the equivalence relation on $\cantor$ given by
$x E_0 y \iff \exists m \forall n \ge m (x_n = y_n)$;
$E_t$ is the equivalence relation on $\cantor$ given by
$x E_t y \iff \exists m \exists n \forall k (x_{m+k} = y_{n+k})$;
$I_\N = \N^2$;
$E_\infty$ is a CBER universal under Borel reducibility;
and $E\times F$ is the product of $E$ and $F$.
Finally $\subset_B$ is the strict part of $\subseteq_B$
and for any quasi-order $\preceq$ with strict part $\prec$
on a set $Q$ and $q,r\in Q$,
we say that $r$ is a \textbf{successor} to $q$
if $q\prec r$ and $(q\prec s\preceq r \implies  r\preceq s)$.
Finally, for each cardinal $\kappa \in \{0,1, 2, 3, \dots, \aleph_0, 2^{\aleph_0}\}$,
let $\ap_\kappa$ be the class of all $E \in \ap$ such that $|\EINV_E| = \kappa$.
Thus by Nadkarni's Theorem (\Cref{nadkarni}),
$\ap_0$ is the class of compressible relations.
We also let for $\kappa > 0$,
$\kappa E = SE$,
where $S$ is a standard Borel space of cardinality $\kappa$.

\begin{thm}
    \leavevmode
    \begin{enumerate}[label=(\arabic*)]
        \item
            $\R E_0 \pinc \N E_0 \pinc \cdots \pinc 3 E_0 \pinc 2 E_0 \pinc E_0 \pinc E_t$,
            each equivalence relation in this list
            is a successor in $\inc$ of the one preceding it
            and $\N E_0$ is the infimum in $\inc$
            of the $n E_0, n\in \N\setminus\{0\}.$
        \item
            $\R I_\N  \pinc E_t$ and
            $E_t$ is a successor of $\R I_\N$ in $\inc$. 
        \item
            $\R I_\N$ is $\inc$-minimum in ${\ap}_0$ and
            $E_t$ is $\inc$-minimum among the non-smooth elements of ${\ap}_0$.
        \item
            (B. Miller) $E_\infty\times I_\N$ is $\inc$-maximum in ${\ap}_0$.
        \item
            For each $\kappa >0$,
            $\kappa E_0$ is a $\inc$-minimum element of $\ap_\kappa$
            but $\ap_\kappa$ has no $\inc$-maximum element.
        \item
            Let $\kappa \le \lambda$.
            Then for every $E \in \ap_\lambda$,
            there is $F\in \ap_\kappa$ such that $E\subseteq_B F$.
    \end{enumerate}
\end{thm}
In particular $\R E_0$ is $\subseteq_B$-minimum non-smooth in $\ap$
and $E_\infty \times I_\N$ is $\subseteq_B$-maximum in $\ap$.
Thus one has the following version of the Glimm-Effros Dichotomy
\Cref{theorem-smooth-equivalence}
for $\subseteq_B$ (see \cref{2.13}):

\begin{thm}
    Let $E\in \ap$. Then exactly one of the following holds:
    \begin{enumerate}[label=(\roman*)]
        \item
            $E$ is smooth,
        \item
            $\R E_0 \inc E$.
    \end{enumerate}
\end{thm}

\section{2-adequate groups}
Consider now a Borel action of $\Gamma$ on an uncountable standard Borel space,
which we can assume is equal to $\R$.
Then the map $f \colon X\to \R^\Gamma$ given by $x\mapsto p_x$,
where $p_x(\gamma) = \gamma^{-1}\cdot x$,
is an equivariant Borel embedding of this action to the shift action on $\R^\Gamma$.
Thus every aperiodic CBER $E$ induced by a Borel action of $\Gamma$
can be realized as (i.e., is Borel isomorphic to)
the restriction of $E^{\mathrm{ap}}(\Gamma, \R)$ to an invariant Borel set.
By a result in \cite[5.5]{JKL02},
we also have $E^{\mathrm{ap}}(\Gamma, \R)\cong_B E^{\mathrm{ap}}(\Gamma, \N)$,
so such realizations exist for $E^{\mathrm{ap}}(\Gamma, \N)$ as well.
We consider here the question of whether these realizations can be achieved in the optimal form,
i.e., replacing $E^{\mathrm{ap}}(\Gamma, \N)$ by $E^{\mathrm{ap}}(\Gamma, 2)$.
This is equivalent to the statement that
$E^{\mathrm{ap}}(\Gamma, \R) \cong_B E^{\mathrm{ap}}(\Gamma, 2)$.
If this happens, then we call the group $\Gamma$ \textbf{$2$-adequate}.

Using a recent result of Hochman-Seward,
we show the following
(see \cref{amenable-2-adequate}):

\begin{thm}
    Every infinite countable amenable group is 2-adequate.
\end{thm}

This in particular answers in the negative a question of Thomas \cite[Page 391]{Tho12},
who asked whether there are infinite countable amenable groups $\Gamma$
for which $E(\Gamma, \R)$ is not Borel reducible to $E(\Gamma, 2)$. 

We also show the following (see \cref{4.9} and \cref{4.11}):

\begin{thm}
    \leavevmode
    \begin{enumerate}[label=(\arabic*)]
        \item
            The free product of any countable group with a group
            that has an infinite amenable factor and thus,
            in particular,
            the free groups $\F_n$,
            $1 \le n \le \infty$,
            are $2$-adequate.
        \item
            Let $\Gamma$ be $n$-generated, $1 \le n \le \infty$.
            Then $\Gamma\times \F_n$ is $2$-adequate.
            In particular,
            all products $\F_m\times \F_n$,
            $1 \le m, n \le \infty$,
            are $2$-adequate.
    \end{enumerate}
\end{thm}

On the other hand there are groups which are not 2-adequate (see \cref{4.12}).

\begin{thm}
    The group $\SL_3(\Z)$ is not 2-adequate.
\end{thm}

We do not know if there is a characterization of 2-adequate groups.

\section{Some other classes of groups}
In the course of the previous investigations two other classes of groups have been considered. A countable group $\Gamma$ is called \textbf{hyperfinite generating} if for every $E\in \aph$ there is a Borel action of $\Gamma$ that generates $E$. We provide equivalent formulations of this property in \cref{5.1} and show in \cref{5.2} that all countable groups with an infinite amenable factor are hyperfinite generating. On the other hand, no infinite countable group with property (T) has this property (see \cref{5.3}). 

Finally we say that an infinite countable group $\Gamma$ is \textbf{dynamically compressible} if every $E\in \ap$ generated by a Borel action of $\Gamma$ can be Borel reduced to a compressible $F\in \ap$ induced by a Borel action of $\Gamma$. We show  in \cref{5.7} that every infinite countable amenable group is dynamically compressible and the same is true for any countable group that contains a non-abelian free group (see \cref{5.6}). However there are infinite countable groups that fail to satisfy these two conditions but they are still dynamically compressible (see \cref{5.9}). We do not know if \textit{every} infinite countable group is dynamically compressible.

\section{Organization}
The paper is organized as follows.
In \Cref{inclusion},
we study the structure of the Borel inclusion order
on countable Borel equivalence relations.
In \Cref{realizations},
we consider topological realizations
of countable Borel equivalence relations.
In \Cref{subshifts},
we discuss subshifts in connection to this theory.
\Cref{ksigma} concerns $K_\sigma$ and $F_\sigma$ realizations.
In \Cref{generators},
we introduce and study the concept of 2-adequate groups,
and in \Cref{additional},
we discuss results concerning the concepts of
hyperfinite generating groups and dynamically compressible groups.
In \Cref{problems},
we collect some of the main open problems discussed in this paper.
In \Cref{appendix-amenable-actions},
we discuss various notions of amenability for
actions of countable groups that are relevant to the results in \cref{subshifts-space}.
In \Cref{appendix-weak-containment},
we give the proof of \cref{395} concerning weak containment of actions.
Finally in \cref{appendix-correspondence-theorem},
we give a proof of the Correspondence Theorem of Hochman
(see \cref{correspondence}).

\chapter{Preliminaries}\label{prelim}
\section{Equivalence relations}\label{prelim-eqrel}
For general background on equivalence relations,
see \cite{Kec25}.

Given an equivalence relation $E$ on a set $X$ and an element $x \in X$,
the \textbf{$E$-class of $x$}
is the set $[x]_E := \{x' \in X : x \mathrel E x'\}$.

An equivalence relation $E$ on a set $X$
is \textbf{finite} if $\forall x \in X \; |[x]_E| < \infty$,
it is \textbf{countable} if $\forall x \in X \; |[x]_E| \le \aleph_0$,
and it is \textbf{aperiodic} if $\forall x \in X \; |[x]_E| = \infty$.

A \textbf{Borel equivalence relation}
is an equivalence relation $E$ on a standard Borel space $X$
such that $E$ is Borel as a subset of $X^2$.
A \textbf{countable Borel equivalence relation (CBER)}
is a Borel equivalence relation which is countable.

Let $E$ be an equivalence relation on a set $X$
and let $A$ be a subset of $X$.
The \textbf{$E$-saturation of $A$}
is the set $[A]_E := \bigcup_{a \in A} [a]_E$.
If $E$ is a CBER and $A$ is Borel,
then $[A]_E$ is also Borel.
We say that $A$ is \textbf{$E$-invariant} if $[A]_E = A$,
a \textbf{partial transversal of $E$}
if $\forall x \in X \; |A \cap [x]_E| \le 1$,
and a \textbf{transversal of $E$}
if $\forall x \in X \; |A \cap [x]_E| = 1$.

Let $E$ and $F$ be equivalence relations on sets $X$ and $Y$ respectively.
A function $f \colon X \to Y$ is  a \textbf{homomorphism from $E$ to $F$}
if $\forall x, x' \in X \; [xEx' \implies f(x) F f(x')]$,
a \textbf{cohomomorphism from $E$ to $F$}
if $\forall x, x' \in X \; [f(x) F f(x') \implies xEx']$,
a \textbf{reduction from $E$ to $F$}
if it is both a homomorphism and a cohomomorphism,
an \textbf{embedding from $E$ to $F$}
if it is an injective reduction,
an \textbf{invariant embedding from $E$ to $F$}
if it is an embedding with $F$-invariant image,
and an \textbf{isomorphism from $E$ to $F$}
if it is a bijective reduction from $E$ to $F$.

Given a Borel equivalence relation $F$,
a Borel equivalence relation $E$
is \textbf{Borel reducible to $F$}
(denoted $E \le_B F$)
if there is a Borel reduction from $E$ to $F$,
\textbf{Borel embeddable in $F$}
(denoted $E \sqsubseteq_B F$)
if there is a Borel embedding from $E$ to $F$,
\textbf{invariantly Borel embeddable in $F$}
(denoted $E \sqsubseteq^i_B F$)
if there is a invariant Borel embedding from $E$ to $F$,
and \textbf{Borel isomorphic to $F$}
(denoted $E \cong_B F$)
if there is a Borel isomorphism from $E$ to $F$.
Note that $\le_B$,
$\sqsubseteq_B$,
and $\sqsubseteq^i_B$
are quasi-orders on the class of Borel equivalence relations.
We respectively write $<_B$,
$\sqsubset_B$,
and $\sqsubset^i_B$
for their strict parts
(the strict part of a quasi-order $\preceq$
is the relation $\prec$ where $x \prec y$
iff $x \preceq y \mathrel\& y \not\preceq x$).
We say that $E$ and $F$ are \textbf{Borel bireducible},
denoted $E \sim_B F$,
if $E \le_B F \mathrel\& F \le_B E$;
note that $\sim_B$ is an equivalence relation
since $\le_B$ is a quasi-order.
If $E \sqsubseteq^i_B F$ and $F \sqsubseteq^i_B E$,
then $E \cong_B F$
(see \cite[Page 195]{DJK94}).

Given an equivalence relation $E$ on $X$ and a subset $A$ of $X$,
we denote by $E \uhr A$ the \textbf{restriction of $E$ to $A$},
i.e. the equivalence relation on $A$ defined by $E \uhr A := E \cap A^2$.
Note that the inclusion map $A \hra X$ is an embedding from $E \uhr A$ to $E$.

Let $E$ be an equivalence relation on $X$.
A \textbf{selector} for $E$
is a homomorphism $s$ from $E$ to $\Delta_X$
satisfying $\forall x \in X \; [x E s(x)]$.

Let $E$ and $F$ be equivalence relations on $X$ and $Y$.
The \textbf{direct sum of $E$ and $F$}
is the equivalence relation $E \oplus F$ on $X \sqcup Y$ defined by
$(z, z') \in E \oplus F$ iff
$(z, z' \in X \mathrel\& zEz') \lor (z, z' \in Y \mathrel\& zFz')$.
The \textbf{product of $E$ and $F$}
is the equivalence relation $E \times F$ on $X \times Y$ defined by
$((x, y), (x', y')) \in E \times F$ iff $xEx' \mathrel\& yFy'$.
If $E$ and $F$ are CBERs,
then $E \oplus F$ is a CBER with
$E \sqsubseteq^i_B E \oplus F$ and $F \sqsubseteq^i_B E \oplus F$,
and $E \times F$ is a CBER.

Note that if $X$ is a set with a partition $X = A \sqcup B$,
and there are equivalence relations $E_A$ on $A$ and $E_B$ on $B$ ,
then $E_A \oplus E_B$ is an equivalence relation on $X$.

More generally,
for every $n \in \N$,
let $E_n$ be an equivalence relation on a set $X_n$.
The \textbf{direct sum of $(E_n)_{n \in \N}$},
denoted $\bigoplus_{n \in \N} E_n$,
is the equivalence relation on $\bigsqcup_n X_n$
defined by $z \mathrel{\bigoplus_{n \in \N} E_n} z'$
iff $\exists n \in \N \; [z, z' \in X_n \mathrel\& z E_n z']$.
If every $E_n$ is a CBER,
then so is $\bigoplus_{n \in \N} E_n$,
and for every $n \in \N$,
we have $E_n \sqsubseteq^i_B \bigoplus_{n \in \N} E_n$.

Let $X$ be a set.
The \textbf{discrete equivalence relation on $X$},
denoted $\Delta_X$,
is the equivalence relation defined by
$x \Delta_X x'$ iff $x = x'$.
The \textbf{indiscrete equivalence relation on $X$},
denoted $I_X$,
is the equivalence relation such that for all $x, x' \in X$,
we have $x I_X x'$.
If $X$ is a standard Borel space,
then $\Delta_X$ is a CBER,
and if furthermore $X$ is countable,
then $I_X$ is CBER.

Given an equivalence relation $E$ on a set $X$ and a set $S$,
we define $SE := \Delta_S \times E$
(this is an equivalence relation on $S \times X$).
If $E$ is a CBER and $S$ is a standard Borel space,
then $SE$ is a CBER.

Given equivalence relations $E$ and $F$ on $X$,
the \textbf{join of $E$ and $F$},
denoted $E \vee F$,
is the smallest equivalence relation on $X$ containing $E$ and $F$.
If $E$ and $F$ are CBERs,
then $E \vee F$ is also a CBER.

Let $\Gamma$ be a group,
let $X$ be a set,
and let $\mathbf a \colon \Gamma \times X \to X$ be a group action.
The \textbf{orbit equivalence relation of $\mathbf a$},
or the \textbf{equivalence relation generated by $\mathbf a$},
is the equivalence relation $E_{\mathbf a}$ on $X$ defined by
$x E_{\mathbf a} x'$
iff
$\exists \gamma \in \Gamma \; [\gamma \cdot x = x']$.
We will denote the orbit equivalence relation by $E_\Gamma^X$,
or even by $E_\Gamma$,
when the action is understood.
If $\Gamma$ is countable,
$X$ is a standard Borel space,
and $\mathbf a$ is Borel,
then $E_{\mathbf a}$ is a CBER.
Every CBER is an orbit equivalence relation:
\begin{thm}[{Feldman-Moore \cite[Theorem 1]{FM77}}]\label{theorem-feldman-moore}
    Every {\CBER} is generated by a Borel action of a countable group.
\end{thm}

The \textbf{eventual equality relation}
is the CBER $E_0$ on $\cantor$ defined by
$x E_0 y$
iff
$\exists k \, \forall n \; [x_{k + n} = y_{k + n}]$.
Note that $E_0$ is the orbit equivalence relation
of the continuous action $(\Z/2)^{\oplus \N} \car 2^\N$.

The \textbf{tail equivalence relation}
is the CBER $E_t$ on $2^\N$ defined by
$x E_t y$
iff
$\exists k, l \,\forall n \; [x_{k + n} = y_{l + n}]$.

\section{Measures}\label{prelim-measure}
Given a standard Borel space $X$,
we denote by $\Prob(X)$
the standard Borel space of Borel probability measures on $X$.

Let $X$ be a standard Borel space,
let $f \colon X \to X$ be a Borel function,
and let $\mu$ be a Borel probability measure on $X$.
We say that $f$ \textbf{preserves $\mu$} or is \textbf{$\mu$-preserving}
(or just \textbf{measure-preserving} if $\mu$ is understood),
or that $\mu$ is \textbf{$f$-invariant}
(or just \textbf{invariant} if $f$ is understood),
if $f_* \mu = \mu$.

Let $E$ be a CBER on $X$,
and let $\mu$ be a Borel probability measure on $X$.
We say that $E$ \textbf{preserves $\mu$} or is \textbf{$\mu$-preserving}
(or \textbf{measure-preserving} if $\mu$ is understood),
or that that $\mu$ is \textbf{$E$-invariant}
(or just \textbf{invariant} if $E$ is understood),
if any of the following equivalent conditions hold
(see \cite[Proposition 2.1]{KM04} for a proof of equivalence):
\begin{enumerate}[label=(\roman*)]
    \item
        For all Borel sets $A, B \subseteq X$ and all Borel bijections $f \colon A \to B$
        whose graph is contained in $E$,
        we have $\mu(A) = \mu(B)$.
    \item
        There is a countable group $\Gamma$
        and a Borel action $\Gamma \car X$ generating $E$
        such that every $\gamma \in \Gamma$ is $\mu$-preserving.
    \item
        For every countable group $\Gamma$
        and every Borel action $\Gamma \car X$ generating $E$,
        every $\gamma \in \Gamma$ is $\mu$-preserving.
\end{enumerate}

A Borel probability measure $\mu$ on $X$ is \textbf{$E$-ergodic}
if every $E$-invariant Borel subset $A$ of $X$
is $\mu$-null or $\mu$-conull.
We denote by $\INV_E$ the Borel subset of $\Prob(X)$
consisting of $E$-invariant measures,
and we denote by $\EINV_E$ the Borel subset of $\INV_E$
consisting of $E$-ergodic measures,
which is also the set of extreme points of $\INV_E$.
Given $\kappa \in \N \cup \{\aleph_0, 2^{\aleph_0}\}$,
we say that $E$ is \textbf{$\kappa$-ergodic} if $|\EINV_E| = \kappa$.
We say that $E$ is \textbf{uniquely ergodic} if it is $1$-ergodic.

A CBER $E$ on $X$ is \textbf{compressible}
iff there is a Borel injection $f \colon X \to X$
such that for every $E$-class $C$,
we have $f(C) \subsetneq C$.
Nadkarni's theorem says that this characterizes $0$-ergodicity:
\begin{thm}[{Nadkarni \cite{Nad90}}]\label{nadkarni}
    A {\CBER} $E$ is compressible iff $\EINV_E$ is empty.
\end{thm}

Given a non-compressible CBER $E$ on $X$,
an \textbf{ergodic decomposition} of $E$
is a Borel homomorphism $x \mapsto e_x$
from $E$ to $\Delta_{\EINV_E}$
such that for every $e \in \EINV_E$,
the set $X_e := \{x \in X : e_x = e\}$ is $e$-conull.
An ergodic decomposition is often specified
as a partition $(X_e)_{e \in \EINV_E}$ of $X$.
Ergodic decompositions exist,
and are ``unique mod compressible'':
\begin{thm}[{Farrell \cite{Far62}, Varadarajan \cite{Var63}}]\label{ergodic-decomposition}
    Every non-compressible {\CBER} $E$ has an ergodic decomposition.
    Moreover,
    if $(X_e)_{e \in \EINV_E}$ and $(Y_e)_{e \in \EINV_E}$
    are ergodic decompositions of $E$,
    then $E \uhr (X \setminus\bigsqcup_{e \in \EINV_E} X_e \cap Y_e)$
    is compressible.
\end{thm}

Let $X$ be a standard Borel space equipped with a probability measure $\mu$.
We denote by $\Aut(X, \mu)$ the group of $\mu$-preserving Borel bijections $X \to X$,
where two such are identified if they agree on a $\mu$-conull set.
Given a $\mu$-preserving CBER $E$ on $X$,
the \textbf{measure-theoretic full group of $E$ with respect to $\mu$},
denoted $[E]$,
is the subgroup of $T \in \Aut(X, \mu)$ for which
$\{x \in X : T(x) E x\}$ is $\mu$-conull.

\section{Classes of CBERs}\label{prelim-classes}
A Borel equivalence relation is \textbf{smooth}
if it is Borel reducible to $\Delta_\R$.
We have the following equivalences:
\begin{thm}\label{theorem-smooth-equivalence}
    For a {\CBER} $E$,
    the following are equivalent:
    \begin{enumerate}[label=(\roman*)]
        \item \label{theorem-smooth-equivalence-smooth}
            $E$ is smooth.
        \item \label{theorem-smooth-equivalence-transversal}
            $E$ has a Borel transversal.
        \item \label{theorem-smooth-equivalence-selector}
            $E$ has a Borel selector.
        \item \label{theorem-smooth-equivalence-E0}
             $E_0 \not \sqsubseteq_B E$.
        \item \label{theorem-smooth-equivalence-Et}
            $E_t \not \sqsubseteq^i_B E$.
        \item \label{theorem-smooth-equivalence-stable}
            $E \not\cong_B E \oplus E_t$.
    \end{enumerate}
\end{thm}
The equivalence between
\ref{theorem-smooth-equivalence-smooth} and \ref{theorem-smooth-equivalence-E0}
is known as the ``Glimm-Effros dichotomy'',
and holds more generally for Borel equivalence relations,
see \cite[Theorem 1.1]{HKL90}.
To see that \ref{theorem-smooth-equivalence-Et}
implies \ref{theorem-smooth-equivalence-E0},
note that if there is a Borel injective reduction from $E_0$ to $E$,
then $E_0 \sqsubseteq_B E$,
so since $E_t \sqsubseteq_B E_0$,
we have $E_t \sqsubseteq_B E$,
and hence $E_t \sqsubseteq^i_B E$
by compressibility of $E_t$
(see \cite[Proposition 2.3]{DJK94}).
To see that \ref{theorem-smooth-equivalence-stable}
implies \ref{theorem-smooth-equivalence-Et},
note that if $E_t \sqsubseteq^i_B E$,
then $E \cong_B F \oplus E_t$ for some $F$,
and hence
$E \oplus E_t \cong_B F \oplus E_t \oplus E_t \cong_B F \oplus E_t \cong_B E$,
since $E_t \oplus E_t \cong_B E_t$.

Let $E$ be a CBER on an uncountable standard Borel space.
By Mycielski's theorem \cite{Myc73},
there is a Borel embedding from $\Delta_{2^\omega}$ to $E$.
By taking the saturation of the image,
we can write $E = E_{\text{smooth}} \oplus F$,
where $E_{\text{smooth}}$ is a smooth CBER with continuum many classes.
In particular,
if $E$ is aperiodic,
then $E_{\text{smooth}} \cong_B \R I_\N$,
so $E \cong_B \R I_\N \oplus F$.
But since $\R I_\N \oplus \R I_\N \cong_B \R I_\N$,
we have $E \cong_B E \oplus \R I_\N$.
More generally,
this means that if $E$ is an aperiodic CBER on an uncountable space,
and $F$ is a smooth aperiodic CBER,
then $E \cong E \oplus F$.
In particular,
deleting or adding countably many infinite classes
from an aperiodic CBER on an uncountable space
does not change its isomorphism type.

A CBER $E$ on $X$ is \textbf{hyperfinite}
if there is an increasing sequence $(F_n)_n$ of finite CBERs on $X$
such that $E = \bigcup_{n \in \N} F_n$
(increasing meaning that $\forall n \in \N \; F_n \subseteq F_{n+1}$).

The hyperfinite CBERs are completely classified up to Borel isomorphism:
\begin{thm}[{Dougherty-Jackson-Kechris \cite[Corollary 9.3]{DJK94}}]\label{thm-hf-classif}
    Up to Borel isomorphism,
    every aperiodic hyperfinite {\CBER}
    is Borel isomorphic to exactly one of
    $\R I_\N$,
    $E_t$,
    $E_0$,
    $2E_0$,
    $3E_0$,
    $\ldots$,
    $\N E_0$,
    $\R E_0$.
    Furthermore,
    we have
    \[
        \R I_\N \sqsubset^i_B 
        E_t \sqsubset^i_B
        E_0 \sqsubset^i_B
        2E_0 \sqsubset^i_B
        3E_0 \sqsubset^i_B
        \cdots \sqsubset^i_B
        \N E_0 \sqsubset^i_B
        \R E_0
    \]
\end{thm}

A CBER $E$ is \textbf{amenable}
if there is a sequence of Borel maps
$p_n \colon E \to [0, 1]$ such that
\begin{enumerate}[label=(\roman*)]
    \item
        for every $x \in X$,
        we have $p_n^x \in \Prob([x]_E)$;
    \item
        for every $(x, y) \in E$,
        we have $\|p_n^x - p_n^y\|_1 \to 0$ in $\ell^1([x]_E)$.
\end{enumerate}
Every hyperfinite CBER is amenable,
and every orbit equivalence relation
of a Borel action of a countable amenable group is amenable.

Let $\Phi$ be a property of CBERs
(such as ``smooth'', ``hyperfinite'', ``amenable'', etc.),
and let $E$ be a CBER on $X$.
Given a Borel probability measure $\mu$ on $X$,
we say that $E$ is \textbf{$\mu$-$\Phi$}
if there is an $E$-invariant $\mu$-conull Borel subset
$Y \subseteq X$ such that $E\uhr Y$ is $\Phi$.
We say that $E$ is \textbf{measure-$\Phi$}
if for every Borel probability measure $\mu$ on $X$,
we have that $E$ is $\mu$-$\Phi$.

By the Connes-Feldman-Weiss theorem \cite{CFW81},
for every Borel probability measure $\mu$,
a CBER is $\mu$-hyperfinite iff it is $\mu$-amenable;
in particular,
a CBER is measure-hyperfinite iff it is measure-amenable.
Thus for instance,
every orbit equivalence relation of a Borel action
of a countable amenable group is measure-hyperfinite.

We will use Segal's effective witness to measure-hyperfiniteness
(see \cite[Theorem 1.7.8]{CM17} for a proof):
\begin{thm}[Segal]\label{segal}
    Let $E$ be a {\CBER} on $X$.
    Then there is a hyperfinite Borel subequivalence relation
    $F$ of $\Delta_{\Prob(X)} \times E$
    such that for every $\mu \in \Prob(X)$
    and every Borel $A \subseteq X$ with $E\uhr A$ hyperfinite,
    the set
    $A \setminus \{x \in X : [(\mu, x)]_F = \{\mu\} \times [x]_E\}$
    is $\mu$-null.
\end{thm}

We will need the following consequence:
\begin{cor}\label{corollary-hf-compressible}
    Let $E$ be a measure-hyperfinite {\CBER}.
    Then $E \cong_B E_{\textup{hf}} \oplus E_{\textup{comp}}$,
    for some hyperfinite $E_{\textup{hf}}$
    and compressible $E_{\textup{comp}}$.
\end{cor}
\begin{proof}
    If $E$ is compressible,
    then we are done by taking $E_{\textup{comp}} = E$,
    so assume that $E$ is not compressible.
    
    By Segal's effective witness to measure-hyperfiniteness (\Cref{segal}),
    there is a hyperfinite Borel subequivalence relation
    $F$ of $\Delta_{\Prob(X)} \times E$
    such that for every $\mu \in \Prob(X)$
    and every Borel $A \subseteq X$ with $E\uhr A$ hyperfinite,
    the set
    $A \setminus \{x \in X : [(\mu, x)]_F = \{\mu\}\times [x]_E\}$
    is $\mu$-null.
    Then for every $\mu \in \Prob(X)$,
    since the restriction of $E$
    to some $\mu$-conull $E$-invariant Borel set is hyperfinite,
    we have that
    $\{x \in X : [(\mu, x)]_F = \{\mu\} \times [x]_E\}$
    is $\mu$-conull.

    Fix an ergodic decomposition $(X_e)_{e \in \EINV_E}$ of $E$
    (see \Cref{ergodic-decomposition}).
    Consider the $E$-invariant Borel set
    $X_{\textup{hf}} = \{x \in X : [(e_x, x)]_F = \{e_x\} \times [x]_E\}$.
    On one hand,
    the restriction of $E$ to $X_{\textup{hf}}$ is hyperfinite,
    since the map $x \mapsto (e_x, x)$ is an invariant Borel embedding into $F$.
    On the other hand,
    since $X_{\textup{hf}}$ is $e$-conull for every $e \in \EINV_E$,
    we have by Nadkarni's theorem (\Cref{nadkarni})
    that the restriction of $E$ to $X \setminus X_{\textup{hf}}$
    is compressible.
\end{proof}

A CBER $E$ on $X$ is \textbf{treeable}
if there is an acyclic Borel graph on $X$
whose connected components are the $E$-classes.

A CBER $E$ is \textbf{universal} if every CBER $F$ satisfies $F \le_B E$,
or equivalently by \cite[Theorem 3.6]{MSS16},
if every CBER $F$ satisfies $F \sqsubseteq_B E$.

\section{Group actions}\label{prelim-action}
Fix a countable group $\Gamma$.

A \textbf{$\Gamma$-set} is a set $X$
equipped with an action $\Gamma \car X$.
A \textbf{Borel $\Gamma$-space} is a standard Borel space $X$
equipped with a Borel action $\Gamma \car X$.
A \textbf{Polish $\Gamma$-space} is a Polish space $X$
equipped with a continuous action $\Gamma \car X$.
A Polish $\Gamma$-space is \textbf{(topologically) transitive}
if it is non-empty and there is a dense orbit,
and a Polish $\Gamma$-space is \textbf{minimal}
if it is non-empty and every orbit is dense.
A \textbf{$\Gamma$-flow} is a compact Polish $\Gamma$-space,
and a \textbf{subflow} of a $\Gamma$-flow
is a closed $\Gamma$-invariant subset.

Let $A$ be a set,
and let $L$ be a countable $\Gamma$-set.
The \textbf{shift action} on $A^L$ is
the $\Gamma$-action defined by $(\gamma\cdot y)_l = y_{\gamma^{-1} l}$.
We view $A^L$ as a $\Gamma$-set equipped with the shift action.
In particular,
if $L = \Gamma$,
then this defines the shift action on $A^\Gamma$.
If $A$ is a Polish space,
then we view $A^L$ as a Polish $\Gamma$-space with the product topology,
and a \textbf{subshift} of $A^L$ is a closed $\Gamma$-invariant subset of $A^L$.
If $A$ is a standard Borel space,
then we view $A^L$ as a Borel $\Gamma$-space with the product Borel structure.
If $A$ is a countable set,
then we view $A$ as a discrete Polish space,
so that $A^L$ is a Polish $\Gamma$-space.

\chapter[The Borel Inclusion Order of CBER]{The Borel Inclusion Order of Countable Borel Equivalence Relations}\label{inclusion}

\section{General properties}
\begin{defi}
    Let $E,F$ be CBER on standard Borel spaces $X,Y$, resp.
    We put $E\subseteq_B F$ if there is a Borel isomorphism
    $f \colon X\to Y$ with $f(E)\subseteq F$.
\end{defi}

It is clear that $\subseteq_B$ is a quasi-order on CBER,
which we call the \textbf{Borel inclusion order}.
We also let $E\subset_B F \iff E\subseteq_B F  \ \& \ F\nsubseteq_B E$
be the strict part of this order. 

We will study in this section the structure of this inclusion order
on aperiodic CBER in uncountable standard Borel spaces.

\begin{prop}\label{prop-smbasic}
    \leavevmode
    \begin{enumerate}[label=(\arabic*)]
        \item
            If $E\subseteq_B F$ and $F$ is smooth, then $E$ is smooth.
        \item \label{item-smbasic-comp}
            $E$ is compressible iff $\R I_\N\subseteq_B E$.
            Therefore if $E\subseteq_B F$ and $E$ is compressible,
            then $F$ is compressible.
    \end{enumerate}
\end{prop}

\begin{proof}
    \leavevmode
    \begin{enumerate}[label=(\arabic*)]
        \item
            By the Feldman-Moore theorem (\Cref{theorem-feldman-moore}),
            there is a countable group $\Gamma = (\gamma_n)_{n \in \N}$
            and a Borel action $\Gamma \car X$ with orbit equivalence relation $F$.
            Let $f$ be a Borel selector for $F$
            and for each $x \in X$,
            let $n(x)$ be the least $n$ with $\gamma_n \cdot f(x) \mathrel E x$.
            Then $g(x) = \gamma_{n(x)} \cdot f(x)$ is a Borel selector for $E$.
        \item
            This follows from \cite[Proposition 2.5]{DJK94};
            see also \cite[Proposition 2.23]{Kec25}.
    \end{enumerate}
\end{proof}

We note here the following basic fact:
\begin{prop}\label{2.3}
    If $E \subseteq_B F$,
    then $|\EINV_E| \ge |\EINV_F|$.
\end{prop}
\begin{proof}
    Assume without loss of generality that $E \subseteq F$
    are CBERs on the same standard Borel space $X$.
    This is clear when $|\EINV_F| = 0$,
    so assume that $|\EINV_F| \ge 1$.
    Fix an ergodic decomposition $(X_e)_{e \in \EINV_F}$ of $F$
    (see \Cref{ergodic-decomposition}).
    Then for each $e \in \EINV_F$,
    $X_e$ is $E$-invariant and $e$ is an invariant measure for $E \uhr X_e$,
    thus $X_e$ supports at least one ergodic,
    invariant measure for $E$, say $e'$.
    Since the map $e\mapsto e'$ is injective the proof is complete.
\end{proof}

We will next show that many subclasses of $\ap$,
including $\ap$ itself,
admit maximum under $\subseteq_B$ elements.
This was proved for $\ap$ by Ben Miller,
see \cite[Proposition 11.34]{Kec25},
and the proof below is an adaptation of his argument to a more general context.
Later we will show the existence of a minimum under $\inc$ non-smooth element of $\ap$
(see the paragraph following \cref{2.13}).

\begin{thm}\label{2.4}
    Let $\e\subseteq \ap$ be a class of {\CBER} such that
    $\e$ contains a maximum under $\sqsubseteq_B$ element $E$
    such that $E \times I_\N\in \e$.
    Then $E \times I_\N\in \e$ is $\subseteq_B$-maximum for $\e$.
\end{thm}
\begin{proof}
    \begin{lem}\label{2.5}
        Let $R$ be compressible.
        Then for any $S\in \ap$, $S\subseteq_B R\oplus S$.
    \end{lem}
    \begin{proof}
        We have $S \cong_B \R I_\N \oplus S$
        (see \Cref{prelim}).
        Then by \cref{prop-smbasic},
        we have $S \cong_B  \R I_\N \oplus S \subseteq_B R \oplus S$.
    \end{proof}
    
    We next show that for every $F\in \e$,  $F\subseteq_B E\times I_\N$. Since $F\sqsubseteq_B E$, there is $G$ such that $F\oplus G \subseteq_B E$. Recalling (see, e.g., \cite[Theorem 3.23]{Kec25}) that for any CBER $R$, $R\times I_\N$ is compressible, we have, by \cref{2.5}, that $F\subseteq_B F\oplus(F\times I_\N) \oplus (G\times I_\N)$. Note now that $F\oplus (F\times I_\N) \subseteq_B F\times I_\N$, therefore $F\subseteq_B F\oplus(F\times I_\N) \oplus (G\times I_\N)\subseteq_B (F\times I_\N)\oplus (G\times I_\N)\cong_B (F\oplus G)\times I_\N\subseteq_B E\times I_\N$.
\end{proof}
In particular this applies to the following classes $\e$:
hyperfinite,
$\alpha$-amenable (see \cite[Section 9.2]{Kec25}),
treeable (see \cite[Section 10.2]{Kec25}),
$\ap$.

\section{Hyperfiniteness}
We will discuss here the inclusion order on the hyperfinite equivalence relations.
Recall first the following well-known fact (see, e.g., \cite[Theorem 8.23]{Kec25}):

\begin{prop}
    If $E$ is hyperfinite and $F\subseteq_B E$,
    then $F$ is hyperfinite.
\end{prop}

Thus the class $\boldsymbol{\aph}$ of hyperfinite aperiodic CBER
forms an initial segment in $\inc$.
It is also downwards cofinal in $\inc$ in view of the following standard result
(see, e.g., \cite[Theorem 8.16]{Kec25}):

\begin{thm}\label{2.7}
    For any $E\in \ap$, there is $F\in \aph$ with $F\subseteq E$.
\end{thm}
We will actually need a more precise version of this result,
see again \cite[Theorem 8.16]{Kec25}.
Since a proof of this result has not appeared in print before,
we will include it below. 

\begin{thm}\label{2.8}
    For any $E \in \ap$,
    there is $F \in \aph$ with $F \subseteq E$
    such that moreover $\EINV_E = \EINV_F$.
\end{thm}

\begin{proof}
    If $E$ is compressible,
    then the result follows from \cref{prop-smbasic}\ref{item-smbasic-comp},
    so assume that $E$ is not compressible.

    Let $\mu$ be the unique $E_0$-invariant measure on $2^\N$.
    By the proof of Dye's theorem
    (see e.g., \cite[Theorem 7.13]{KM04} and \cite[5.26]{Kec94}),
    there is a bijective Borel homomorphism $f$
    from $\Delta_{\EINV_E} \times E_0$ to $E$
    such that for every $e \in \EINV_E$,
    the pushforward of $\mu$ along the map
    $2^\N \to X$ defined by $x \mapsto f(e, x)$
    is $e$.
    Then we are done by setting
    $F := f(\Delta_{\EINV_E} \times E_0)$.
\end{proof}

Below for a quasi-order $\preceq$ with strict part $\prec$
on a set $Q$ and $q, r\in Q$,
we say that $r$ is a successor to $q$
if $q \prec r$ and $(q \prec s \preceq r \implies r \preceq s)$. 

We now have:
\begin{thm}\label{2.10}
    \leavevmode
    \begin{enumerate}[label=(\arabic*)]
        \item \label{thm-incl-hf-nonsm}
            $\R E_0 \pinc \N E_0 \pinc \cdots \pinc 3 E_0 \pinc 2 E_0 \pinc E_0 \pinc E_t$,
            each equivalence relation in this list is a successor in $\inc$ of the one preceding it
            and $\N E_0$ is the infimum in $\inc$ of the $n E_0, n\in \N\setminus\{0\}$.
        \item
            $\R I_\N \pinc E_t$ and $E_t$ is a successor of $\R I_\N$ in $\inc$. 
    \end{enumerate}
\end{thm}

\begin{proof}
    \leavevmode
    \begin{enumerate}[label=(\arabic*)]
        \item
            Clearly $E_0\subseteq E_t$ and thus $E_0 \pinc E_t$ as $E_0$ is not compressible.
            To see that $2E_0 \inc E_0$, note that $\cantor = X_0\sqcup X_1$, where $X_i = \{x\in \cantor : x_0 = i\}$, and $E_0 \uhr X_i \cong_B E_0$. From this it follows immediately that $(n+1) E_0 \inc n E_0$, for each $n\in \N, n \ge 1$. 
            
            To show that $\N E_0 \inc n E_0$, for each $n\in \N\setminus \{0\}$,
            it is enough to show that $\N E_0 \inc E_0$,
            since then we have $\N E_0 \cong_B n\N E_0 \subseteq_B nE_0$.
            Let $s_n = 1^n 0$ be the finite sequence staring with $n$ 1's followed by one 0, for $n\in \N$. Let $X_n$ be the subset of $\cantor$ consisting of all sequences starting with $s_n$, let $\bar{1}$ be the constant 1 sequence and put $X= \cantor \setminus \{\bar{1}\}$.
            Then $X =\bigsqcup_n X_n$ and
            $E_0\cong_B E_0 \uhr X \cong_B E_0 \uhr X_n$,
            for each $n\in \N$, which completes the proof that $\N E_0 \inc  E_0$. 
            
            Finally to show that $\R E_0 \inc \N E_0$,
            it is enough to show that $\R E_0 \inc  E_0$
            since then we have $\R E_0 \cong_B (\N\times\R) E_0 \subseteq_B \N E_0$.
            To prove this, let for each $y\in \cantor$, $X_y = \{x \in \cantor : \forall n\in \N (x_{2n} = y_n)\}$. Then $\cantor = \bigsqcup_y X_y$ and $E_0 \uhr X_y \cong_B E_0, \forall y\in \cantor$, which immediately implies that $\R E_0 \inc  E_0$.
            
            This establishes the non-strict orders in the list of \ref{thm-incl-hf-nonsm}.
            The strict orders and the last two statements of \ref{thm-incl-hf-nonsm} 
            now follow from \cref{2.3}.
        \item
            Since $E_t$ is compressible and not smooth, by \cref{prop-smbasic}, $\R I_\N  \pinc E_t$. It is also clear that $E_t$ is a successor of $\R I_\N$.
    \end{enumerate}
\end{proof}

The following is an immediate corollary of \cref{2.10} and the observations preceding this theorem:

\begin{cor}\label{2.12}
    Let $E,F \in \aph$ be non-smooth.
    Then 
    \[
        E \inc F \iff |\EINV_E| \ge |\EINV_F|.
    \]
\end{cor}

\section{A global decomposition}

For each cardinal $\kappa \in \{0, 1, 2, 3, \dots, \aleph_0, 2^{\aleph_0}\}$,
let $\ap_\kappa$ be the class of all $E\in \ap$ which are $\kappa$-ergodic.
Clearly $\ap = \bigsqcup_\kappa \ap_\kappa$ and each $\ap_\kappa$ is invariant under the equivalence relation associated with the quasi-order $\inc$, by \cref{2.3}.
We also let for $\kappa >0$,
$\kappa E = S E$, where $S$ is a standard Borel space of cardinality $\kappa$.

\begin{prop}\label{2.13a}
    \leavevmode
    \begin{enumerate}[label=(\arabic*)]
        \item \label{item-incl-min}
            $\R I_\N$ is $\inc$-minimum in $\ap_0$ and
            $E_t$ is $\inc$-minimum among the non-smooth elements of $\ap_0$.
        \item \label{item-incl-max}
            (B. Miller)
            $E_\infty\times I_\N$ is $\inc$-maximum in $\ap_0$.
        \item
            For each $\kappa >0$,
            $\kappa E_0$ is a $\inc$-minimum element of $\ap_\kappa$
            but $\ap_\kappa$ has no $\inc$-maximum element.
        \item
            Let $\kappa \le \lambda$.
            Then for every $E \in \ap_\lambda$,
            there is $F\in \ap_\kappa$ such that $E \inc F$.
        \item
            (with R. Chen)
            The map $E \mapsto E \oplus E_0$ is an order embedding
            of the non-smooth elements of $\ap$ into $\ap$,
            i.e., for non-smooth $E, F \in \ap$,
            $E \subseteq_B F \iff E \oplus E_0 \subseteq_B F \oplus E_0$.
            It maps $\ap_\kappa$ into $\ap_{\kappa + 1}$ if $\kappa$ is finite,
            and $\ap_\kappa$ into itself if $\kappa$ is infinite.
    \end{enumerate}
\end{prop}
The following picture illustrates parts
\ref{item-incl-min} and \ref{item-incl-max}
of \cref{2.13a}.
\[
  \resizebox{3.50in}{3.50in}{
    \begin{tikzpicture}
      \node [circle, fill, scale=0.6, label=below:$E_t$] (Et) {};
      \node [circle, fill, scale=0.6, label=below:$\R I_\N$] at ($(Et) + (2.5,-1.5)$) (smooth) {};
      \node [circle, fill, scale=0.6, label=below:$E_0$] at ($(Et) + (-2.5,-1.5)$) (E0) {};
      \node [circle, fill, scale=0.6, label=below:$2E_0$] at ($(E0) + (-2.5,-1.5)$) (2E0) {};
      \node [circle, fill, scale=0.6, label=below:$\N E_0$] at ($(2E0) + (-3,-2)$) (NE0) {};
      \node [circle, fill, scale=0.6, label=below:$\R E_0$] at ($(NE0) + (-2.5,-1.5)$) (RE0) {};

      \draw (Et) to (E0);
      \draw (Et) to (smooth);
      \draw (E0) to (2E0);
      \draw [dashed] (2E0) to (NE0);
      \draw (NE0) to (RE0);

      \node [circle, fill, scale=0.6, label=above:$E_\infty\times I_\N$, above = 7 of Et] (aboveEt) {};
      \node at ($(E0) + (0.5,6)$) (aboverightE0) {};
      \node at ($(E0) + (-0.5,6)$) (aboveleftE0) {};
      \node at ($(2E0) + (0.5,6)$) (aboveright2E0) {};
      \node at ($(2E0) + (-0.5,6)$) (aboveleft2E0) {};
      \node at ($(NE0) + (0.5,6)$) (aboverightNE0) {};
      \node at ($(NE0) + (-0.5,6)$) (aboveleftNE0) {};
      \node at ($(RE0) + (0.5,6)$) (aboverightRE0) {};
      \node at ($(RE0) + (-0.5,6)$) (aboveleftRE0) {};

      \node [above = 3.5 of Et] (EINV0) {$|\EINV| = 0$};
      \node [above = 3 of E0] (EINV1) {$|\EINV| = 1$};
      \node [above = 3 of 2E0] (EINV2) {$|\EINV| = 2$};
      \node [above = 3 of NE0] (EINVN) {$|\EINV| = \aleph_0$};
      \node [above = 3 of RE0] (EINVR) {$|\EINV| = 2^\aleph_0$};

      \draw (Et) to [bend right] (aboveEt);
      \draw (Et) to [bend left] (aboveEt);
      \draw (E0) to [bend right] (aboverightE0);
      \draw (E0) to [bend left] (aboveleftE0);
      \draw (2E0) to [bend right] (aboveright2E0);
      \draw (2E0) to [bend left] (aboveleft2E0);
      \draw (NE0) to [bend right] (aboverightNE0);
      \draw (NE0) to [bend left] (aboveleftNE0);
      \draw (RE0) to [bend right] (aboverightRE0);
      \draw (RE0) to [bend left] (aboveleftRE0);
    \end{tikzpicture}
  }
\]
\begin{proof}
    \leavevmode
    \begin{enumerate}[label=(\arabic*)]
        \item
            That $\R I_\N$ is $\subseteq_B$-minimum in ${\ap}_0$
            follows from \cref{prop-smbasic}.

            If $E \in \ap_0$,
            then $E$ is compressible,
            so we have $E_t \subseteq_B E \oplus E_t$ by \Cref{2.5}.
            But if $E$ is non-smooth,
            then $E \oplus E_t \cong_B E$
            (see \Cref{theorem-smooth-equivalence}),
            so we are done.
        \item
            That $E_\infty\times I_\N$ is $\inc$-maximum
            in ${\ap}_0$ follows from \cref{2.4}.
        \item
            The fact that $\kappa E_0$ is a $\inc$-minimum element of $\ap_\kappa$ is clear from \cref{2.8}. That $\ap_\kappa$ has no $\inc$-maximum element can be seen as follows.
            
            Assume that $E$ is such a $\inc$-maximum,
            towards a contradiction.
            Say $E$ lives on the space $X$.
            Fix an invariant measure $\mu$ for $E$.
            We will show that every infinite countable group $\Gamma$
            embeds algebraically into $[E]$,
            the measure-theoretic full group of $E$ with respect to $\mu$
            (see \Cref{prelim-measure}).
            contradicting a result of Ozawa,
            see \cite[page 29]{Kec10}.  
            
            The group $\Gamma$ admits a free Borel action on a standard Borel space $Y$,
            with associated equivalence relation $G$ that has exactly $\kappa$ ergodic,
            invariant measures.
            To see this, consider the free part of the shift action of $\Gamma$ on $2^\Gamma$, which has $2^{\aleph_0}$ ergodic components, and restrict the action to $\kappa$ many ergodic components.
            Since $E$ is $\inc$-maximum in $\ap_\kappa$, let $f \colon Y \to X$ be a Borel isomorphism such that $f(G) = F\subseteq E$.
            Then $\Gamma$ acts freely in a Borel way on $X$ inducing $F$,
            so that $\Gamma$ can be algebraically embedded in $[F]$,
            the measure-theoretic full group of $F$ with respect to $\mu$
            (which is clearly invariant for $F$).
            But $[F] \le [E]$, so $\Gamma$ embeds algebraically into $[E]$.
        \item
            For $\kappa = 0$,
            we are done by \ref{item-incl-max},
            so assume $\kappa > 0$.
            Let $E\in \ap_\lambda$.
            Fix an $E$-invariant Borel subset $Y$ of $X$ such that
            $E \uhr Y$ is $\kappa$-ergodic:
            this exists since if $(X_e)_{e \in \EINV_E}$ is an ergodic decomposition of $E$
            (see \Cref{ergodic-decomposition}),
            then we can find a Borel subset $A \subseteq \EINV_E$ of size $\kappa$,
            then set $Y = \bigsqcup_{e \in A} Y_e$.
            Let $G$ be a compressible CBER on $X \setminus Y$ containing $E \uhr (X \setminus Y)$,
            and let $F = (E \uhr Y) \oplus G$.
            Then $E \subseteq F$ and $F\in \ap_\kappa$.
        \item
            We show that $E\mapsto E\oplus E_0$ is an order embedding
            on non-smooth aperiodic CBERs (on uncountable standard Borel spaces).
            (Note that the only failure is that $E_t\oplus E_0\cong_B \R I_\N\oplus E_0$.)
            
            Clearly,
            if $E\inc F$,
            then $E\oplus E_0\inc F\oplus E_0$.
            Conversely,
            suppose that $E\oplus E_0\inc F\oplus E_0$.
            We want to show that $E\inc F$.
            
            We can write $E \cong_B R\oplus R'$ and $E_0 \cong_B S\oplus S'$
            with $R\oplus S\inc F$ and $R'\oplus S'\inc E_0$.
            Note that $R',S,S'$ are all aperiodic hyperfinite
            (although they may live in a countable space,
            possibly empty),
            and since $E_0 \cong_B S\oplus S'$,
            exactly one of $S$ or $S'$ must be Borel isomorphic to $E_0$,
            and the other is compressible hyperfinite. Also since $E$ is non-smooth,
            we have $E\cong_B E\oplus E_t$,
            and similarly for $F$.
            
            We have two cases:
            \begin{enumerate}[label=(\alph*)]
                \item
                    If $S \cong_B E_0$ and $S'$ is compressible,
                    then since $R'\oplus S'\inc E_0$,
                    we must have $R'\inc E_0 \cong_B S$.
                \item
                    If $S$ is compressible,
                    then we have $R'\oplus E_t\inc S\oplus E_t$,
                    since $R'\oplus E_t$ is hyperfinite and $S\oplus E_t \cong_B E_t$.
            \end{enumerate}
            In both cases,
            we get
            \[
                E
                \cong_B E\oplus E_t
                \cong_B R\oplus R'\oplus E_t
                \inc R\oplus S\oplus E_t
                \inc F\oplus E_t
                \cong_B F.
            \]
    \end{enumerate}
\end{proof}

The next result,
which is an immediate corollary of the above,
is a version of the Glimm-Effros dichotomy (\Cref{theorem-smooth-equivalence})
for the inclusion order $\inc$ instead of $\sqsubseteq_B$.

\begin{cor}\label{2.13}
    Let $E\in \ap$.
    Then exactly one of the following holds:
    \begin{enumerate}[label=(\roman*)]
        \item
            $E$ is smooth,
        \item
            $\R E_0 \inc E$.
    \end{enumerate}
\end{cor}

Thus, by \cref{2.13},
$\R E_0$ is a $\inc$-minimum among all the non-smooth relations in $\ap$ and,
by \cref{2.4},
$E_\infty\times I_\N$ is a $\inc$-maximum relation in $\ap$. 

It is interesting to consider the problem of existence of $\inc$-maximum elements in ${\e}_\kappa = \ap_\kappa \cap \e$ for other classes $\e\subseteq \ap$.
This is clearly the case if $\kappa = 0$ and $\e$ satisfies the conditions of \cref{2.4},
so we will consider $\kappa \ge 1$.

Clearly $\kappa E_0$ is $\inc$-maximum in $\aph_\kappa$.
Denote by $\boldsymbol{\apt}$
the subclass of $\ap$ consisting of the treeable equivalence relations.

\newcommand{\probtreeable}{
    Let $\kappa \ge 1$.
    Does $\apt_\kappa$ have a $\inc$-maximum element?
}
\begin{prob}\label{prob-treeable}
    \probtreeable
\end{prob}
If $E$ is $\inc$-maximum in ${\apt}_1$, then $\kappa E$ is $\inc$-maximum in $\apt_\kappa$, for every $1 \le \kappa \le \aleph_0$, so we will concentrate on the case $\kappa = 1$, i.e., the class of \textbf{uniquely ergodic} elements of $\apt$. We do not know the answer to this problem but we would like to point out that a positive answer has an implication in the context of the theory of measure preserving CBER, see \cite{Kec24}. 

Let $(X, \mu)$ be a standard Borel space equipped with a probability measure $\mu$.
We will consider as in \cite{Kec24} \textbf{pmp} CBER on $X$,
i.e., $\mu$-preserving CBER on $X$,
where we identify two such CBERs if they agree $\mu$-a.e.
Inclusion of pmp CBERs is also understood in the $\mu$-a.e. sense.
Such a relation is treeable if it has this property $\mu$-a.e.

\begin{prop}\label{215}
    If $E$ on a standard Borel space $X$ is a $\inc$-maximum uniquely ergodic,
    equivalence relation in $\apt$,
    with (unique) invariant measure $\mu$,
    then for every treeable pmp relation $F$ on $(X, \mu)$,
    there is some $T \in \Aut(X, \mu)$ with $T(F) \subseteq E$.
\end{prop}
\begin{proof}
    We will use the following lemma.
    \begin{lem}\label{2.16}
        Let $G$ be a treeable pmp {\CBER} on $(X,\mu)$.
        Then there is an ergodic treeable pmp {\CBER} $H$ on $(X,\mu)$ with $G\subseteq H$.
    \end{lem}
    \begin{proof}
        For each $T \in \Aut(X, \mu)$,
        denote by $E_T$ the equivalence relation induced by $T$.
        By \cite[Theorem 8]{CM14},
        the set of $T \in \Aut(X, \mu)$ such that $E_T$ is independent of $G$
        (see \cite[Section 27]{KM04} for the notion of independence)
        is comeager in $\Aut(X, \mu)$,
        equipped with the usual weak topology.
        So is the set of all ergodic $T \in \Aut(X, \mu)$, see \cite[Theorem 2.6]{Kec10}.
        Thus there is an ergodic $T \in \Aut(X, \mu)$ such that $E_T$ is independent of $G$.
        Then put $H = E_T\vee G$.
    \end{proof}
    
    By \cref{2.16}, we can assume that $F$ is ergodic.
    We can also assume that there is $F'\in \apt$ which agrees with $F$ $\mu$-a.e.
    By considering the ergodic decomposition of $F'$
    (see \Cref{ergodic-decomposition}),
    we can also assume that $\mu$ is the unique invariant measure for $F'$.
    Fix then a Borel automorphism $T \colon X\to X$ such that $T(F')\subseteq E$.
    Then both $T_*\mu$ and $\mu$ are $T(F')$-invariant.
    Since $T(F')$ is uniquely ergodic, it follows that $T_*\mu =\mu$,
    i.e., $T\in \Aut(X, \mu)$ and the proof is complete.
\end{proof}
\begin{remark}
    We note here that an analog of the conclusion of \cref{215} is valid for the class $\aph$.
    More precisely, let $X = \cantor$ and let $\mu$ be the usual product measure on $X$.
    Then for every hyperfinite pmp CBER $F$ on $(X, \mu)$,
    there is an automorphism $T \in \Aut(X, \mu)$ such that $T(F) \subseteq E_0$.
    This can be seen as follows:
    By \cite[5.4]{Kec10} (in which the aperiodicity of $E$ is not needed),
    we can find an ergodic hyperfinite pmp relation $F'$ such that $F\subseteq F'$.
    By Dye's Theorem (see, e.g., \cite[3.13]{Kec10}),
    there is an automorphism $T \in \Aut(X, \mu)$ such that
    $T(F')= E_0$ and thus $T(F) \subseteq E_0$.
\end{remark}

\chapter{Topological Realizations}\label{realizations}

\section{Dense realizations and Lusin marker schemes}
We will first use the results in \Cref{inclusion} to prove the following:

\begin{thm}\label{3.1}
    For every equivalence relation $E\in \ap$ and every perfect Polish space $Y$,
    there is a minimal topological realization of $E$ in $Y$.
\end{thm}
\begin{proof}
    First,
    since for every perfect Polish space $Y$
    there is a continuous bijection from the Baire space $\baire$ onto $Y$
    (see \cite[7.15]{Kec95}),
    we can assume that $Y = \baire$.
    Moreover by \cref{2.13},
    it is enough to prove this result for $E=\R E_0$ and $E= \R I_\N$.
    
    \textit{Case 1}: $\R E_0$.
    
    Consider the shift map of $\Z$ on $2^\Z$ with associated equivalence relation $F'$.
    Let $Y = \{x \in 2^\Z : \text{$[x]_{F'}$ is dense in $2^\Z$}\}$.
    Clearly $Y$ is a dense, co-dense $G_\delta$ set in $2^\Z$,
    so,
    in particular,
    it is a zero-dimensional Polish space
    (with the relative topology from $2^\Z$).
    We next check that every compact subset of $Y$ has empty interior (as a subset of $Y$).
    Indeed let $K \subseteq Y$ be compact.
    If now $V$ is open in $2^\Z$ and $\emptyset \neq V\cap Y\subseteq K$,
    then since $Y$ is dense in $2^\Z$,
    by looking at $V\setminus K$ we see that $V\subseteq K$,
    contradicting that $Y$ is also co-dense in $2^\Z$. 
    
    By \cite[7.7]{Kec95},
    $Y$ is homeomorphic to $\baire$.
    Moreover if $F = F' \uhr Y$,
    $F$ has dense classes and $|\EINV_F| = 2^{\aleph_0}$,
    since for each $p \in (0, 1)$,
    the set $Y$ is conull for the ergodic invariant measure
    $(p \delta_0 + (1-p)\delta_1)^\Z$ on $2^\Z$.
    Thus $F\cong_B \R E_0$.
    
    \textit{Case 2}: $\R I_\N$. 
    
    Consider the equivalence relation $R$ on $\baire$ given by
    \[
        x R y \iff \exists m \forall n \ge m (x_n = y_n).
    \]
    Let $A\subseteq \baire$ be an uncountable Borel partial transversal for $R$ (i.e., no two distinct elements of $A$ are in $R$). Then, as $R$ is not smooth, denoting by $B= [A]_R$ the $R$-saturation of $A$, we also have that $Y = \baire \setminus B$ is uncountable. Fix then a Borel bijection $f \colon A\to Y$ and let $F$ be the equivalence relation obtained by adding to each $[a]_R$, $a\in A$, the point $f(a)$. Then $F$ is a smooth CBER, so $F\cong_B \R I_\N$, and every $F$-class is dense in $\baire$.
\end{proof}

A \textbf{complete section} of an equivalence relation $E$ on $X$
is a subset $Y\subseteq X$ which meets every $E$-class.
Recall that a \textbf{vanishing sequence of markers} for a CBER $E$
is a decreasing sequence of complete Borel sections $\{A_n\}$ for $E$
such that $\bigcap_n A_n = \emptyset$.
A very useful result in the theory of CBER is the Marker Lemma of Slaman-Steel,
which asserts that every $E\in\ap$ admits a vanishing sequence of markers,
see \cite[Lemma 6.7]{KM04} for a proof.
We will see next that \cref{3.1} implies a strong new version of the Marker Lemma.
Recall also here the definition of Lusin marker scheme from
\cref{intro-top}.

\begin{thm}\label{3.3}
    Every $E\in \ap$ admits a Lusin marker scheme of type \textup{I}
    and a Lusin marker scheme of type \textup{II}.
\end{thm}
\begin{proof}
    Type I: By \cref{3.1}, we can assume that $E$ lives on $\baire$ and that every equivalence class is dense.
    Let then for each $s\in \N^n$, $A_s = \{ x : x \uhr n = s \}$. 
    
    Type II: By \cref{3.1}, we can assume that $E$ lives on $\R$ and that every equivalence class is dense. By induction on $n$, we can easily construct open sets $A_s, s\in \N^n$, such that $\{A_s\}_{s\in \N^{<\N}}$ is a Lusin marker scheme for $E$ and moreover it has the following properties:
    \begin{enumerate}[label=(\alph*)]
        \item
            Each $A_s, s \in \N^n, n \ge 1$,
            is contained in $(n, \infty)$;
        \item
            Each $A_s, s\in \N^n, n \ge 1$,
            has non-empty intersection with the interval $(k, k+1)$ for every $k \ge n$.
    \end{enumerate}
    Then clearly $\{A_s\}_{s\in \N^{<\N}}$ is of type II.
\end{proof}

\begin{remark}
    \leavevmode
    \begin{enumerate}[label=(\alph*)]
        \item
            We can also easily see that every $E\in \ap$ admits a
            \textbf{Cantor marker scheme} $\{A_s\}_{s \in 2^{<\N}}$ of each type,
            which is defined in an analogous way.
        \item
            By applying \cref{3.3} to $\R E$,
            and using the ccc property for category,
            we can see that every $E\in \ap$ admits a variant of
            a Lusin marker scheme of type I,
            where condition \ref{item-markerschemei}
            in the definition is replaced by the following condition:
            \begin{enumerate}[label=(\roman*)*, start=4]
                \item
                    For each $x \in \baire$,
                    $\bigcap_n A_{x \uhr n}$ has at most one element
                    and for a comeager set of $x$ it is empty.
            \end{enumerate}
    \end{enumerate}
\end{remark}

\section{Continuous action realizations}\label{realizations-cts}
Any CBER has a continuous action realization, i.e., a topological realization induced by a continuous action of a countable group on a Polish space. We will consider what additional properties of the action and the Polish space of the realization are possible. For example, we have the following:

\begin{prop}\label{BaireActionRealization}
    Every $E \in \ap$ has a continuous action realization
    in the Baire space $\baire$.
\end{prop}
\begin{proof}
    By the usual change of topology arguments,
    we can assume that $E$ is induced by a
    continuous action of a countable group on a $0$-dimensional space $X$.
    Let $P\subseteq X$ be the perfect kernel of $X$.
    Since $P$ is $E$-invariant and cocountable,
    we have $E \cong_B E \uhr P$
    (see \Cref{prelim-classes}),
    so we can assume that $X$ is perfect.
    Let then $D$ be a countable dense subset of $X$
    which is also invariant under the action and put $Y = X\setminus D$.
    Then again $E\cong_B E \uhr Y$.
    The space $Y$ is a nonempty, $0$-dimensional Polish space
    in which every compact set has empty interior and thus is homeomorphic to the Baire space
    (see \cite[Theorem 7.7]{Kec95}).
\end{proof}

Recall that a continuous action realization of a CBER $E$ on $X$
is a compatible Polish topology on $X$
for which $E$ is induced by a continuous action $\F_\infty \car X$.
\begin{defi}\label{3.5}
    A continuous action realization of a CBER is
    \begin{enumerate}[label=(\roman*)]
        \item \textbf{transitive}
            if $E$ is transitive
            (i.e. there is a dense $E$-class).
        \item \textbf{minimal}
            if $E$ is minimal
            (i.e. every $E$-class is dense).
        \item \textbf{compact}
            (resp. locally compact, $\sigma$-compact)
            if $X$ is compact
            (resp. locally compact, $\sigma$-compact).
    \end{enumerate}
\end{defi}
We will omit the word ``continuous'' when using these adjectives,
for instance in ``minimal action realization'',
``transitive, compact action realization'',
and so on.

We first note the following fact:

\begin{prop}\label{3.6}
    If $E\in\ap$ has a compact action realization
    or a transitive action realization on a perfect Polish space
    or a minimal action realization,
    then $E$ is not smooth.
\end{prop}
 
\begin{proof}
    Suppose a smooth $E$ has a compact action realization $F$, towards a contradiction. Then there is a compact invariant subset $K$ in which the action is minimal. Since $F \uhr K$ is also smooth, by \cite[8.46]{Kec95} some orbit in $K$ is non-meager in $K$, thus consists of isolated points in $K$. Minimality then implies that $K$ consists of a single infinite orbit, contradicting compactness.

    The proof of the case of a transitive action realization on a perfect Polish space or a minimal action realization follows also from \cite[8.46]{Kec95}.
\end{proof}

We first note here that the hypothesis of perfectness in \cref{3.6} is necessary.

\begin{prop}\label{3.7}
    Every smooth equivalence relation in $\ap$ has a transitive locally compact action realization
    (in some non-perfect space).
\end{prop}

\begin{proof}
    Let $\N = \bigsqcup_{q \in \Q} N_q$ be a decomposition of $\N$ into infinite sets indexed by the rationals.
    Define then recursively $\{z_n\}_{n \in \N} \subseteq \C$,
    with $\Imagpart z_n > 0$,
    $\Imagpart z_{n+1} < \Imagpart z_n$,
    $\Imagpart z_n \to 0$,
    and pairwise disjoint closed squares $S_n$ with center $z_n$ with $\Imagpart S_n > 0$ as follows:

    If $0 \in N_q$,
    choose $z_0 \in \{q\} \times \R$ and let $S_0$ be a very small square around $z_0$.
    At stage $n+1$,
    if $n+1 \in N_q$,
    choose $z_{n+1} \in \{q\}\times \R$ so that
    $0 < \Imagpart z_{n+1} < \frac{1}{n+1}$,
    $\Imagpart z_{n+1} < \Imagpart z_n$,
    $z_{n+1} \notin \bigcup_{m \le n} S_m$,
    and then choose $S_{n+1}$ to be a small square around $z_{n+1}$ so that
    it has empty intersection with all $S_m, m \le n$.

    Put $X = \R \cup \{z_n\}_{n \in \N}$.
    Then $X$ is closed in $\C$, so it is locally compact.
    Next define $T \colon X \to X$ as follows:

    If $x \in \R$, then $T(x) = x + 1$.
    If $x = z_n$ with $n \in N_q$, so that $x \in \{q\} \times \R$,
    and if in the increasing enumeration of $N_q$, $n$ is the $i$th element, then put $T(x) = z_m$,
    where $m$ is the $i$th element in the increasing enumeration of $N_{q+1}$.
    It is not hard to check that $T$ is a homeomorphism of $X$.
    For example, to check that $T$ is continuous (a similar argument works for $T^{-1}$),
    let $w_n, w \in X$, with $w_n \to w$, in order to show that $T(w_n) \to T(w)$.
    We can assume of course that $w_n \notin \R, w\in \R, \Imagpart w_n \to 0$.
    Now $\Realpart T(w_n) = \Realpart w_n + 1$ and $\Imagpart T(w_n) \to 0$,
    thus $T(w_n) = \Realpart w_n + 1 + i\Imagpart T(w_n) \to w + 1 = T(w)$.

    Next for each pair $(m,n) \in \N^2$,
    let $T_{m,n}$ be the homeomorphism of $X$ that switches $z_m$ with $z_n$ and keeps every other point of $X$ fixed.
    Then the group generated by all $T_{m,n}$ and $T$ acts continuously on $X$.
    One of its orbits is $\{z_n\}$ which is dense in $X$, thus the action is topologically transitive.
    The equivalence relation $F$ it generates has as classes the set $\{z_n\}$ and the sets of the form  $x+\Z$,
    for $x\in \R$, so it is aperiodic and smooth, with transversal $\{z_0\} \cup [0,1)$.
\end{proof}

Also the hypothesis of compactness in \cref{3.6} is necessary.

\begin{prop}
    Every smooth equivalence relation in $\ap$ has a locally compact action realization on a perfect space,
    in fact one in the space $\cantor\setminus \{\bar{1}\}$, where $\bar{1}$ is the constant 1 sequence.
\end{prop}
\begin{proof}
    We use an example in \cite[page 200, (b)]{DJK94}. Consider the space $X = \cantor\setminus \{\bar{1}\}$. For each $m \neq n$, let $h_{m,n}$ be the homeomorphism of $X$ defined by: $h_{m,n}(1^m0\concat y) = 1^n0\concat y, h_{m,n}(1^n0\concat y) = 1^m0\concat y, h_{m,n}(x) =x$, otherwise. Then the group generated by these homeomorphisms acts continuously on $X$ and generates the equivalence relation $F$ given by: $x F y \iff \exists z(x =1^m0\concat z \ \& \ y =1^n0\concat z)$, which is smooth aperiodic. 
\end{proof}
We next show that non-smooth hyperfinite equivalence relations in $\ap$ have the strongest kind of topological realization. For a countable group $\Gamma$, recall that a subshift of $2^\Gamma$ is the restriction of the shift action of $\Gamma$ to a nonempty closed invariant subset.

\begin{thm}\label{3.9}
    Every non-smooth hyperfinite equivalence relation in $\ap$ has
    a minimal, compact action realization on the Cantor space $\cantor$.
    In fact, we have the following:
    \begin{enumerate}[label=(\arabic*)]
        \item
            If it is compressible,
            then it can be realized by a minimal subshift of $2^{\F_2}$.
        \item
            If it is not compressible,
            then it can be realized by a minimal subshift of $2^\Z$.
    \end{enumerate}
\end{thm}
\begin{proof}
    \leavevmode
    \begin{enumerate}[label=(\arabic*)]
        \item
            Consider $E_t$.
            Then $E_t$ is generated by a continuous action of $\F_2$,
            see \cite[Section 3.2]{Kec25},
            defined as follows: 
            The first generator acts via
            $i\concat x \mapsto (1-i)\concat x$,
            and the second generator acts via
            \begin{align*}
                00 \concat x & \mapsto 0 \concat x \\
                1 \concat x & \mapsto 11 \concat x \\
                01 \concat x & \mapsto 10 \concat x
            \end{align*}
            This action has a clopen 2-generator,
            namely the partition given by
            $\{X_0 = 0 \concat \cantor, X_1 = 1 \concat \cantor\}$.
            This means that the sets $\gamma\cdot X_i, \gamma \in \F_2, i \le 1$, separate points.
            This implies that this action is (topologically) isomorphic to a subshift of $2^{\F_2}$.
        \item
            Assume that $E\in \aph$ is non-compressible and let $\kappa = |\EINV_E| >0$.
            By a theorem of Downarowicz \cite[Theorem 5]{Dow91},
            for every metrizable Choquet simplex $K$,
            there is a minimal subshift of $2^\Z$ such that
            $K$ is affinely homeomorphic to the simplex of invariant measures for this subshift.
            In particular the cardinality of the set of ergodic, invariant measures for this subshift
            is the same as the cardinality of the set of extreme points of $K$.
            Fix now a compact Polish space $X$ of cardinality $\kappa$ and
            let $K$ be the Choquet simplex of measures on $X$.
            The extreme points are the Dirac measures, so there are exactly $\kappa$ many of them.
            Thus we can find a minimal subshift of $2^\Z$ with exactly $\kappa$ many ergodic, invariant measures
            and therefore if $F$ is the equivalence relation induced by this subshift,
            we have that $E\cong_B F$.
    \end{enumerate}
\end{proof}

Although $E_t$ does not have a minimal, compact action realization where the acting group is amenable (otherwise it would have an invariant measure), we have the following:

\begin{prop}
    A compressible, non-smooth, hyperfinite {\CBER} has a minimal,
    locally compact action realization where the acting group is $\Z$.
\end{prop}
\begin{proof}
    It is known that there are minimal homeomorphisms on uncountable locally compact spaces with no invariant measure,
    which thus generate a compressible non-smooth hyperfinite CBER; see, e.g., \cite[Section 2]{Dan01}.
    Below we give a simple example:

    Let $A = \Z/4\Z$ as an abelian group,
    and let $X\subseteq A^\N$ be the set of sequences
    which eventually lie in $\{1, 2\}$.
    Consider $X_n = A^n\times \{1, 2\}^\N$,
    so that $X_0\subseteq X_1 \subseteq X_2\subseteq \dots$
    and $X = \bigcup_n X_n$.
    We give $X_n$ the usual product topology,
    so that $X_n$ is clopen in $X_{n+1}$,
    and we give $X$ the inductive limit topology,
    so that $U\subseteq X$ is open iff
    $\forall n (U\cap X_n$ is open in $X_n$).
    This is Hausdorff,
    locally compact and second countable,
    with basis $\bigcup_n {\B}_n$,
    where ${\B}_n$ is a countable basis for $X_n$.
    Thus $X$ is a locally compact Polish space,
    see, e.g., \cite[5.3]{Kec95}.
    Note that this topology on $X$ is finer than the subspace topology from $A^\N$
    since the $X_n$ are not open in $A^\N$.
    
    Let now $\phi \colon A^\N \to A^\N$ be the odometer map,
    i.e., addition by $1$ with carry,
    which is a homeomorphism of $A^\N$.
    Note that $\phi(X)\subseteq X$ and $\phi^{-1}(X)\subseteq X$.
    We next check that $\phi\uhr X$ is a homeomorphism of $X$.
    It enough to check that $\phi\uhr X_n \colon X_n\to X$
    and $\phi^{-1}\uhr X_n \colon X_n \to X$ are continuous.
    This follows from noticing that $\phi(X_n) \subseteq X_{n+1}$
    and $\phi^{-1}(X_n) \subseteq X_{n+1}$.
    
    Let $E$ be the equivalence relation on $X$ induced by $\phi\uhr X$.
    Denote by $E_0'$ the equivalence relation on $A^\N$ defined by
    $x E_0' y \iff \exists m \forall n \ge m (x_n = y_n)$.
    Then $E = E_0'\uhr X$ and $E\uhr X_n = E_0'\uhr X_n$,
    so $\phi\uhr X$ is minimal, i.e., has dense orbits. 
    
    Finally,
    we show that $E$ is compressible.
    For every $x\in X$, let $n_x$ be least such that $x\in X_{n_x}$,
    and define the Borel map $f \colon X \to X$ as follows:
    \[
        f(x)_n =
        \begin{cases}
            x_n + 2 & n = n_x \\
            x_n & n\neq n_x
        \end{cases}
    \]
    Then $f$ is a compression of $E$.
\end{proof}

\begin{remark}
    Here are also some other minimal, locally compact action realizations of a compressible, non-smooth, hyperfinite CBER (but where the acting group is not $\Z$).
    
    \begin{enumerate}[label=(\arabic*)]
        \item
            Let $X$ be the locally compact space constructed in the proof of \cref{3.7}, whose notation we use below. For each $q\in \Q$, let $T_q \colon X\to X$ be the homeomorphism which is translation by $q$ on $\R$ and defined on $\{z_n\}$ in a way similar to translation by 1 in the proof of \cref{3.7}. Also define a homeomorphism $T \colon X\to X$ as follows: $T$ is the identity on $\R$. Next let for each $q\in \Q$, $N_q = \{n^q_0 < n^q_1 < n^q_2 < \dots\}$ be the increasing enumeration of $N_q$ and define $T(z_{n^q_{2n+3}}) = z_{n^q_{2n+1}}, T(z_{n^q_1}) = z_{n^q_0}, T(z_{n^q_{2n}}) = z_{n^q_{2n+2}}, n\in \N$. 
        
            The group generated by $T, T_q, q\in \Q$ is abelian and acts continuously on $X$. The orbits consist of $\{z_n\}$ and the sets of the form $x+\Q$ for $x\in \R$, so the action is minimal. Finally there is clearly no invariant measure for this action.
    
        \item
            Another construction, where the acting group is actually $\Z^2$ is the following: Let $S$ be a minimal homeomorphism on an uncountable  compact metric space $K$, inducing the equivalence relation $F$, and let $X= K\times \Z$. Then let $\Z^2$ act by homeomorphisms on $X$, where one of the generators acts like $S$ on $K$ and the other as translation by 1 on $\Z$. The associated equivalence relation of this action is Borel isomorphic to $F\times I_N$ so it is compressible, non-smooth and hyperfinite.
    \end{enumerate}
\end{remark}

We next discuss an implication of \cref{3.9} to a new characterization of non-smoothness of a CBER.

Below for a Borel action of a countable group $\Gamma$ on a standard Borel space $X$
and a probability measure $\zeta$ on $\Gamma$ whose support generates $\Gamma$ as a semigroup,
we say that a measure $\mu$ on $X$ is \textbf{$\boldsymbol{\zeta}$-stationary}
if $\mu = \int \gamma_* \mu \dd{\zeta(\gamma)}$.
 
It is easy to see that if $\mu$ is $\zeta$-stationary,
then $\mu$ is \textbf{quasi-invariant} under the action,
i.e., the action sends $\mu$-null sets to $\mu$-null sets.
Next we check that if the action has infinite orbits, then $\mu$ is non-atomic.
We show that every orbit $O$ is $\mu$-null.
By aperiodicity,
it suffices to show that every point in $O$ has the same $\mu$-measure.
Let $O' = \{x \in O : \forall y \in O \; [\mu(\{y\}) \le \mu(\{x\})]\}$.
Since $O$ has bounded measure,
the set $O'$ is non-empty.
It suffices to show that $O' = O$.
Given $x \in O'$,
we have $\mu(\{x\}) = \int \mu(\{\gamma^{-1} x\}) \dd \zeta(\gamma)$,
so by maximality,
for every $\gamma \in \Supp(\zeta)$,
we have $\mu(\gamma^{-1} x) = \mu(\{x\})$,
and hence $\gamma^{-1} x \in O'$.
Thus $\Supp(\zeta)^{-1} O' \subseteq O'$.
Since $\Supp(\zeta)^{-1}$ generates $\Gamma$ as a semigroup,
we have $O = \Gamma O' \subseteq O'$.

We use these facts and \cref{3.9} to prove the following:

\begin{prop}
    Let $E\in\ap$ be an equivalence relation on a standard Borel space $X$.
    Then the following are equivalent:
    \begin{enumerate}[label=(\roman*)]
        \item \label{item-stationary-nonsm}
            $E$ is not smooth;
        \item \label{item-stationary-every-meas}
            There is a Borel action of a countable group $\Gamma$ on $X$ generating $E$,
            such that for \textup{every} measure $\zeta$ on $\Gamma$,
            whose support generates $\Gamma$ as a semigroup,
            there is a $\zeta$-stationary, ergodic for this action measure on $X$.
        \item \label{item-stationary-some-meas}
            There is a Borel action of a countable group $\Gamma$ on $X$ generating $E$,
            such that for \textup{some} measure $\zeta$ on $\Gamma$,
            whose support generates $\Gamma$ as a semigroup,
            there is a $\zeta$-stationary, ergodic for this action measure on $X$.
    \end{enumerate}
\end{prop}
\begin{proof}
    If \ref{item-stationary-some-meas} holds,
    then $E$ admits a non-atomic, ergodic, quasi-invariant measure,
    so it is not smooth.
    Thus \ref{item-stationary-some-meas} implies \ref{item-stationary-nonsm}.
    Clearly \ref{item-stationary-every-meas} implies \ref{item-stationary-some-meas}.
    We finally prove that \ref{item-stationary-nonsm} implies \ref{item-stationary-every-meas}.
    
    Since $E$ is not smooth,
    by the Glimm-Effros dichotomy (\Cref{theorem-smooth-equivalence}),
    there is an $E$-invariant Borel set $Y\subseteq X$ such that $E \uhr Y$ is non-smooth, hyperfinite.
    Then, by \cref{3.9},
    there is a continuous action of a group $\Gamma$ on a compact space $Z$
    inducing an equivalence relation $F\cong_B E \uhr Y$.
    Let $\zeta$ by any measure on $\Gamma$, whose support generates $\Gamma$ as a semigroup.
    Then there is a $\zeta$-stationary for this action measure on $Z$, see, e.g.,  \cite[Corollary 9]{CKM13}.
    The set of $\zeta$-stationary for this action measures is thus
    a non-empty compact, convex set of measures,
    so it has an extreme point which is therefore ergodic.
    Transferring this back to $Y$ and extending the $\Gamma$ action to $X$
    so that it generates $E \uhr (X \setminus Y)$ on $X\setminus Y$,
    we see that \ref{item-stationary-every-meas} holds.
\end{proof}

The following question is open:

\newcommand{\probrealization}{
    Does \textup{every} non-smooth $E\in \ap$ have any of the topological realizations stated in \cref{3.5}?
    In particular,  does \textup{every} non-smooth $E\in \ap$ admit a compact action realization?
}
\begin{prob}\label{prob-realization}
    \probrealization
\end{prob}
We will consider the case of compact action realizations in the next two sections.
The answer to the following is also unknown:

\newcommand{\probcantor}{
    If a {\CBER} admits a compact action realization, does it admit one in which the underlying space is $\cantor$?
}
\begin{prob}\label{prob-cantor}
    \probcantor
\end{prob}

We note here that the following weaker version of \cref{prob-realization} is also open:

\newcommand{\probbired}{
    Is \textup{every} non-smooth $E\in \ap$ Borel bireducible to some $F\in \ap$ which has any of the topological realizations stated in \cref{3.5}? In particular, can one find such an $F$ that admits a compact action realization?
}
\begin{prob}\label{prob-bired}
    \probbired
\end{prob}

\section{Compact action realizations}\label{realizations-cpt}

\subsection{\nopunct}
We have seen in \cref{3.9} that the answer to \cref{prob-realization}
is affirmative in the strongest sense for hyperfinite $E$
but the situation for general $E$ is unclear.
The following results provide several cases of non-hyperfinite equivalence relations
that admit compact action realizations.

\begin{thm}\label{3166}
For every infinite countable group $\Gamma$, $F(\Gamma, \cantor)$ admits a compact action realization. If $\Gamma$ is also finitely generated, then $E^{\mathrm{ap}}(\Gamma, \cantor)$ admits a compact action realization. In fact in both cases such a realization can be taken to be a subshift of $(\cantor)^\Gamma$.
\end{thm}

\begin{proof}
    The result will follow easily from the following lemma, which is an extended version of the result in \cite{Ele18}, who dealt with the case of free actions. Denote below by $\mathbf{s}= \mathbf{s}_{\Gamma, 3^\N}$ the shift action of $\Gamma$ on $(3^{\N})^\Gamma$ and for each $y\in (3^{\N})^\Gamma$, let $\Stab(y) \le \Gamma$ be the stabilizer of $y$ in this action.

    \begin{lem}\label{332}
        Let $\mathbf{a}$ be a Borel action of an infinite countable group $\Gamma$ on a standard Borel space $X$. Then there is an equivariant Borel embedding $p \colon X \to (3^{\N})^\Gamma$ of the action $\mathbf{a}$ to the shift action $\mathbf{s}$
        of $\Gamma$ on $(3^{\N})^\Gamma$ such that if $y\in \overline{p(X)}$, then there are $y_0, y_1, \dots \in p(X)$ with
        \[
            \Stab(y) \subseteq \bigcup_m\bigcap_{n \ge m} \Stab(y_n).
        \]
    \end{lem}
    
    \begin{proof}
        Let $\Gamma = \{\gamma_n\}$ and let $\varphi_n \colon X \to \{0,1,2\}$ be a Borel coloring of the graph of $\gamma_n^{\mathbf{a}}$, where $\gamma_n^\mathbf{a} (x) = \mathbf{a} (\gamma_n, x)$ (see \cite[4.6]{KST99}). By changing the topology, we can assume that $X$ is 0-dimensional, so a $G_\delta$ subset of $3^\N$, and $\mathbf{a}, \varphi_n$ are continuous. 
        
        Let $\theta \colon X \to 3^\N$ be defined by 
        \[
            \theta(x) = (x(0), \varphi_0 (x), x(1), \varphi_1 (x), \dots).
        \]
        Then $\theta$ is a homeomorphism of $X$ with a $G_\delta$ subset of $3^\N$, and we have that if $\gamma_m^\mathbf{a} \cdot x = y \neq x$, then $\theta (y) (2m+1) \neq \theta (x) (2m+1)$, where we write $\gamma_m^\mathbf{a} \cdot x$ for $\gamma_m^\mathbf{a} (x)$.
        
        Thus identifying $x$ with $\theta(x)$, we can thus assume that
        
        \begin{enumerate}[label=(\roman*)]
            \item
                $X$ is a $G_\delta$ subset of $3^\N$,
            \item
                The action $\mathbf{a}$ is continuous,
            \item \label{item-elek-color-distinct}
                $\gamma_m^\mathbf{a} \cdot x = y \neq x \implies y(2m + 1) \neq x(2m + 1)$.
        \end{enumerate}
        
        Consider now the standard equivariant Borel embedding of the action $\mathbf{a}$
        into the shift action on $(3^\N)^\Gamma$ given by 
        \[
            p \colon X \to (3^\N)^\Gamma
        \]
        \[
            x\mapsto (\gamma\mapsto (\gamma^{-1})^\mathbf{a} \cdot x).
        \]
        Let now $y\in \overline{p(X)}$ and take $x_n \in X$ with $y_n = p(x_n)\to y$.
        
        Let $\gamma = \gamma_m \in \Stab(y)$.
        Then, as $p(x_n)\to y, p(x_n)(1)\to y(1)$, i.e., $x_n\to y(1)$.
        Also $\gamma^\mathbf{s}\cdot p(x_n) \to \gamma^\mathbf{s}\cdot y$,
        so $(\gamma^\mathbf{s}\cdot p(x_n))(1) \to (\gamma^\mathbf{s}\cdot y) (1)$,
        i.e., $p(x_n)(\gamma^{-1})\to y(\gamma^{-1})$ or
        $\gamma^\mathbf{a} \cdot x_n \to (\gamma^\mathbf{s} \cdot y) (1) = y(1) $,
        so both $x_n (2m+1)$ and $ (\gamma^\mathbf{a} \cdot x_n ) (2m+1)$ converge to $y(1)(2m+1)$,
        in the discrete space $\{0,1,2\}$,
        thus they are eventually equal.
        Then by \ref{item-elek-color-distinct} above $x_n, \gamma^\mathbf{a} \cdot x_n$ are eventually equal,
        so $\gamma\in \Stab(x_n) = \Stab(y_n)$,
        for all large enough $n$.
    \end{proof}
    
    From \cref{332} it is clear that if $\mathbf{a}$ is a free action, then $\overline{p(X)}$ is contained in $\Fr((3^\N)^\Gamma)$ and thus in particular by taking $\mathbf{a}$ to be the restriction of the shift action to $\Fr((3^\N)^\Gamma)$, we see that this action is Borel isomorphic to a subshift contained in $\Fr((3^\N)^\Gamma)$ (we are using here that Borel invariant biembeddability implies Borel isomorphism). Of course we can trivially replace $3^\N$ by $\cantor$ here, so this proves the first  statement of the theorem.
    
    Assume now that $\Gamma$ is finitely generated. In a similar way, to prove the second statement of the theorem, it is enough to show, in the notation of \cref{332}, that if the action $\mathbf{a}$ is aperiodic, so that the stabilizers of the points $y_n$ have infinite index, then the stabilizer of the point $y$ also has infinite index. This follows from the conclusion of \cref{332} and the fact that a finite index subgroup of a finitely generated group is also finitely generated.
\end{proof}

\newcommand{\probaperiodiccompact}{
    Is it true that for an \textup{arbitrary} infinite countable group $\Gamma$,
    $E^{\mathrm{ap}}(\Gamma, \cantor)$ admits a compact action realization?
}
\begin{prob}\label{prob-aperiodiccompact}
    \probaperiodiccompact
\end{prob}

We next note the following fact,
which can be used to provide more examples of CBER that admit compact action realizations.
\begin{prop}
    Let $E$ be an aperiodic {\CBER} on an uncountable Polish space $X$.
    Then there is an invariant meager Borel set $M \subseteq X$
    such that $E \cong_B E \uhr M$.
\end{prop}
Note that given such an $M$,
if $Y$ is an invariant Borel set containing $M$,
then $E \cong_B E \uhr M \sqsubseteq^i_B E \uhr Y \sqsubseteq^i_B E$,
and hence $E \cong_B E \uhr Y$
(see \Cref{prelim-eqrel}).
So in particular,
if $X$ is compact and $E$ is the orbit equivalence relation of a continuous action,
then $E\uhr Y$ has a compact action realization as well.

\begin{proof}
    If $E$ is compressible hyperfinite,
    then $E$ is Borel isomorphic to either $\R I_\N$ or $E_t$
    (see \Cref{thm-hf-classif}),
    and thus $E \cong_B \R E$.
    So by the countable chain condition for category,
    one of these copies of $E$ must be meager,
    and we are done by setting $M$ to be this meager set.

    So assume that $E$ is not both compressible and hyperfinite.
    Fix by \cite[12.1 and 13.3]{KM04}
    an invariant meager Borel set $M \subseteq X$
    such that $E \uhr (X \setminus M)$ is compressible hyperfinite.
    Then $E \uhr M$ is non-smooth,
    since if it were smooth,
    then it would be compressible hyperfinite,
    making $E$ compressible hyperfinite.
    Thus
    $E \uhr M \cong_B (E \uhr M) \oplus (E \uhr (X \setminus M)) \cong_B E$
    (see \Cref{prelim-classes}).
\end{proof}

Since for every $E \in \ap$ on a Polish space $X$
there is an invariant comeager Borel set $Y\subseteq X$
such that $E \uhr Y$ is hyperfinite,
it follows that if $E\in \ap$ is not smooth
when restricted to any invariant comeager Borel set,
then there is an invariant comeager Borel set $Y\subseteq X$
such that $E \uhr Y$ admits a minimal, compact action realization.
Whether this holds for measure instead of category is an open problem.

\newcommand{\probconull}{
    Let $E\in \ap$ be on a standard Borel space $X$
    and let $\mu$ be a measure on $X$
    such that the restriction of $E$ to any invariant Borel set
    of measure 1 is not smooth.
    Is there is an invariant Borel set $Y\subseteq X$ with $\mu(Y) = 1$
    such that $E \uhr Y$ admits a compact action realization? 
}
\begin{prob}\label{prob-conull}
    \probconull
\end{prob}

\subsection{\nopunct}
We next describe a ``gluing'' construction of two continuous actions of
groups on compact Polish spaces at an orbit of one of the actions
and apply it to the compact action realization problem.
We thank Aristotelis Panagiotopoulos for a useful discussion on this construction. 

Let the countable group $\Gamma$ act continuously on the compact Polish space $X$ and
let $X_0\subseteq X$ be an infinite orbit of this action.
Let also the countable group $\Delta$ act continuously on the compact Polish space $Y$
with a fixed point $y_0 \in Y$.
Fix compatible metrics $d_X \le 1$ and $d_Y \le 1$ for $X$ and $Y$, respectively.
Fix also a map $x\mapsto |x|$ from $X_0$ to $\R^+$ such that $\lim _{x\to \infty} |x| = \infty$, i.e., for every $M\in \R^+$, there is a finite $F\subseteq X_0$ such that $x \notin F \implies |x| > M$.
For each $x\in X_0$,
let $Y_x \ni x$ be a set and
let $\pi_x$ be a bijection $\pi_x \colon Y \to Y_x$ such that $\pi_x (y_0) = x$ and $x_1 \neq x_2 \implies Y_{x_1} \cap Y_{x_2} =\emptyset$. Put $Y'_x = Y_x\setminus \{x\}$ and let $Z= X\sqcup \bigsqcup_{x\in X_0} Y'_x$. Define a metric $d_x$ on $Y_x$ as follows:
\[
    d_x (y_1, y_2) = \frac{d_Y(\pi_x^{-1}(y_1), \pi_x^{-1}(y_2))}{|x|}.
\]
Then define  a metric $d_Z$ on $Z$ as follows:
\[
    d_Z(x_1, x_2) = d_X(x_1, x_2), \textrm{if}  \ x_1, x_2\in X,
\]
\[
    d_Z(y_1, y_2) = d_x (y_1, y_2), \textrm{if} \ y_1, y_2 \in Y_x, x\in X_0,
\]
\[
    d_Z(y,x') = d_x(y,x) + d_X(x,x'), \textrm{if} \  y\in Y_x, x\in X_0, x'\in X,
\]
\[
    d_Z(y_1, y_2) = d_{x_1}(y_1, x_1) + d_X(x_1, x_2) + d_{x_2}(x_2, y_2), \textrm{if} \ y_1\in Y_{x_1}, y_2\in Y_{x_2}, x_1 \neq x_2.
\]

\begin{remark}\label{gluingRemark}
    We note here that in the preceding ``gluing'' construction, if the spaces $X,Y$ are 0-dimensional, so is the space $Z$. To see this we start with metrics $d_X, d_Y$ as above which are actually ultrametrics (these exist since $X,Y$ are 0-dimensional). Then it is enough to show that for every $z\in Z$, there is an $\epsilon_z >0$ such that every open ball (in the metric $d_Z$) $B_z (\epsilon) $, for $\epsilon < \epsilon_z$, is closed. Below recall that open balls in ultrametrics are closed.

    Consider first the case where $ z\in X$ and fix $z_1, z_2, \dots \in B_z(\epsilon)$ with $z_n \to z_\infty$. If infinitely many $z_n$ are in $X$, then clearly $z_\infty\in B_z(\epsilon)$ as $d_X$ is an ultrametric. Otherwise, we can assume that all $z_n$ are in $Z\setminus X$. If now there is some $x\in X_0$ such that infinitely many $z_n \in Y'_x$, so that $z_\infty \in Y_x$, we have $d_Z(z_n, z) = d_x(z_n, x) + d_X(x,z)$, so $d_x(z_n, x) < \epsilon - d_X(x,z)$, thus, since $d_x$ is an ultrametric, $d_x(z_\infty , x) <  \epsilon - d_X(x,z)$ and thus $d_Z(z_\infty , z) < \epsilon$. Otherwise there is a subsequence $(z_{n_i})$ and $x_i \in X_0$ with $z_{n_i}\in Y'_{x_i}$ and $x_i$ converges to $x\in X$ and thus $z_{n_i}\to z_\infty = x$ (since $d_Z(z_{n_i}, x_i)  < \frac{1}{|x_i|})$. Now $d_Z(z, z_{n_i}) = d_X( z, x_i) + d_{x_i}(x_i, z_{n_i}) < \epsilon$, so $d_X (z, x_i) < \epsilon$ and, since $d_X$ is an ultrametric, $d_Z(z, z_\infty) = d_X (z, z_\infty) < \epsilon$. 
    
    The other case is when $z\in Y'_x$, for some $x\in X_0$. Take $\epsilon_z = d_x(z,x)$. Then for $\epsilon < \epsilon_z$ the open ball  $B_z (\epsilon)$ is the same as the open ball of radius $\epsilon$ in the metric $d_x$, so the proof is complete.
\end{remark}

\begin{prop}
    $(Z,d_Z)$ is a compact metric space.
\end{prop}
\begin{proof}
    It is routine to check that $d_Z$ is a metric on $Z$. We next verify compactness. Let $(z_n)$ be a sequence in $Z$ in order to find a converging subsequence. The other cases being obvious, we can assume that $z_n\in Y_{x_n}$ with $x_n\in X_0$ distinct and therefore $|x_n|\to \infty$, in which case, by going to a subsequence, we can also assume  that  $x_n\to x\in X$. Since $d_Z(z_n, x_n) \le \frac{1}{|x_n|}$, it follows that $z_n\to x$.
\end{proof}
We next define an action of $\Delta$ on $Z$. Given $\delta\in \Delta$ and $z\in Z$ we define $\delta\cdot z$ as follows:

\[
    \delta\cdot z = \pi_x(\delta\cdot \pi_x^{-1}(z)), \textrm{if} \ z\in Y_x, x\in X_0,
\]
\[
    \delta\cdot z = z, \textrm{if} \ z\in X.
\]
If we identify each $Y_x$ with $Y$, then this action ``extends'' the action of $\Delta$ on $Y$.

We finally extend the action of $\Gamma$ from $X$  to all of $Z$. Given $\gamma \in \Gamma$ and $z\in Y_x, x\in X_0$, define $\gamma\cdot z$ as follows:
\[
\gamma \cdot z =  \pi_{\gamma\cdot x}(\pi_x^{-1}(z)), \textrm{if} \ z\in Y_x, x\in X_0.
\]
It is easy to see that these two actions commute, so they give an action of $\Gamma\times \Delta$ on $Z$.
\begin{prop}
    The action of $\Gamma\times \Delta$ on $Z$ is continuous.
\end{prop}
\begin{proof}
    It is enough to check that the action of $\Gamma$ on $Z$ is continuous and so is the action of $\Delta$.
    
    Let first $\gamma\in \Gamma$ and $z_n\in Z$ be such that $z_n \to z$, in order to show that $\gamma\cdot z_n \to \gamma\cdot z$. It is enough to find a subsequence $(n_i)$ such that $\gamma \cdot z_{n_i} \to \gamma\cdot z$. Again, the other cases being trivial, we can assume that $z_n\in Y_{x_n}$ with $x_n\in X_0$ distinct, so that also $|x_n|\to \infty$, in which case, by going to a subsequence, we can also assume that $x_n\to x\in X$. Then $\gamma \cdot x_n \to \gamma\cdot x$ and $d_Z(\gamma\cdot z_n,\gamma\cdot  x_n) \le \frac{1}{|\gamma\cdot x_n|}\to 0$, as the $\gamma\cdot x_n$ are also distinct and thus $|\gamma\cdot x_n| \to \infty$. Since $d_Z( z_n, x_n) \le \frac{1}{|x_n|}$, clearly $x=z$, and thus  $\gamma\cdot z_n \to \gamma\cdot z$.
    
    Let now $\delta \in \Delta$ and  $z_n\in Z$ be such that $z_n \to z$, in order to show that $\delta\cdot z_n \to \delta\cdot z$. It is enough again to find a subsequence $(n_i)$ such that $\delta \cdot z_{n_i} \to \delta \cdot z$ and as before we can assume that $z_n\in Y_{x_n}$ with $x_n\in X_0$ distinct, so that also $|x_n|\to \infty$, in which case, by going to a subsequence, we can also assume that $x_n\to x\in X$. Then $\delta \cdot x_n = x_n \to \delta \cdot x =x$. Now $\delta\cdot z_n \in Y_{x_n}$, so that $\delta_Z( \delta \cdot z_n , x_n) \to 0$ and $d_Z( z_n , x_n) \to 0$. Thus $z=x$ and $\delta\cdot z_n \to \delta\cdot z =z$.
\end{proof}

Let now $E$ be the equivalence relation induced by the action of $\Gamma$ on $X$,  let $F$ be the equivalence relation induced by the action of $\Delta$ on $Y\setminus \{y_0\}$ and finally let $G$ be the equivalence relation induced by the action of $\Gamma\times \Delta $ on $Z$. Then it is easy to check the following;

\begin{prop}\label{321}
    $G \cong_B E\oplus (F\times I_\N)$
\end{prop}

We present now an application of this construction to the problem of compact action realizations.

\begin{thm}\label{322}
    Let $F$ be a non-smooth {\CBER}
    which admits a locally compact action realization.
    Then $F\times I_\N$ admits a compact action realization.
    In particular,
    if $F$ is compressible,
    then $F$ admits a compact action realization.
    
    Moreover,
    if the locally compact space is $0$-dimensional,
    $F\times I_\N$ admits a compact action realization on the Cantor space $\cantor$.
\end{thm}
\begin{proof}
    In the preceding ``gluing'' construction, take $X = \cantor$ and a continuous action of $\Gamma= \F_2$ such that $E= E_t$. Fix also a countable group $\Delta$ and a continuous action of $\Delta$ on a locally compact space $Y'$ which induces $F$. Let $Y= Y'\sqcup \{y_0\}$ be the one-point compactification of $Y'$ (if $Y'$ is already compact, we obtain $Y$ by adding an isolated point to $Y'$). Then the action of $\Delta$ can be continuously extended to $Y$ by fixing $y_0$. Thus we have by \cref{321} that $E\oplus (F\times I_{\N})$ admits a compact action realization.
    Since $F$ is not smooth,
    we have $E \oplus (F\times I_\N) \cong_B (F\times I_\N)$
    (see \Cref{prelim-classes}).
    
    In the case where $Y'$ is 0-dimensional,
    by \cref{gluingRemark} $F\times I_\N$ admits a compact action realization
    on a $0$-dimensional space $Z$.
    By going to the perfect kernel of $Z$,
    we can assume that $Z$ is perfect (see the proof of \cref{BaireActionRealization}),
    thus homeomorphic to the Cantor space.
\end{proof}

The following is an immediate  consequence of \cref{322}.

\begin{cor}\label{323a}
    Let each $E_n\in \ap$ admit a compact action realization. Then $\bigoplus_n E_n\times I_\N$ also admits a compact action realization. In particular, if also every $E_n$ is compressible,  $\bigoplus_n E_n$ admits a compact action realization.
\end{cor}

Note that by \cite[Proposition 5.23 (d)]{CK18},
there is a unique,
up to Borel isomorphism, compressible, universal CBER.
The following are immediate consequences of \cref{322}.
\begin{cor}\label{3233}
    Let $E$ be a compressible, universal {\CBER}. Then $E$ admits a transitive, compact action realization on the Cantor space $\cantor$.
\end{cor}
\begin{proof}
    Let us first note that there exists a compressible, universal CBER $F$ that is generated by a continuous action of a countable group on $\cantor$.
    Indeed, let $E(\mathbb{F}_2,2)$ be the equivalence relation generated by the shift action of $\mathbb{F}_2$ on $2^{\mathbb{F}_2}$. Consider the equivalence relation $F = E(\mathbb{F}_2,2) \times I_{\mathbb{N}}$. This equivalence relation is compressible, universal. By \cref{322}, $F$ has a continuous action realization on $\cantor$. An inspection of the ``gluing'' construction involved in the proof of \cref{322} shows that this action is topologically transitive.
\end{proof}

\begin{cor}\label{3266}
    Let $E$ be a compressible, universal {\CBER}.
    Then $E$ admits a minimal action realization on the Baire space $\baire$.
\end{cor}
\begin{proof}
    By \cref{3233} consider a continuous action $\boldsymbol{a}$ of a countable group $\Gamma$ on $\cantor$, which induces an equivalence relation $F$ which is Borel isomorphic to $E$.  Then consider the Borel map $f$ that sends $x\in \cantor$ to the closure of its orbit (which is a member of the space of all compact subsets of $\cantor$),
    By \cite[Theorem 3.1]{MSS16}, there is some $K$ such that $F\uhr f^{-1}(K)$ is universal. But clearly $Z =f^{-1}(K)$ is a $G_\delta$ set, so a Polish, 0-dimensional space, invariant under the action $\boldsymbol{a}$. Moreover this action restricted to $Z$ is minimal. As in the proof of \cref{BaireActionRealization}, we can find a subspace $Y$ of $Z$ homeomorphic to $\baire$ invariant under the action, such that $F \uhr Z\cong_B F \uhr Y$. Thus $F \uhr Y$ is induced by a minimal action on the Baire space and is compressible, universal, therefore $E \cong_B F \uhr Y$.
\end{proof}

The following is an open problem:
\newcommand{\probcbercpt}{
    Does an arbitrary (not necessarily compressible) aperiodic,
    universal {\CBER} admit a compact action realization?
}
\begin{prob}\label{prob-cbercpt}
    \probcbercpt
\end{prob}

In \cref{subshifts-realizations} we will consider realizations of equivalence relations by subshifts
and in particular prove a considerable strengthening of \cref{3233}.

\section[Compressibility and paradoxicality]{Continuous actions on compact spaces, compressibility and paradoxicality}\label{realizations-compress}
\subsection{\nopunct}\label{realizations-compress-borel}
In connection with \cref{prob-realization}, \textit{for the case of compact action realizations},
we discuss here some special properties of continuous actions of countable groups
on compact Polish spaces that may have some relevance to this question.

Let $\Gamma$ be a countable group
and let $\ac a$ be a Borel action of $\Gamma$ on a standard Borel space $X$
(we are not assuming that $X$ is uncountable here).
Put $\gamma\cdot x = \ac a(\gamma, x)$.
We denote by $\ev{\ac a}$ the set of all Borel maps $T \colon X \to X$
such that $\forall x \exists \gamma \in \Gamma \;(T(x) = \gamma \cdot x)$.
Equivalently this means that there is a Borel partition
$X = \bigsqcup_{\gamma \in \Gamma} X_\gamma$ such that
$T(x) = \gamma\cdot x$ for $x\in X_\gamma$.
We also let $\ev{\ac a}^f$ consist of all Borel maps $T \colon X\to X$
for which there is a \textit{finite} subset $F\subseteq \Gamma$ such that
$\forall x \exists \gamma\in F \;(T(x) = \gamma\cdot x)$.
Equivalently this means that there is a Borel partition
$X = \bigsqcup_{\gamma\in F} X_\gamma$ such that
$T(x) = \gamma \cdot x$ for $x \in X_\gamma$.

We say that the action $\ac a$ is \textbf{compressible}
(resp., \textbf{finitely compressible})
if there is an injective Borel map in $T \in \ev{\ac a}$
(resp., $T \in \ev{\ac a}^f$) such that
for every orbit $C$ of $\ac a$,
$T(C) \subsetneq C$ or equivalently $\Gamma\cdot (X\setminus T(X)) = X$.
Clearly the action $\ac a$ is compressible
iff the associated equivalence relation is compressible.
The action $\ac a$ is \textbf{paradoxical}
(resp., \textbf{finitely paradoxical})
if there are two injective Borel maps $T_1, T_2$ in $\ev{\ac a}$
(resp., in  $\ev{\ac a}^f$) such that
$T_1(X) \cap T_2(X) = \emptyset$,
$T_1(X) \cup T_2(X) = X$.

Clearly if $\ac a$ is paradoxical (resp., finitely paradoxical),
then $\ac a$ is compressible (resp., finitely compressible).
It is also known that if $\ac a$ is compressible,
then $\ac a$ is paradoxical,
see \cite[Proposition 2.1]{DJK94}.

\begin{remark}\label{rmk-fincomp-notimp-finpara}
    It is easy to see that finite compressibility does not imply imply finite paradoxicality. Take for example $\Z$ acting on itself by translation. Since $\Z$ is amenable this action is not finitely paradoxical. On the other hand the map $T \colon \Z\to \Z$ such that $T(n) = n$, if $n<0$, and $T(n) = n+1$, if $n \ge 0$, shows that this action is finitely compressible.
\end{remark}
\begin{remark}
    One can easily see that finite paradoxicality is equivalent to the following strengthening of finite compressibility: There is an injective Borel map $T\in \ev{\ac a}^f$ and a finite subset $F\subseteq \Gamma$ such that $F\cdot (X\setminus T(X)) = X$. 
\end{remark}

For $n \ge 1$,
let $[n] = \{ 1,2, \dots , n\}$.
The {\bf $\boldsymbol{n}$-amplification} of $\ac a$ is the action $\ac a_n$
of the group $\Gamma \times S_n$ on $X \times [n]$ given by
$(\gamma, \pi) \cdot (x, i) = (\gamma \cdot x, \pi(i))$,
where $S_n$ is the group of permutations of $[n]$.
An \textbf{amplification} of $\ac a$ is an $n$-amplification of $\ac a$,
for some $n$.

\begin{thm}\label{thm-actionparadox}
    Let $\ac a$ be a continuous action of a countable group $\Gamma$ on a compact Polish space $X$.
    Then the following are equivalent:
    \begin{enumerate}[label=(\roman*)]
        \item \label{item-actionparadox-comp}
            $\ac a$ is compressible;
        \item \label{item-actionparadox-para}
            $\ac a$ is paradoxical;
        \item \label{item-actionparadox-fincomp}
            an amplification of $\ac a$ is finitely compressible;
        \item \label{item-actionparadox-finpara}
            an amplification of $\ac a$ is finitely paradoxical.
    \end{enumerate}
\end{thm}

The proof will be based on Nadkarni's Theorem and the following two results.
We first recall some standard terminology:

Let $X$ be a standard Borel space and let $B(X)$ be the $\sigma$-algebra of its Borel sets.
A \textbf{finitely additive Borel probability measure} is a map $\mu \colon B(X) \to [0,1]$
such that $\mu(\emptyset) = 0, \mu(X) = 1$ and  $\mu(A \cup B) = \mu(A) + \mu(B)$,
if $A \cap B = \emptyset$.
It is \textbf{countably additive} if moreover $\mu(\bigcup_n A_n) = \sum_n \mu (A_n)$,
for any pairwise disjoint family $(A_n)$.
Recall that we call these simply \textit{measures}.
If $\ac a$ is a Borel action of a countable group $\Gamma$ on $X$,
then $\mu$ is invariant if for any Borel set $A$ and $\gamma \in \Gamma$,
$\mu(\gamma \cdot A) = \mu(A)$.
    
\begin{thm}[{\cite[5.3]{Tse15}}]\label{fam}
    Let $\Gamma$ be a countable group and
    let $\ac a$ be a continuous action of $\Gamma$ on a compact Polish space $X$.
    If $\ac a$ admits an invariant finitely additive Borel probability measure,
    then it admits an invariant measure.
\end{thm}
\begin{remark}
    The hypothesis that $X$ is compact Polish is necessary here.
    We show below in \cref{notCompactRemark} that there is a counterexample to this statement
    even with $X$ Polish locally compact.
\end{remark}

\begin{thm}[{\cite[11.3]{TW16}}]\label{dlh}
    Let $\Gamma$ be a countable group and let $\ac a$ be a Borel action of $\Gamma$ on a standard Borel space $X$.
    Then the following are equivalent:
    \begin{enumerate}[label=(\roman*)]
        \item
            there is no invariant finitely additive Borel probability measure on $X$;
        \item
            there is a finitely paradoxical amplification of $\ac a$.
    \end{enumerate}
\end{thm}

We now prove \cref{thm-actionparadox}.
\begin{proof}[Proof of \Cref{thm-actionparadox}]
     We have already mentioned
     (in the paragraph preceding \cref{rmk-fincomp-notimp-finpara})
     the equivalence of \ref{item-actionparadox-comp} and \ref{item-actionparadox-para}).
    
    \ref{item-actionparadox-comp} $\implies$ \ref{item-actionparadox-finpara}:
    If $\ac a$ is compressible,
    then by Nadkarni's Theorem \Cref{nadkarni}
    it does not admit an invariant measure,
    so by \cref{fam} it does not admit an invariant finitely additive Borel probability measure.
    Then by \cref{dlh} some amplification of $\ac a$ is finitely paradoxical. 
    
    \ref{item-actionparadox-finpara} $\implies$ \ref{item-actionparadox-fincomp} is obvious.
    
    \ref{item-actionparadox-fincomp} $\implies$ \ref{item-actionparadox-comp}:
    Assume that for some $n$ the amplification $\ac a_n$ is finitely compressible but,
    towards a contradiction,
    $\ac a$ is not compressible.
    Then by Nadkarni's Theorem \Cref{nadkarni},
    $\ac a$ admits an invariant measure and thus so does $\ac a_n$,
    contradicting the compressibility of $\ac a_n$.
\end{proof}

\newcommand{\probactionparadox}{
    In \cref{thm-actionparadox},
    can one replace
    \ref{item-actionparadox-fincomp} by
    ``$\ac a$ is finitely compressible''
    and similarly for \ref{item-actionparadox-finpara}?
}
\begin{prob}\label{prob-actionparadox}
    \probactionparadox
\end{prob}

\begin{remark}
    It follows from \cref{thm-actionparadox} that for a continuous action $\ac a$ of a countable group on a compact Polish space, the property ``$\ac a$ has a finitely compressible (resp., finitely paradoxical) amplification'' is a property of the induced equivalence relation $E_{\ac a}$. More precisely, if $\ac a, {\ac b}$ are two continuous actions of groups $\Gamma, \Delta$ on compact metrizable spaces $X,Y$, resp.,  and $E_\ac a\cong_B E_{\ac b}$, i.e., $E_{\ac a}, E_{\ac b}$ are Borel isomorphic, then $\ac a$ admits a finitely compressible (reap., finitely paradoxical) amplification iff ${\ac b}$ admits a finitely compressible (reap., finitely paradoxical) amplification.
    In view of \cref{prob-actionparadox}, this may not be true for the property ``$\ac a$ is finitely compressible'' or ``$\ac a$ is finitely paradoxical''. In fact one way to try to decide \cref{prob-actionparadox} is to search for two continuous actions $\ac a,{\ac b}$  of countable groups $\Gamma, \Delta$ on a compact metrizable space $X$ with $E_{\ac a} = E_{\ac b}$, for which $\ac a$ is finitely compressible (or finitely paradoxical) but ${\ac b}$ is not.
\end{remark}

\begin{remark}\label{notCompactRemark}
    \cref{thm-actionparadox} fails if the space $X$ is not compact.
    In fact there are even counterexamples with $X$ Polish locally compact.
    Recall that an action of a group $\Gamma$ on a set $X$ is \textbf{amenable}
    if there is a finitely additive probability measure defined on all subsets of $X$ and invariant under the action.
    Any action of a countable amenable group is amenable.
    Take now $\Gamma$ to be a locally finite, infinite group and
    consider the left-translation action of $\Gamma$ on itself.
    This action is not finitely compressible.
    Let then $X = \cantor \times \Gamma$ (with $\Gamma$ discrete).
    This is Polish locally compact and $\Gamma$ acts on it continuously by the action $\ac a$
    given by $\gamma\cdot (x,\delta) = (x, \gamma\delta)$.
    This action is clearly compressible via the map $T(x, \gamma) = (x, f(\gamma))$,
    where $f \colon \Gamma \to \Gamma$ is an injection with $f(\Gamma) \neq \Gamma$,
    so \ref{item-actionparadox-comp} in \Cref{thm-actionparadox} holds.
    On the other hand,
    all amplifications $\ac a_n$ are amenable,
    so not finitely paradoxical and
    \ref{item-actionparadox-finpara} in \cref{thm-actionparadox} fails.
    Also all the actions $\ac a_n$ are not finitely compressible and
    \ref{item-actionparadox-fincomp} in \cref{thm-actionparadox} also fails. 
    
    In this counterexample the action $\ac a$ is smooth.
    One can find another counterexample where the action $\ac a$ is not smooth as follows:
    Let $\Gamma$ be as before and consider again the translation action of $\Gamma$ on itself.
    Let also $\Gamma$ act on $2^\Gamma$ by shift and consider the action $\ac b$ of
    $\Delta = \Gamma^2$ on $X = 2^\Gamma \times \Gamma$ given by
    $(\gamma, \delta) \cdot (x, \epsilon) = (\gamma \cdot x, \delta\epsilon)$.
    This action is not smooth and is compressible but it is also amenable,
    since the action of each factor of $\Delta$ is amenable on the corresponding space
    and therefore the action of $\Delta$ is amenable by taking the product of
    finitely additive probability measures witnessing the amenability of these two actions.
    (By the product of a finitely additive probability measure $\mu$ defined on all subsets of a set $A$
    and a finitely additive probability measure $\nu$ defined on all subsets of a set $B$,
    we mean the finitely additive probability measure $\mu \times \nu$ on $A \times B$
    defined by $\mu \times \nu(C) = \int_A \nu(C_x) \dd{\mu(x)}$.)
    Also all the actions $\ac a_n$ are not finitely compressible.
\end{remark}

\begin{prop}\label{prop-cberparadox}
    Let $E$ be a {\CBER} on a standard Borel space $X$.
    Then the following are equivalent:
    \begin{enumerate}[label=(\roman*)]
        \item \label{item-cberparadox-comp}
            $E$ is compressible.
        \item \label{item-cberparadox-para}
            $E$ is generated by a paradoxical Borel action.
        \item \label{item-cberparadox-fincomp}
            $E$ is generated by a finitely compressible Borel action.
        \item \label{item-cberparadox-finpara}
            $E$ is generated by a finitely paradoxical Borel action.
    \end{enumerate}
\end{prop}
\begin{proof}
    We have already seen that
    \ref{item-cberparadox-finpara}
    implies both
    \ref{item-cberparadox-fincomp}
    and 
    \ref{item-cberparadox-para},
    and that both of the latter imply
    \ref{item-cberparadox-comp}.
    So it remains to show that
    \ref{item-cberparadox-comp}
    implies
    \ref{item-cberparadox-finpara}.

    Note that being generated by a finitely paradoxical Borel action
    is closed upwards under $\subseteq_B$,
    since if $\ac a$ is a finitely paradoxical Borel action of some group $\Gamma$,
    then for every $F$ containing $E_{\ac a}$,
    if $\ac b$ is a Borel action of some group $\Delta$ generating $F$,
    then the free product action of $\Gamma * \Delta$
    is finitely paradoxical and generates $F$.
    
    So since $RI_\N$ is $\subseteq_B$-minimum for compressible CBERs,
    it suffices to show that $RI_\N$ is finitely paradoxical.
    Fix a transitive action of the free group on two generators $\F_2$ on $\N$
    which is finitely paradoxical.
    Then the induced action on $\R \times \N$
    is finitely paradoxical and generates $RI_\N$.
\end{proof}
 
\begin{remark}
    Ronnie Chen pointed out that
    \ref{item-actionparadox-finpara} $\implies$ \ref{item-actionparadox-para}
    in \cref{thm-actionparadox} can be also proved
    by using the cardinal algebra $K(E\times I_{\N})$ as in \cite{Che21}
    and the cancellation law for cardinal algebras.
\end{remark}

Recall also that a CBER $E$ admits an invariant measure iff \textit{some} Borel action of a countable group that generates $E$ has an invariant measure iff \textit{every} Borel action of a countable group that generates $E$ has an invariant measure (iff $E$ is not compressible). On the other hand, there are aperiodic CBER $E$ such that some Borel action of a countable group that generates $E$ has an invariant finitely additive Borel probability measure but some other Borel action of a countable group that generates $E$ has no invariant finitely additive Borel probability measure. For example, let $E= E_t$. There is a continuous action of $\F_2$ on $\cantor$ that generates $E$ (see the proof of \cref{3.9}) and this action has no invariant finitely additive Borel probability measure by \cref{fam}. On the other hand, $E_t$ is induced by a Borel action of $\Z$ and this action has in fact an invariant  finitely additive probability measure defined on all subsets of $\cantor$. 

However in view of \cref{prop-cberparadox}
we have the following equivalent formulation
of existence of invariant measures for a CBER:
\begin{prop}
    For every aperiodic {\CBER} $E$,
    $E$ admits an invariant measure iff every Borel action of a countable group that generates $E$
    admits an  invariant finitely additive Borel probability measure.
\end{prop}

\subsection{\nopunct}
The preceding results in \ref{realizations-compress-borel}
of \Cref{realizations-compress} can be generalized as follows.

Let $\Gamma$ be a countable group and let $\ac a$ be an action of $\Gamma$ on a set $X$.
Let also $\al$ be an algebra of subsets of $X$ invariant under this action.
For $A,B \in \al$, let $A \sim_{\al} B$ iff there are partitions $ A =\bigsqcup_{i=1}^n A_i, B =\bigsqcup_{i=1}^n B_i$,
where $A_i, B_i \in \al$, and $\gamma_i\in \Gamma$ such that $\gamma_i\cdot A_i = B_i$.
We say that the action is $\boldsymbol{\al}${\bf -finitely compressible}
if $X \sim_{\al}Y$, for some $Y\in \al$, with witnesses $X_i, Y_i, \gamma_i$ as above,
so that if $T \colon X\to X$ is such that $T(x) = \gamma_i \cdot x$, for $x\in X_i$,
then for every orbit $C$ of the action, $ T(C)\subsetneq C$.
Also the action is $\boldsymbol{\al}${\bf -finitely paradoxical}
if there is a partition $X=Y\sqcup Z$, with $Y, Z\in \al$ and $X\sim_{\al} Y\sim_{\al} Z$.
The concept of an invariant finitely additive probability measure on $\al$ is defined as before.

We extend the algebra $\al$ to an algebra ${\al}_{n}$ of subsets of $X\times [n]$ by letting $A\in { \al}_{n} \iff A =\bigcup^n_{i=1} A_i\times\{i\}$, where $A_i\in \al$. We say that $\ac a_n$ is $\al$-finitely compressible if it is ${\al}_{n}$-finitely compressible. Similarly we define what it means for $\ac a_n$ to be $\al$-finitely paradoxical.

We now have the following generalization of \cref{thm-actionparadox}:

\begin{thm}\label{thm-algparadox}
    Let $\ac a$ be a continuous action of a countable group $\Gamma$ on a compact Polish space $X$. Let $\al$ be an algebra of Borel subsets of $X$ which is invariant under the action and contains a basis for $X$.
    Then the following are equivalent:
    \begin{enumerate}[label=(\roman*)]
        \item
            there is no invariant finitely additive probability measure $\mu$ on $\al$;
        \item
            there is no invariant measure $\nu$;
        \item
            an amplification of $\ac a$ is $\al$-finitely compressible;
        \item
            an amplification of $\ac a$ is $\al$-finitely paradoxical.
    \end{enumerate}
\end{thm}
The proof of \cref{thm-algparadox} is similar to the proof of \cref{thm-actionparadox}
using the following generalizations of \cref{fam} and \cref{dlh}.

\begin{thm}[{\cite[5.3]{Tse15}}]\label{fam*}
    Let $\Gamma$ be a countable group and let $\ac a$ be  a continuous action of $\Gamma$ on a second countable Hausdorff space $X$. Let $\al$ be an algebra of subsets of $X$ which is invariant under the action and contains a basis for $X$ and a compact set $K$. If there is an invariant finitely additive probability measure $\mu$ on $\al$ with $\mu (K) >0$, then there is an invariant (Borel probability, countably additive) measure $\nu$.
\end{thm}

\begin{thm}[{\cite[11.3]{TW16}}]\label{dlh*}
    Let $\Gamma$ be a countable group and let $\ac a$ be an action of $\Gamma$ on a set $X$.
    Let $\al$ be an algebra of subsets of $X$ invariant under this action.
    Then the following are equivalent:
    \begin{enumerate}[label=(\roman*)]
        \item
            there is no invariant finitely additive probability measure on $\al$;
        \item
            there is a $\al$-finitely paradoxical amplification of $\ac a$.
    \end{enumerate}
\end{thm}

As a particular case of \cref{thm-algparadox} we have the following.
Let $\ac a$ be a continuous action of a countable group $\Gamma$
on a zero-dimensional compact Polish space $X$
(e.g., the Cantor space).
Let $\mathcal C$ be the algebra of clopen subsets of $X$.
Then the following are equivalent:

\begin{enumerate}[label=(\roman*)]
    \item
        there is no invariant finitely additive probability measure $\mu$ on $\mathcal C$;
    \item
        there is no invariant measure $\nu$;
    \item
        an amplification of $\ac a$ is $\mathcal C$-finitely compressible;
    \item
        an amplification of $\ac a$ is $\mathcal C$-finitely paradoxical;
    \item
        $\ac a$ is compressible;
    \item
        $\ac a$ is paradoxical;
    \item
        an amplification of $\ac a$ is finitely compressible;
    \item
        an amplification of $\ac a$ is finitely paradoxical.
\end{enumerate}

Thus,
rather surprisingly,
for a continuous action of a countable group on a zero-dimensional compact Polish space,
existence of a (countable Borel) paradoxical decomposition is equivalent to
the existence of an amplification with a finite paradoxical decomposition using Borel pieces
and also equivalent to the existence of an amplification
with a finite paradoxical decomposition using \textit{clopen} pieces.

\section{Turing and arithmetical equivalence}\label{realizations-turing}
Below let $\equiv_T$ denote {\bf Turing equivalence} and $\equiv_A$ {\bf arithmetical equivalence} on $\cantor$.

The following is an immediate consequence of \cref{3233},
since $\equiv_A$ is compressible and universal by \cite{MSS16}:
\begin{cor}
    Arithmetical equivalence $\equiv_A$ on $\cantor$ admits a compact action realization on $\cantor$.
\end{cor}

In fact, in \cref{3333}, we will see that it admits a realization which is a minimal subshift of $2^{\F_4}$.
On the other hand the following is open:

\newcommand{\probturingcompact}{
    Does Turing equivalence $\equiv_T$ on $\cantor$ admit a compact action realization?
}
\begin{prob}\label{prob-turingcompact}
    \probturingcompact
\end{prob}
A negative answer to this question would provide
a new proof of the failure of hyperfiniteness for $\equiv_T$ but,
much more importantly,
give a negative answer to the long-standing problem of the universality of $\equiv_T$,
see \cite{DK00}.

Concerning Turing equivalence, we know from \cref{BaireActionRealization} that it admits a continuous action realization on the Baire space $\baire$, i.e., that  there is a Borel isomorphism of $\cantor$ with $\baire$ which sends $\equiv_T$ to an equivalence relation induced by a continuous action of a countable group on $\baire$. We calculate below an upper bound for the  Baire class of such a Borel isomorphism. A version of the next theorem was first proved by Andrew Marks, in response to an inquiry of the authors, with ``Baire class 3" instead of ``Baire class 2''. The proof of \cref{t:main} below uses some of his ideas along with other additional arguments.

\begin{thm}
	\label{t:main}
	There exists a Baire class 2 bijection $\Phi \colon \cantor \to \baire$
	that is an isomorphism between $\equiv_T$
	and an equivalence relation given by a continuous group action on $\baire$.
\end{thm}

	The most natural construction of the isomorphism will yield \cref{p:main} below. We will show later that it in fact implies \cref{t:main}.

\begin{prop}
	\label{p:main}
	There exists a Baire class 2 map $\Psi$ that is an isomorphism between $\equiv_T$ on $\cantor$ and an equivalence relation given by a continuous group action on a 0-dimensional Polish space.
\end{prop}

\begin{proof}
	Let $\phi^i$ denote the partial function computed by the $i$th Turing machine, in some recursive enumeration of all the Turing machines. That is, we consider Turing machines with oracle and input tapes, and $\phi^i(x)=y$ iff for each $n$ the $i$th Turing machine with oracle $x$ and input $n$ halts with the output $y(n)$.
	
	We start with an easy observation. Below, for $s\in 2^{<\N}$, put $[s] = \{x\in \cantor : s\subseteq x\}$.
	\begin{lem}
		\label{l:easy}
		Assume that $x\equiv_T y$. There exists an $i$ with $\phi^i(x)=y$ and $\phi^i(y)=x$.
	\end{lem}
	\begin{proof}
		We can assume that $x \neq y$. Pick $j,k \in \mathbb{N}$ with $\phi^j(x)=y$ and $\phi^k(y)=x$, and $n$ with $\restriction{x}{n} \neq \restriction{y}{n}$. Then, an $i$ with $ \restriction{\phi^i}{[\restriction{x}{n}]}=\restriction{\phi^j}{[\restriction{x}{n}]}$ and $\restriction{\phi^i}{[\restriction{y}{n}]}=\restriction{\phi^k}{[\restriction{y}{n}]}$ clearly works.
	\end{proof}

	The idea is to define a coding function $\Psi=(\alpha,\beta,\gamma)$ that will serve as an isomorphism. The crucial property of $\Psi(x)$ is that $\alpha$ encodes for each $i$ whether $\phi^i$ is an involution on $x$ (and does this for every $y \equiv_T x$), $\beta$ will ensure that $\Psi$ is continuous, while $\gamma$ will be the identity, in order to encode $x$.

	Let us now give the precise definitions. Fix a function $\iota \colon \mathbb{N}^2 \to \mathbb{N}$ such that for each $i,j \in \mathbb{N}$ we have $\phi^{\iota(i,j)}=\phi^i \circ \phi^j$. 
	
	Let $\beta \colon \cantor \to (\N\cup\{*\})^{\N^3}$ be defined by $\beta(x)(i,j,m)=n$, if $n$ is least such that both the $i$th and the $j$th Turing machines with oracle $x$ and input $m$ halt with the same output in at most $n$ steps, and let $\beta(x)(i,j,m)=*$, if such an $n$ does not exist.
	
	Define a map $\alpha \colon \cantor \to \N^{\N^2}$ by letting \[\alpha(x)(i,j)=0 \iff \beta(x)(i,j) \in \baire,\]
	and
	\[\alpha(x)(i,j)=m+1 \iff \text{$m$ is least with }\beta(x)(i,j,m)=*.\]
	
	Finally, let $\Psi(x)=(\alpha(x),\beta(x),x)$. Let us denote the space $\N^{\N^2} \times (\N\cup\{*\})^{\N^3} \times \cantor$ by $X$, where $\N \cup \{*\}$ is endowed with the discrete topology.

	\begin{lem}
		$\Psi(\cantor)$ is closed in $X$.
	\end{lem}
	\begin{proof}
		Assume that $(\alpha(x_k),\beta(x_k),x_k)_k$ is a convergent sequence, and let $x$ be the limit of $(x_k)_k$. Take any $i,j,m \in \N$. It is clear from the definition of $\beta$ that $\beta(x)(i,j,m)=n$ holds for some $n \in \N$ if and only if $\beta(x_k)(i,j,m)=n$ is true for every large enough $k$. This shows that $\beta(x_k) \to \beta(x)$. Using this, it is easy to check that $\alpha(x_k) \to \alpha(x)$ holds as well.
	\end{proof}
     
	For $i \in \mathbb{N}$ define a map $\delta_{i}$ from $\Psi(\cantor)$ to itself as follows: 
	\[\delta_{i}(\Psi(x))=\Psi(x) \text{ if } \alpha(x)(\iota(i,i),0) \neq 0,\] otherwise \[\delta_{i}(\Psi(x))=\Psi(\phi^i(x)).\]

	\begin{lem}
		The maps $(\delta_{i})_{i \in \N}$ are $\Psi(\cantor) \to \Psi(\cantor)$ homeomorphisms. 
		
	\end{lem}
\begin{proof}
	Fix $i \in \mathbb{N}$. First we check that $\delta_i$ is continuous on the set $S=\{\Psi(x):\alpha(x)(\iota(i,i),0)=0\}$. For each $i',j',m \in \mathbb{N}$ we have that
	\[\delta_{i}(\alpha(x),\beta(x),x)(0)(i',j')=(\alpha(x))(\iota(i',i),\iota(j',i)),\]
	\[\delta_{i}(\alpha(x),\beta(x),x)(1)(i',j',m)=(\beta(x))(\iota(i',i),\iota(j',i),m),\]
    thus, $\delta_i$ selects and permutes some of the coordinates of $\Psi(x)$. Moreover, if $\Psi(x) \in S$, as $\Psi$ is injective, we have $\alpha(x)(\iota(i,i),0)=0$. Therefore, $x \in dom(\phi^i)$ and as $\phi^i$ is continuous on its domain, this shows the continuity of $\delta_i$. 

    Note that the set $S$
	is relatively clopen and on $\Psi(\cantor) \setminus S$ the function $\delta_i$ is the identity, showing that $\delta_i$ is continuous on the entire $\Psi(\cantor).$
 
	Finally, it follows from the definition of $\alpha$ and $\delta_i$ that $\delta_{i}(\delta_{i}(\Psi(x)))=\Psi(x)$ holds for each $x$: indeed, $\Psi^{-1}(S)$ is the collection of $x \in \cantor$ for which $\phi^i \circ \phi^i(x)=x$. Hence, $\delta_i$ is a continuous involution $\Psi(\cantor) \to \Psi(\cantor)$.
\end{proof}

Let $E_\Delta$ be the equivalence relation on $\Psi(\cantor)$ generated by the maps $\{\delta_i:i \in \N\}$. 

\begin{lem}
	$\Psi$ is an isomorphism between $\equiv_T$ and $E_\Delta$.
\end{lem}
\begin{proof}
	First, it is clear from the definition of $\delta_i$ that $\delta_i(\Psi(x))=\Psi(y)$ implies that $x \equiv_T y$. So $\Psi^{-1}$ is a homomorphism.

    Second, assume that $x \equiv_T y$. Then by Lemma \ref{l:easy} there exists an $i$ with $\phi^i(x)=y$ and $\phi^i(y)=x$. Then $\alpha(x)(\iota(i,i),0)=0$, so $\delta_i(\Psi(x))=\Psi(\phi^i(x))=\Psi(y)$, so $\Psi(x)E_\Delta\Psi(y)$.
\end{proof}

Now we turn to the calculation of the complexity of the map $\Psi$. 
	\begin{lem}
	\label{l:complexity}
	 The map $\beta$ is Baire class $1$ and the map
	$\alpha$ is Baire class $2$.
	Consequently, the map $\Psi$ is Baire class $2$.
\end{lem}
\begin{proof}
	For $\beta$, take any $i,j,m \in \N$. Then for each natural number $n$, the set $\{x:\beta(x)(i,j,m)=n\}$ is open. Thus, the set $\{x:\beta(x)(i,j,m)=*\}$ is closed. This shows that $\beta$ preimages of basic clopen sets are $\mathbf{\Delta}^0_2$. 
	
	For $\alpha$, for a given $i$, the set $\{x:\alpha(x)(i,j)=0\}=\{x:\forall m (\beta(x)(i,j,m) \in \N)\}$ is $\mathbf{\Pi}^0_2$, and also, for $m \neq 0$ we have that 
	$\{x:\alpha(x)(i,j)=m\}=\{x:\forall m' < m \ (\beta(x)(i,j,m') \in \N \mathrel\& \beta(x)(i,j,m)=*)\}$, which shows that these sets are $\mathbf{\Pi}^0_2$ as well, and thus $\alpha$ is indeed Baire class $2$. 
\end{proof}
	
	This completes the proof of \cref{p:main}
	\end{proof}
	In order to finish the proof of Theorem \ref{t:main} we need a last observation.
	
	\begin{lem}
		\label{l:bairespace} Assume that $\Gamma$ acts continuously on an uncountable and zero-dimensional Polish space $X$, so that the induced equivalence relation $E_\Gamma^X$ is aperiodic. Then there exist an invariant under the action set $X' \subseteq X$ that is homeomorphic to $\baire$, and an isomorphism $\varphi$ between $E_\Gamma^X$ and $\restriction{E_\Gamma^X}{X'}$ that moves only countably many points. 
	\end{lem}	
	\begin{proof}
	As in the proof of \cref{BaireActionRealization}.
	\end{proof}
	
\begin{proof}[Proof of Theorem \ref{t:main}]
    By Proposition \ref{p:main} there exists a Baire class $2$ isomorphism between $\equiv_T$ and some equivalence relation of the form $E_\Gamma^X$, where $\Gamma$ acts continuously on a zero-dimensional Polish space $X$. Applying Lemma \ref{l:bairespace} we get an isomorphism with an equivalence relation on the Baire space. Moreover, as countable modifications of Baire class $2$ functions do not change their class, we are done.
\end{proof}

We do not know if the complexity of the Borel isomorphism in \cref{t:main} is optimal.

\newcommand{\probturingbc}{
    Is there a Baire class $1$ bijection that is an isomorphism between
    $\equiv_T$ and an equivalence relation given by a continuous group action on $\baire$?
}
\begin{prob}\label{prob-turingbc}
    \probturingbc
\end{prob}

For the bijection $\Phi$ that was constructed in the proof of \cref{t:main},
it is easy to see that $\Phi (x) \equiv_T x''$,
since $x''$ can be easily computed from the map $\alpha$ defined in the proof of \cref{p:main}.
Similar to the problem \cref{prob-turingbc},
we have the following,
where $x \le_T y$ iff $x$ is recursive in $y$:

\newcommand{\probturingjump}{
    Is there a Borel bijection $\Phi \colon \cantor \to \baire$
    that is an isomorphism between $\equiv_T$ and an equivalence relation given by
    a continuous group action on $\baire$ such that $\Phi (x) \le_T x'$ on a cone?
}
\begin{prob}\label{prob-turingjump}
    \probturingjump
\end{prob}

On the other hand we have the following result:
\begin{prop}
	There is no Borel map $\Phi \colon \cantor \to \baire$
    that is an isomorphism between $\equiv_T$ and
    an equivalence relation given by a continuous group action on $\baire$ such that
    $\Phi(x) \le_T x$ on a cone.
\end{prop}
\begin{proof}
    Recall that a {\bf pointed perfect tree} is a perfect binary tree $S\subseteq 2^{<\N}$ such that $x \in [S] \implies S \le_T x$, where $[S] \subseteq \cantor$ is the set of infinite branches of $S$. Below we will use certain properties of pointed perfect trees due to Martin, whose proofs can be found, for example, in \cite{Kec88}. 
    
    Assume that $\Phi (x)  \le_T  x$ on a cone, towards a contradiction. Then by \cite[Theorem 1.3]{Kec88} there is a perfect pointed tree $T$ such that $x\in [T]\implies \Phi (x)  \le_T  x$. Then by \cite[Lemma 1.4]{Kec88}, there is a perfect pointed subtree $S\subseteq T$ and $i\in \N$ such that if $x\in [S]$, then $\varphi^i (x)$ is defined and $\varphi^i (x) = \Phi (x)$. Thus $\Phi$ is continuous on $[S]$. It follows that $\equiv_T $ restricted to $[S]$ is $\boldsymbol{\Sigma^0_2}$. Let now $\Psi$ be the canonical homeomorphism of $\cantor$ with $[S]$, so that $\Psi (x) \equiv_T x$, if $S \le_T x$. It follows that for $S \le_T x, y$, we have $x\equiv_T y \iff \Psi (x) \equiv_T \Psi (y)$, thus, in particular, for some $z\in \cantor$, the Turing degree of $z$, i.e., the set $\{w\in \cantor : w\equiv_T z\}$ is $\Sigma^0_2 (z)$. This is false in view of the following well-known fact:
    
    \begin{lem}
        For any $z \in \cantor$, the Turing degree of $z$ is in $\Sigma^0_3(z)$ but not in $\Pi^0_3(z)$.
    \end{lem}
    \begin{proof}
        It is easy to check that the Turing degree of $z$ is $\Sigma^0_3(z)$.
        Assume now that it is in $\Pi^0_3(z)$, towards a contradiction.
        Then if $A = \{w \in \cantor : w \le_T z\}$,
        we have that $A$ is also $\Pi^0_3(z)$,
        since $w \in A \iff \ev{w,z} \equiv_T z$.
        But then $\cantor \setminus A$ is a comeager $\Sigma^0_3(z)$ set,
        so by the relativized version of the basis theorem of Shoenfield \cite{Sho58},
        it contains a recursive in $z$ real, a contradiction.
    \end{proof}
\end{proof}

\chapter{Subshifts}\label{subshifts}

\section{Realizations by subshifts}\label{subshifts-realizations}

\textit{In this and the next two sections,
unless it is otherwise stated or clear from the context,
we assume that all groups are countable}.

Let $\Gamma$ be a group.

Given a countable $\Gamma$-set $L$ and a compact Polish space $X$,
denote the $\Gamma$-flow $X^L$ by $\ac s_{L, X}$.
In particular,
$\ac s_{\Gamma, X}$ is the shift action of $\Gamma$,
where $\Gamma$ acts on itself by left multiplication.

For $\Gamma$-flows $\ac a$ and $\ac b$ on $X$ and $Y$ respectively,
a $\boldsymbol\Gamma$-{\bf map} $\ac a\to \ac b$ is
a $\Gamma$-equivariant continuous function $X\to Y$.
Let $\Hom_\Gamma(\ac a, \ac b)$ denote the set of
$\Gamma$-maps $\ac a\to \ac b$.

Below for any action $\ac a$, we denote by $X_{\ac a}$ the space of $\ac a$ and by $E_{\ac a}$ the induced orbit equivalence relation.

In \ref{subshifts-realizations-coind}
and \ref{subshifts-realizations-jumps}
of \Cref{subshifts-realizations},
we will develop some properties of flows that we will use in \cref{3333}
to prove a strong realization result for compressible,
universal CBER in terms of subshifts.

\subsection{Coinduction and generators}\label{subshifts-realizations-coind}
Let $\Gamma\le\Delta$ be groups.
Given a $\Delta$-flow $\ac b$,
we denote the $\Gamma$-restriction of $\ac b$ by $\ac b\uhr_\Gamma$.

Given a $\Gamma$-flow $\ac a$ on $X$,
the \textbf{coinduced $\Delta$-flow} of $\ac a$,
denoted by $\CInd_\Gamma^\Delta(\ac a)$,
is the $\Delta$-subflow of $\ac s_{\Delta, X}$ on the subspace
\[
    \{ x\in X^\Delta :
    \forall \gamma\in\Gamma\,
    \forall \delta\in\Delta\,
    [x_{\delta\gamma} = \gamma^{-1}\cdot x_\delta]
    \}.
\]
In particular,
$\ac s_{\Delta, X}$ is isomorphic to $\CInd_1^\Delta(\ac s_{1,X})$,
where $1$ is the trivial group
(note that $\ac s_{1, X}$ is the $1$-flow on $X$).

There is a natural bijection
\[
    \Hom_\Delta(\ac b, \CInd_\Gamma^\Delta(\ac a))
    \cong \Hom_\Gamma(\ac b\uhr_\Gamma, \ac a)
\]
taking $f$ to the map $y\mapsto (f(y))_1$.

Let $\ac a$ and $\ac b$ be flows of $\Gamma$ and $\Delta$ on $X_{\ac a}$, $X_{\ac b}$, respectively.
A $\Gamma$-map $f \colon \ac b\uhr_\Gamma\to \ac a$
is an \textbf{$\ac a$-generator} for $\ac b$
if its corresponding $\Delta$-map
$\ac b \to \CInd_\Gamma^\Delta(\ac a)$ is injective.
Explicitly,
$f$ is an $\ac a$-generator for $\ac b$
if for every $x, x'\in X_{\ac a}$,
if $f(\delta\cdot x) = f(\delta\cdot x')$
for every $\delta$,
then $x = x'$.

We note the following facts:

\begin{enumerate}[label=(\arabic*)]
    \item \label{item-clopengenerator}
        Let $\ac a$ be a $\Gamma$-flow on $X$,
        and let $n \ge 2$.
        Considering $n$ as a discrete space,
        an $\ac s_{1, n}$-generator for $\ac a$
        coincides with the usual notion of a
        \textbf{clopen $n$-generator} for $\ac a$,
        that is,
        a partition $(A_i)_{i < n}$ of $X$ into clopen sets
        such that for every $x, x'\in X$,
        if for every $\gamma\in\Gamma$ and every $i < n$
        we have
        \[
            \gamma\cdot x \in A_i
            \iff \gamma\cdot x' \in A_i,
        \]
        then $x = x'$.
        Equivalently $\ac a$ admits a clopen $n$-generator
        iff it is (topologically) isomorphic to a subshift of $n^\Gamma$.
    \item
        Every injective $\Gamma$-map $\ac b\uhr_\Gamma\hra \ac a$
        is an $\ac a$-generator for $\ac b$.
    \item
        If $\ac b = \CInd_\Gamma^\Delta(\ac a)$,
        then the map $y\mapsto y_1$ is an $\ac a$-generator for $\ac b$,
        since it corresponds to the identity on $\CInd_\Gamma^\Delta(\ac a)$.
    \item
        Let $\Gamma\le\Delta\le\Lambda$ be groups
        with flows $\ac a$, $\ac b$ and $\ac c$ respectively.
        If $\ac c$ has a $\ac b$-generator $f$,
        and $\ac b$ has an $\ac a$-generator $g$,
        then the composition $f\circ g$
        is an $\ac a$-generator for $\ac c$.
        To see this,
        let $x, x'\in X_{\ac a}$
        and suppose that
        $f(g(\lambda\cdot x)) = f(g(\lambda\cdot x'))$
        for every $\lambda\in \Lambda$.
        Then for every $\delta\in \Delta$ and every $\lambda\in \Lambda$,
        we have $f(\delta\cdot g(\lambda\cdot x))
        = f(\delta\cdot g(\lambda\cdot x'))$.
        Thus since $f$ is a generator,
        we have $g(\lambda\cdot x) = g(\lambda\cdot x')$
        for every $\lambda \in \Lambda$.
        Since $g$ is a generator,
        we have $x = x'$.
\end{enumerate}

Below we call a flow {\bf compressible} iff the induced equivalence relation is compressible.
Equivalently by Nadkarni's Theorem \Cref{nadkarni},
this means that the flow admits no invariant Borel probability measure.

\begin{prop}\label{liftProps}
    Let $\Gamma\le\Delta$ be groups,
    let $\ac a$ be a $\Gamma$-flow,
    and let $\ac b$ be a $\Delta$-flow.
    \begin{enumerate}[label=(\arabic*)]
        \item
            Suppose there is a $\Gamma$-map $\ac b\uhr_\Gamma\to \ac a$.
            If $\ac a$ is compressible,
            then $\ac b$ is compressible.
        \item
            Suppose $\ac b$ has an $\ac a$-generator.
            If $\ac a$ has a clopen $n$-generator,
            then $\ac b$ has a clopen $n$-generator.
        \item
            Suppose $\ac b\uhr_\Gamma = \ac a$.
            If $\ac a$ is minimal,
            then $\ac b$ is minimal.
    \end{enumerate}
\end{prop}
\begin{proof}
    \leavevmode
    \begin{enumerate}[label=(\arabic*)]
        \item 
            If $\mu$ is an invariant Borel probability measure for $\ac b$,
            then it is invariant for $\ac b\uhr_\Gamma$,
            so its pushforward to $\ac a$ is invariant.
        \item
            If $\ac a$ has a clopen $n$-generator,
            then it has an $\ac s_{1, n}$-generator.
            Composing them gives a clopen $n$-generator for $\ac b$.
        \item
            This is obvious.
    \end{enumerate}
\end{proof}

\begin{cor}\label{cindProps}
    Let $\Gamma\le\Delta$ be groups.
    The following properties of a $\Gamma$-flow
    pass to its coinduced $\Delta$-flow:
    \begin{enumerate}[label=(\roman*)]
        \item Compressibility.
        \item Existence of a clopen $n$-generator.
    \end{enumerate}
\end{cor}
\begin{proof}
    Take $\ac b = \CInd_\Gamma^\Delta(\ac a)$ in \cref{liftProps}.
\end{proof}

\subsection{Jumps}\label{subshifts-realizations-jumps}
Let $\Gamma$ and $\Lambda$ be groups,
and let $L$ be a countable $\Lambda$-set.
The \textbf{unrestricted wreath product}
is the group $\Gamma\wr_L\Lambda$
defined by
\[
    \Gamma\wr_L \Lambda
    := \Gamma^L\rtimes\Lambda,
\]
where $\Lambda$ acts on the product group
$\Gamma^L$ by shift.
If $L = \Lambda$ with the left-translation action,
then we omit the subscript and write $\Gamma\wr\Lambda$.
Denote by $\Gamma^{\oplus L}$
the subgroup of $\Gamma^L$
consisting of those elements
which are the identity on cofinitely many coordinates.
Note that the shift $\Lambda$-action on $\Gamma^L$
preserves $\Gamma^{\oplus L}$.
The \textbf{restricted wreath product}
is the subgroup $\Gamma\wr_L^\oplus\Lambda$ of $\Gamma\wr_L\Lambda$
generated by $\Gamma^{\oplus L}$ and $\Lambda$.
If $L$ is a transitive $\Lambda$-set,
and $S$ and $T$ are generating sets for $\Gamma$ and $\Lambda$ respectively,
and $l$ is any element of $L$,
then the set $\{(\tilde s, 1) : s \in S\} \cup \{(\ol 1, t) : t \in T\}$
generates $\Gamma\wr_L^{\oplus}\Lambda$
(see \cite[2.3]{HR94}),
where $\tilde s$ is $s$ at $l$ and $1$ otherwise.

Let $E$ be a CBER on $X$.
The \textbf{unrestricted $L$-jump} of $E$,
denoted $E^{[L]}$
is the Borel equivalence relation on $X^L$ defined by
\[
    x \mr{E^{[L]}} y
    \iff \exists\lambda\,
    [\lambda \cdot x \mr{E^L} y],
\]
where $E^L$ is the product equivalence relation on $X^L$ defined by
\[ 
(x_l ) E^L (y_l) \iff \forall l (x_l E y_l)
\]
(see \cite{CC22} for more uses of this jump).
Let $E^{\oplus L}$ be the subequivalence relation
of the product equivalence relation $E^L$
consisting of pairs which are equal on cofinitely many coordinates.
The \textbf{restricted $L$-jump} of $E$,
is the subequivalence relation $E^{\oplus[L]}$ of $E^{[L]}$
obtained by using $E^{\oplus L}$ instead of $E^L$
in the definition of $E^{[L]}$.

Given a Borel embedding $E\sq_B F$ via a map $X\to Y$,
the induced map $X^L \to Y^L$ witnesses the Borel embeddings
$E^{[L]}\sq_B F^{[L]}$ and $E^{\oplus[L]}\sq_B F^{\oplus[L]}$.

Let $\ac a$ be a $\Gamma$-flow on $X$.
Let $\ac a^L$ be the action of $\Gamma^\Lambda$ on $X^L$
defined by
$( \gamma\cdot x)_l
= \gamma_l\cdot x_l$.
We have $E_{\ac a^L} = (E_{\ac a})^L$.
The \textbf{unrestricted $L$-jump} of $\ac a$,
denoted $\ac a^{[L]}$,
is the unique $\Gamma\wr_L \Lambda$-flow on $X^L$
which simultaneously extends both
$\ac a^L$
and the $\Lambda$-flow $\ac s_{L, X}$ on $X^L$.
Explicitly,
the action is given by
$(\bs\gamma\lambda\cdot \bs x)_{l}
= \gamma_l \cdot x_{\lambda^{-1}l}$.
We have
$E_{\ac a^{[L]}}
= (E_{\ac a})^{[L]}$,
since
\[
    x \mr{E_{\ac a^{[L]}}} y
    \iff \exists\bs\gamma\, \exists\lambda\,
    [\bs\gamma\lambda \cdot x = y]
    \iff \exists\lambda\,
    [\lambda \cdot x \mr{E_{\ac a^L}} y]
    \iff \exists\lambda\,
    [\lambda \cdot x \mr{E_{\ac a}^L} y].
\]
Let $\ac a^{\oplus L}$ be the $\Gamma^{\oplus L}$-flow
$\ac a^L\uhr_{\Gamma^{\oplus L}}$.
We have $E_{\ac a^{\oplus L}} = E_{\ac a}^{\oplus L}$.
The \textbf{restricted $L$-jump} of $\ac a$,
denoted $\ac a^{\oplus[L]}$,
is the restriction $\ac a^{[L]}\uhr_{\Gamma\wr_L^\oplus\Lambda}$.
We have $E_{\ac a^{\oplus[L]}} = (E_{\ac a})^{\oplus[L]}$.

If $L$ is a transitive $\Lambda$-set,
then for any $l_0\in L$,
the map $x \mapsto x_{l_0}$
is an $\ac a$-generator for $\ac a^{\oplus[L]}$,
since if $(\lambda\cdot x)_{l_0} = (\lambda\cdot y)_{l_0}$
for every $\lambda$,
then by transitivity,
we have $x_l = y_l$ for every $l\in L$,
and thus $x = y$.

\begin{prop}\label{jumpProps}
    If $L$ is a transitive $\Lambda$-set,
    then the following properties of a $\Gamma$-flow
    pass to its restricted $L$-jump:
    \begin{enumerate}[label=(\roman*)]
        \item Compressibility.
        \item Existence of a  clopen $n$-generator.
        \item Minimality.
    \end{enumerate}
\end{prop}
\begin{proof}
    Let $\ac a$ be a $\Gamma$-flow.
    
    Since $\ac a^{\oplus[L]}$ has an $\ac a$-generator,
    the first two properties follow from \cref{liftProps}.
    
    If $\ac a$ is a minimal $\Gamma$-flow,
    then $\ac a^{\oplus L}$ is a minimal $\Gamma^{\oplus L}$-flow.
    Since $\ac a^{\oplus[L]}\uhr_{\Gamma^{\oplus L}}
    = \ac a^{\oplus L}$,
    we have by \cref{liftProps} that $\ac a^{\oplus[L]}$ is minimal.
\end{proof}

\subsection{Realizations by minimal subshifts}
A flow is \textbf{orbit-universal}
if its orbit equivalence relation is a universal CBER.

Let $\Gamma$ and $\Lambda$ be groups,
and let $L$ be a countable $\Lambda$-set.

Let $E(L, \R)$ denote the orbit equivalence relation
of the shift action $\Lambda\car \R^L$.
\begin{thm}\label{jumpEmbedding}
    Let $\Gamma$ and $\Lambda$ be countable groups.
    Let $L$ be a countable $\Lambda$-set,
    and let $\ac a$ be a $\Gamma$-flow on $X$, with $X$ uncountable.
    Then there is a Borel injection $f \colon \R^L \to X^L$
    which simultaneously witnesses
    $E(L, \R) \sq_B E_{\ac a^{\oplus[L]}}$
    and $E(L, \R)\sq_B E_{\ac a^{[L]}}$.
    In particular,
    for every group $G$
    (no definability condition required)
    with
    $\Gamma\wr_L^\oplus \Lambda
    \le G
    \le \Gamma\wr_L \Lambda$,
    the map $f$ witnesses $E(L, \R)\sq_B E_{\ac a^{[L]}\uhr_G}$.
    
    In particular,
    if $E(L, \R)$ is universal and $\Delta$ is a countable group
    with
    $\Gamma\wr_L^\oplus \Lambda
    \le \Delta
    \le \Gamma\wr_L \Lambda$,
    then $\ac a^{[L]}\uhr_\Delta$ is orbit-universal.
\end{thm}
\begin{proof}
    Since $X$ is uncountable,
    there is a Borel map $\R\to X$ witnessing $\id_\R\sq_B E_{\ac a}$.
    Let $f \colon \R^L \to X^L$ be the induced map.
    Then $f$ witnesses
    \[
        E(L, \R)
        = (\id_\R)^{\oplus[L]}
        \sq_B (E_{\ac a})^{\oplus[L]}
        = E_{\ac a^{\oplus[L]}}
    \]
    and also
    \[
        E(L, \R)
        = (\id_\R)^{[L]}
        \sq_B (E_{\ac a})^{[L]}
        = E_{\ac a^{[L]}}
    \]
\end{proof}
\begin{cor}\label{univFlow}
    Let $\Gamma$ and $\Lambda$ be groups,
    and suppose that
    there is a countable transitive $\Lambda$-set $L$
    such that $E(L, \R)$ is a universal {\CBER}.
    Let $\Delta$ be a countable group with a factor $\Delta'$
    such that
    $\Gamma\wr_L^\oplus\Lambda
    \le \Delta'
    \le \Gamma\wr_L\Lambda$.
    Then there is an orbit-universal minimal $\Delta$-flow
    with a clopen $2$-generator.
    If $\Gamma$ is non-amenable,
    then this flow can be taken to be compressible.
\end{cor}
\begin{proof}
    It suffices to consider the case where $\Delta' = \Delta$.
    Let $\ac a$ be an uncountable minimal $\Gamma$-flow
    with a clopen $2$-generator;
    for instance,
    take a minimal subshift of a free subshift of $2^\Gamma$
    (these exist by \cite[Theorem 1.5]{GJS09};
    see also \cite{Ber17}).
    If $\Gamma$ is non-amenable,
    by \cref{compressibleSubshift} below
    we can take $\ac a$ to be compressible
    and then pass to a subflow to ensure minimality.
    
    Now consider the $\Gamma\wr_L^\oplus\Lambda$-flow
    $\ac a^{\oplus[\Lambda]}$.
    By \cref{jumpProps},
    this is minimal and has a clopen $2$-generator,
    and is compressible if $\Gamma$ is non-amenable.
    Orbit-universality follows from \cref{jumpEmbedding}.
\end{proof}

\begin{cor}\label{3333}
    \leavevmode
    \begin{enumerate}[label=(\arabic*)]
        \item
            There is an orbit-universal minimal subshift of $2^{\F_3}$.
        \item
            There is a compressible orbit-universal minimal subshift of $2^{\F_4}$.
            In particular any compressible, universal {\CBER} admits
            a minimal, compact action realization
            which is in fact a minimal subshift of $2^{\F_4}$.
    \end{enumerate}
\end{cor}
\begin{proof}
    Recall that $E(\F_2, 2)$ is a universal CBER.
    \begin{enumerate}[label=(\arabic*)]
        \item
            $\F_3$ has the factor $\Z\wr^\oplus \F_2$,
            so by \cref{univFlow},
            $\Z\wr^\oplus\F_2$ admits an orbit-universal minimal flow
            with a clopen $2$-generator.
        \item
            $\F_4$ has the factor $\F_2\wr^\oplus \F_2$,
            so since $\F_2$ is non-amenable,
            by \cref{univFlow},
            $\F_2\wr^\oplus\F_2$ admits a compressible orbit-universal
            minimal flow with a clopen $2$-generator,
    \end{enumerate}
\end{proof}

\newcommand{\probtwogenuniv}{
    Does \cref{3333} hold with $\F_2$ instead of $\F_3,\F_4$?
}
\begin{prob}\label{prob-twogenuniv}
    \probtwogenuniv
\end{prob}
By \cref{univFlow},
it suffices to find some $\Gamma$
and $\Lambda\ge \F_2$
such that there is a $2$-generated group
between $\Gamma \wr_L^\oplus\Lambda$
and $\Gamma \wr_L\Lambda$
(and $\Gamma$ non-amenable
if we want compressibility).

In view of \cref{3.9} and \cref{3333}
one can ask whether the following very strong realization result is true:
\newcommand{\probcbershift}{
    Does every non-smooth aperiodic {\CBER} have a realization
    as a subshift of $2^\Gamma$ for some group $\Gamma$?
    Also does it have a realization as a minimal subshift? 
}
\begin{prob}\label{prob-cbershift}
    \probcbershift
\end{prob}

\subsection{Minimal subshift universality}
Let $\Gamma$ be a countable group. We say that $\Gamma$ is {\bf minimal subshift universal} if there is a minimal subshift $K$ of $2^\Gamma$ such that if $E(\Gamma , 2)$ is the shift equivalence relation on $2^\Gamma$, then $E(\Gamma, 2) \uhr K$ is universal.  We note the following equivalent formulation of this notion. Recall that a point $x\in 2^\Gamma$ is minimal
if $\ol{\Gamma\cdot x}$ is a minimal $\Gamma$-flow;
equivalently,
for every finite $A\subseteq \Gamma$,
the set $\{\gamma\in \Gamma : (\gamma\cdot x)\uhr_A = x\uhr_A\}$
is \textbf{left syndetic},
i.e. finitely many left translates of it cover $\Gamma$
(see \cite[IV(1.2)]{Vri93}).

\begin{prop}
    Let $\Gamma$ be a countable group. Then the following are equivalent:
    \begin{enumerate}[label=(\roman*)]
        \item \label{item-minuniv-minuniv}
            $\Gamma$ is minimal subshift universal;
        \item \label{item-minuniv-flow}
            There is a minimal $\Gamma$-flow
            which admits a clopen 2-generator and
            such that the induced equivalence relation is universal;
        \item \label{item-minuniv-bernoulli}
            If $M$ is the set of minimal points in $2^\Gamma$
            and $E$ is the shift equivalence relation,
            then $E \uhr M$ is universal.
    \end{enumerate}
\end{prop}

\begin{proof}
    Clearly \ref{item-minuniv-minuniv} $\implies$ \ref{item-minuniv-flow} $\implies$ \ref{item-minuniv-bernoulli}.
    Assume now \ref{item-minuniv-bernoulli}).
    Consider the Borel map $f$ that sends $x \in M$ to the closure of its orbit
    (which is an element of the space of compact subsets of $2^\Gamma$).
    Then by \cite[Theorem 3.1]{MSS16},
    there is some $K$ such that $E \uhr f^{-1}(K)$ is universal.
    But clearly $f^{-1}(K) =K$,
    so $K$ is a minimal subshift,
    thus \ref{item-minuniv-minuniv} holds.
\end{proof}

Clearly if $\Gamma$ is minimal subshift universal
and there is a surjective homomorphism of $\Delta$ on $\Gamma$,
then $\Delta$ is minimal subshift universal.
The existence of a minimal subshift universal group was first proved by Brandon Seward,
who showed that $\F_\infty$ has this property.
\cref{univFlow} shows that any wreath product $\Gamma\wr\Lambda$,
where $\Gamma$ is infinite and $\Lambda$ contains $\F_2$,
is minimal subshift universal and in particular by \cref{3333},
$\F_3$ is minimal subshift universal.
We include below Seward's proof for $\F_\infty$ (and somewhat more),
with his permission,
as it is based on a very different method. 

\begin{thm}[Seward]
    Let $\Gamma$ be a group.
    Then $E(\Gamma, 2)\sq_B E(\Gamma*\F_\infty, 2)\uhr_M$,
    where $M \subseteq 2^{\Gamma*\F_\infty}$
    denotes the set of minimal points.
\end{thm}

\begin{proof}
We start with the following lemma.
\begin{lem}
    Let $\Gamma$ be a group
    and let $A\subseteq \Gamma$ be a finite subset.
    There is a $\Gamma$-equivariant Borel embedding $x\mapsto x'$
    from $E(\Gamma, 2)$ to $E(\Gamma*\Z, 2)$ such that for every $x\in 2^\Gamma$,
    \begin{enumerate}[label=(\roman*)]
        \item
            $x'\uhr_\Gamma = x$ (i.e. $x'$ extends $x$),
        \item
            $((\Gamma*\Z)\cdot x')\uhr_\Gamma \subseteq \Gamma\cdot x$,
        \item
            the $\Z$-action on $(\Gamma*\Z)\cdot x'$ factors
            via the restriction map $2^{\Gamma*\Z}\to 2^A$
            to a transitive action
            (on the image).
    \end{enumerate}
\end{lem}
\begin{proof}
    Fix an enumeration of $\Gamma$,
    and let $t$ denote the generator of $\Z$.
    For every nonempty subset $P$ of $2^A$,
    fix a transitive permutation $\sigma_P$ of $P$.
    Let $x\in 2^\Gamma$,
    and let $P_x = (\Gamma\cdot x)\uhr_A$
    be the set of $A$-patterns appearing in $x$.
    We define $x'$ inductively on left cosets of $\Gamma$,
    starting with $x'\uhr_\Gamma = x$.
    
    Let $tw$ be a reduced word in ${\Gamma*\Z}$
    for which $(w\cdot x')\uhr_\Gamma$ is already defined.
    Then set $(tw\cdot x')\uhr_\Gamma = \gamma\cdot((w\cdot x')\uhr_\Gamma)$,
    where $\gamma$ is minimal with
    $(\gamma\cdot ((w\cdot x')\uhr_\Gamma))\uhr_A
    = \sigma_{P_x}((w\cdot x')\uhr_A)$.
    
    Similarly,
    if $t^{-1}w$ is a reduced word for which
    $(w\cdot x')\uhr_\Gamma$ is defined,
    then set $(t^{-1}w\cdot x')\uhr_\Gamma
    = \gamma^{-1}\cdot((w\cdot x')\uhr_\Gamma)$,
    where $\gamma$ is minimal with
    $(\gamma^{-1}\cdot ((w\cdot x')\uhr_\Gamma))\uhr_A
    = \sigma_{P_x}^{-1}((w\cdot x')\uhr_A)$.
\end{proof}

    Let now $t_0, t_1, t_2,\ldots$ be the free generators of $\F_\infty$,
    and let $\Gamma_n = \Gamma*\ev{t_i}_{i < n}$ (this includes the case $n=\infty$).
    Let $(A_n)_n$ be an exhaustive increasing sequence
    of finite subsets of $\Gamma_\infty$ such that $A_n\subseteq \Gamma_n$.
    For every $n$,
    apply the lemma with $\Gamma_n$ and $A_n$
    to obtain a Borel embedding $E(\Gamma_n, 2)\sq_B E(\Gamma_{n+1}, 2)$.
    Given $x_0\in 2^{\Gamma_0}$,
    let $x_1$ denote the extension to $2^{\Gamma_1}$ of $x_0$,
    let $x_2$ denote the extension to $2^{\Gamma_2}$ of $x_1$,
    and so on for $x_n \in 2^{\Gamma_n}$.
    Let $x_\infty = \bigcup_n x_n$.
    We claim that for every $n$ and every $m > n$
    (including $m = \infty$),
    the $\ev{t_n}$-action on $\Gamma_m \cdot x_m$
    factors via the restriction map $2^{\Gamma_m}\to 2^{A_n}$
    to a transitive $\ev{t_n}$-action on the image.
    It suffices to show this for every finite $m$.
    We proceed by induction on $m$,
    for which the base case $m = n + 1$ holds by the lemma.
    Now suppose that this holds for $m$.
    Let $\gamma \in \Gamma_{m+1}$.
    Then by the lemma,
    $(\gamma\cdot x_{m+1})\uhr_{\Gamma_m} = h\cdot x_m$ for some $h \in \Gamma_m$,
    and thus
    \[
        (t_m \gamma\cdot x_{m+1})\uhr_{\Gamma_m}
        = t_m\cdot ((\gamma\cdot x_{m+1})\uhr_{\Gamma_m})
        = t_m\cdot (h\cdot x_m)
        = t_m h\cdot x_m
    \]
    so
    \[
        (t_m \gamma\cdot x_{m+1})\uhr_{A_n}
        = (t_m \gamma\cdot x_{m+1}\uhr_{\Gamma_m})\uhr_{A_n}
        = (t_m h\cdot x_m)\uhr_{A_n}
    \]
    which only depends on $(h\cdot x_m)\uhr_{A_n} = (\gamma\cdot x_{m+1})\uhr_{A_n}$,
    so the $\ev{t_n}$-action factors through $2^{\Gamma_{m+1}}\to 2^{A_n}$,
    and the action is clearly still transitive.
    
    We show that the map $x_0\mapsto x_\infty$ is the desired embedding.
    It is clearly a $\Gamma_0$-invariant Borel injection.
    To see that it is a cohomomorphism,
    if $(x_\infty, y_\infty) \in E(\Gamma_\infty, 2)$,
    then $x_\infty = \gamma\cdot y_\infty$ for some $\gamma \in \Gamma_\infty$.
    Now $\gamma\in \Gamma_n$ for some $n$,
    so $x_n = \gamma\cdot y_n$,
    and thus $(x_n, y_n) \in E(\Gamma_n, 2)$.
    Since each extension map is a cohomomorphism,
    we have $(x_0, y_0) \in E(\Gamma_0, 2)$.
    
    It remains to show that the image lies in $M$.
    Fix $x_\infty$ and let $A\subseteq \Gamma_\infty$.
    We show that the set
    $\{\gamma\in \Gamma_\infty:(\gamma\cdot x_\infty)\uhr_A = x_\infty\uhr_A\}$
    is left syndetic.
    By enlarging $A$,
    we can assume that $A = A_n$ for some $n$.
    Let $T = \{(t_n)^k : 0\le k < 2^{|A_n|}\}$.
    Now let $\gamma\in \Gamma_\infty$.
    Then by transitivity,
    there is some $0 \le k < 2^{|A_n|}$ for which
    $((t_n)^k \gamma\cdot x)\uhr_{A_n} = x_\infty \uhr_{A_n}$,
    so we are done.
\end{proof}
  
We can now restate \cref{prob-twogenuniv}, in a more general form, as follows:
\newcommand{\probminuniv}{
    Is $\F_2$ minimal subshift universal?
    More generally,
    is every group that contains $\F_2$
    minimal subshift universal?
}
\begin{prob}\label{prob-minuniv}
    \probminuniv
\end{prob}

\section{Subshifts as tests for amenability}
It is well known that a group $\Gamma$ is amenable iff every continuous action of $\Gamma$ on a compact space admits an invariant Borel probability measure. Call a class $\mathcal F$ of such actions a {\bf test for amenability} for $\Gamma$ if $\Gamma$ is amenable provided that  every action in $\mathcal F$ admits an invariant Borel probability measure. In \cite{GH97} a compact metrizable space $X$ is called a {\bf test space} for the amenability of $\Gamma$ if the class of all $\Gamma$-flows on $X$ is a test for amenability for $\Gamma$. 
Giordano and de la Harpe show in \cite{GH97} that the Cantor space $\cantor$ is a test space for amenability of any group. Equivalently this says that the class of all subshifts of $(\cantor)^\Gamma$ is a test of amenability for $\Gamma$. We show next that the strongest result along these lines is actually true, namely that the class of all subshifts of $2^\Gamma$ is a test of amenability for $\Gamma$. This gives another characterization of amenability.

\begin{thm}\label{compressibleSubshift}
    Let $\Gamma$ be a group. Then $\Gamma$ is amenable iff every subshift of $2^\Gamma$ admits an invariant Borel probability measure.
\end{thm}
\begin{proof}
We have to show that if $\Gamma$ is not amenable then there is a compressible subshift of $2^\Gamma$. We will first give a proof for the case that $\Gamma$ contains $\F_2$, which is much simpler, and then give the full proof for arbitrary non-amenable $\Gamma$.

\textit{Proof when $\Gamma\ge \F_2$.}

    It suffices to find a compressible $\F_2$-flow with a clopen $2$-generator,
    since if $\ac a$ is such an $\F_2$-flow,
    then $\CInd_{\F_2}^\Gamma(\ac a)$ is a compressible $\Gamma$-flow
    with a clopen $2$-generator by \cref{cindProps}. For the existence of such a $\F_2$-flow, see the proof of \cref{3.9}.

\textit{Proof for all non-amenable $\Gamma$.}

    By nonamenability,
    there is a finite symmetric subset $S\subseteq\Gamma$
    containing $1$ such that:
    \begin{enumerate}[label=(\roman*)]
        \item
            for every finite $F\subseteq\Gamma$,
            we have $|FS|\ge 2|F|$;
        \item
            there is an integer $n$ with
            \[
                4 + 3\log_2(|S|)
                \le n
                \le \frac{|S|-6}{3\log_2(|S|)}
            \]
        \item \label{item-nonamen-noninvol}
            there is some $r \in S$ with $r^2 \neq 1$.
    \end{enumerate}
    Let $T = S^n$.
    Given a point $x \in (T\sqcup\{*\})^\Gamma$,
    let $\Supp(x)$ denote the set of $\gamma \in \Gamma$
    such that $x_\gamma \neq *$.
    Let $X$ be the subshift of $(T\sqcup\{*\})^\Gamma$
    such that $x \in X$ iff the following hold:
    \begin{enumerate}[label=(\roman*)]
        \item
            $\Supp(x)$ is maximal right $S^3$-disjoint
            (a subset $A\subseteq\Gamma$ is \textbf{right $S^3$-disjoint}
            if for any distinct $a, a'\in\Supp(A)$,
            we have $a'\notin a S^3$);
        \item
            the function $\gamma\mapsto \gamma x_\gamma$
            is a $2$-to-$1$ surjection from $\Supp(x)$ onto $\Supp(x)$.
    \end{enumerate}
    We claim that $X$ is the desired subshift.
    We first recall a fact from graph theory.
    \begin{lem}\label{hallSurjection}
        Let $G$ be a locally finite
        (not necessarily simple)
        graph with vertex set $V$,
        such that every finite $F\subseteq V$ satisfies $|N_G(F)|\ge k|F|$,
        where $N_G(F)$ denotes the set of neighbours of $F$.
        Then there is a $k$-to-$1$ surjection $p \colon V \to V$
        such that for every $v\in V$,
        there is an edge from $v$ to $p(v)$.
    \end{lem}
    \begin{proof}
        Consider the bipartite graph $B$ with bipartition $(V_l, V_r)$,
        where $V_l = V_r = V$,
        and where there is an edge from $v \in V_l$ to $w\in V_r$
        if $vw$ is an edge in $V$.
        Then every finite $F\subseteq V_l$ satisfies $|N_B(F)| \ge |F|$,
        and every finite $F\subseteq V_r$ satisfies $|N_B(F)|\ge k|F|$,
        so by Hall's theorem \cite[C.4(b)]{TW16},
        there are matchings $(M_i)_{i < k}$
        such that every vertex in $V_l$ is covered by a unique $M_i$,
        and every vertex is $V_r$ is covered by every $M_i$.
        Then $\bigcup_{i < k} M_i$ is
        (the graph of)
        the desired $k$-to-$1$ surjection.
    \end{proof}
    
    We show that $X$ is nonempty.
    Let $A\subseteq\Gamma$ be any maximal right $S^3$-disjoint subset,
    and consider $A$ as a (non-simple) graph
    where $a$ and $a'$ are adjacent iff $a' \in aT$.
    Let $F\subseteq A$ be a finite subset.
    By maximality of $A$,
    every element of $FS^{n-3}$ is within $S^3$
    of some element of $FT\cap A$,
    and thus
    \[
        |FT\cap A|
        \ge \frac{|F S^{n-3}|}{|S^3|}
        \ge \frac{2^{n-3} |F|}{|S|^3}
        \ge 2|F|
    \]
    by our choice of $n$.
    Thus by \cref{hallSurjection},
    there is a $2$-to-$1$ surjection $p \colon A\to A$
    such that $p(a) \in a T$ for every $a\in A$.
    Define $x\in X^\Gamma$ by
    \[
        x_\gamma =
        \begin{cases}
            \gamma^{-1} p(\gamma) & \gamma \in A, \\
            * & \text{otherwise.}
        \end{cases}
    \]
    Then $x\in X$.
    
    Next,
    we show that $X$ is a compressible subshift.
    Let $Y\subseteq X$ be the set of $x\in X$ with $1 \in \Supp(x)$.
    Consider the Borel map $Y\to Y$ defined by $y\mapsto y_1^{-1}\cdot y$.
    This is a $2$-to-$1$ surjection,
    since the preimage of $y\in Y$ is the set
    $\{\gamma^{-1}\cdot y : \gamma y_\gamma = 1\}$.
    Thus $Y$ is a compressible subset,
    so since $Y$ is a complete section,
    $X$ is also a compressible subset.
    
    It remains to show that there is a clopen $2$-generator,
    which suffices by \ref{item-clopengenerator} before \cref{liftProps}.
    
    Recall that if $G$ is a finite graph with maximum degree $d$,
    then every independent set $I\subseteq G$
    can be extended to an independent set
    of size at least $\frac{|G|}{d+1}$.
    To see this, note that for any set $X\subseteq G$, we have that $|N_G [X]| \le (d+1) |X|$, where $N_G [X]$ is the set of vertices within distance 1 of $X$. If $I$ is maximal independent, then clearly $N_G[I] = G$ and we are done.
    
    Consider $S$ as a graph where   $s'$ are adjacent
    iff $s' = s r^{\pm 1}$ and $\{s, s'\}\neq \{1, r\}$,
    where $r$ is as in \ref{item-nonamen-noninvol} before \cref{hallSurjection}.
    Then by the above,
    there is an independent set $S' \supseteq \{1, r\}$
    of size at least $\frac{|S|}{3}$.
    Fix an injection $\phi \colon T\hra 2^{S'}$
    such that $\phi(t)_1 = \phi(t)_r = 1$
    for every $t\in T$;
    this is possible since
    \[
        \log_2(|T|)
        \le n \log_2(|S|)
        \le \frac{|S|}{3} - 2
        \le |S'|-2
    \]
    by our choice of $n$.
    Define the continuous map $f \colon X\to 2$ by
    \[
        f(x) =
        \begin{cases}
            \phi(x_{s^{-1}})_s &
            \text{$s^{-1} \in \Supp(x)$ for some $s\in S'$} \\
            0 & \text{otherwise}
        \end{cases}
    \]
    This is well-defined,
    since if $s_0^{-1}$ and $s_1^{-1}$ are both in $\Supp(x)$,
    then since $\Supp(x)$ is right $S^3$-disjoint,
    we have $s_0 = s_1$.
    
    We claim that
    \[
        \gamma \in \Supp(x)
        \iff f(\gamma^{-1}\cdot x)
        = f((\gamma r)^{-1}\cdot x) = 1
    \]
    For $(\implies)$,
    since $x_\gamma\in T$,
    we have $\phi(x_\gamma)_1 = \phi(x_\gamma)_r = 1$,
    which is equivalent to what we need.
    For $(\impliedby)$,
    we must have some $s_0, s_1 \in S'$
    such that $s_0^{-1} \in \Supp(\gamma^{-1}\cdot x)$
    and $s_1^{-1} \in \Supp((\gamma r)^{-1}\cdot x)$.
    Thus $\gamma s_0^{-1}, \gamma r s_1^{-1} \in \Supp(x)$,
    but since $\Supp(x)$ is right $S^3$-disjoint,
    we get that $\gamma s_0^{-1} = \gamma r s_1^{-1}$,
    i.e. $s_1 = s_0 r$.
    Thus by our choice of $S'$,
    we have $\{s_0, s_1\} = \{1, r\}$,
    and since $r^2\neq 1$,
    we have $s_0 = 1$.
    and thus $1 \in \Supp(\gamma^{-1}\cdot x)$,
    i.e. $\gamma \in \Supp(x)$.
    
    We now show that $f$ is a generator.
    Let $x, x' \in X$,
    and suppose that $f(\gamma\cdot x) = f(\gamma\cdot x')$
    for every $\gamma \in \Gamma$.
    Then by above,
    we have that $\Supp(x) = \Supp(x')$.
    If $\gamma\notin\Supp(x)$,
    then $x_\gamma = * = x_{\gamma'}$.
    If $\gamma\in\Supp(x)$,
    then for any $s\in S'$,
    we have
    \[
        \phi(x_\gamma)_s
        = f((\gamma s)^{-1}\cdot x)
        = f((\gamma s)^{-1}\cdot x')
        = \phi(x'_\gamma)_s
    \]
    So since $\phi$ is injective,
    we have $x_\gamma = x'_\gamma$.
    Thus $x = x'$,
    and $f$ is a generator.
\end{proof}

It turns out that if one is willing to replace $2^\Gamma$ by $k^\Gamma$, where $k$ depends on $\Gamma$, it is easier to get compressible subshifts.

Fix a group $\Gamma$.
For finite subsets $S$ and $T$ of $\Gamma$,
denote by $X_{S,T}$
the space of $(S, T)$-paradoxical decompositions of $\Gamma$,
that is,
the subshift of $(S\sqcup T)^\Gamma$
such that $x\in X_{S, T}$
iff $\{x^{-1}(s) s\}_{s\in S}$ and $\{x^{-1}(t) t\}_{t\in T}$
are both partitions of $\Gamma$
(we allow pieces of a partition to be empty).

For a finite subset $T$ of $\Gamma$,
denote by $X_T$ be the space of $2$-to-$1$ $T$-surjections of $\Gamma$,
that is,
the subshift of $T^\Gamma$ such that
$x\in X_T$ iff the map $\Gamma\to\Gamma$ defined by
$\gamma\mapsto \gamma x_\gamma$ is a $2$-to-$1$ surjection.

Note that $X_{S, T}$ is a subset of $X_{S\cup T}$.
Also,
$\Gamma$ is non-amenable
iff $X_{S, T}$ is nonempty for some $S$ and $T$
iff $X_T$ is nonempty for some $T$.

The Tarski number $k_\Gamma$ of $\Gamma$ is minimum of $|S| + |T|$
over all $S$ and $T$ with $X_{S, T}$ nonempty
(it's the smallest number of pieces in a paradoxical decomposition).
There is a number $l_\Gamma$
which is the minimum of $|T|$ over all $T$ with $X_T$ nonempty,
or equivalently,
the minimum of $|S\cup T|$ over all $S$ and $T$ with $X_{S, T}$ nonempty
(it's the smallest number of group elements required
in a paradoxical decomposition).
Note that we have $l_\Gamma < k_\Gamma$ for any non-amenable $\Gamma$,
since if $X_{S, T}$ is nonempty,
then $X_{S, T\gamma}$ is nonempty for any $\gamma$,
and thus we can assume that $S$ and $T$ have at least one element in common,
i.e. $|S\cup T| < |S| + |T|$. Note that by \cite{EGS15} there are groups $\Gamma$ with arbitrarily large $k_\Gamma$.

\begin{prop}
    $X_{S, T}$ and $X_T$ are compressible.
    Thus if $\Gamma$ is non-amenable,
    then there is a compressible subshift of $(l_\Gamma)^\Gamma$.
\end{prop}
So, for example,
this easily gives a compressible subshift of $3^{\F_2}$.

\begin{proof}
    For $X_{S, T}$,
    let $P$ and $Q$ be the set of $x\in X_{S, T}$
    such that $x_1 \in S$ and $x_1 \in T$ respectively.
    Then the map defined by $x\mapsto x_1^{-1} x$
    is a bijection from $P\to X_{S, T}$ and a bijection $Q\to X_{S, T}$,
    so $X_{S, T}$ is equidecomposable with two copies of itself,
    and thus it is compressible.
    
    For $X_T$,
    let $P$ be the set of $x\in X_T$
    such that $x_1$ is the least of the two elements of
    $\{\gamma : \gamma x_\gamma = x_1\}$
    (in some fixed ordering).
    Then proceed as above.
\end{proof}

From \cref{compressibleSubshift} a group $\Gamma$ is non-amenable
iff there is a compressible subshift of $2^\Gamma$.
The following question asks whether
an analogous characterization exists for groups that contain $\F_2$.

\newcommand{\probcompuniv}{
    Is it true that a group $\Gamma$ contains $\F_2$
    iff there is a compressible,
    orbit-universal subshift of $2^\Gamma$?
}
\begin{prob}\label{prob-compuniv}
    \probcompuniv
\end{prob}

\section{The space of subshifts}\label{subshifts-space}

\subsection{\nopunct}
We will first review the standard embedding of actions into the shift action.
Consider a continuous action of a countable group $\Gamma$ on a Polish space $Y$ and let $Y$ be a closed subspace of a Polish space $X$. Define $f \colon Y \to X^\Gamma$ by
\[
f(y)_\gamma = \gamma^{-1} \cdot y.
\]
Then it is easy to check that $f$ is $\Gamma$-equivariant, $f(Y)$ is a closed subset of $X^\Gamma$ and $f$ is a homeomorphism of $Y$ with $f(Y)$, i.e., the action of $\Gamma$ on $Y$ is (topologically) isomorphic to a subshift of $X^\Gamma$, where of course $\Gamma$ acts on itself by left translation.

For any Polish space $X$,
define the standard Borel space of subshifts of $X^\Gamma$ as follows:
\[
    \Sh(\Gamma, X)
    = \{F \in F(X^\Gamma) : \textup{$F$ is $\Gamma$-invariant}\}
\]
If $X$ is compact,
we view this as a compact Polish space with the Vietoris topology.

Consider the Hilbert cube $\I^\N$.
Every compact Polish space is
(up to homeomorphism)
a closed subspace of $\I^\N$,
and thus every $\Gamma$-flow is
(topologically) isomorphic to a subshift of $(\I^\N)^\Gamma$.
We can thus consider the compact Polish space $\Sh(\Gamma, \I^\N)$
as the universal space of $\Gamma$-flows.

Similarly consider the product space $\R^\N$.
Every Polish space is
(up to homeomorphism)
a closed subspace of $\R^\N$,
and thus every continuous $\Gamma$-action is
(topologically) isomorphic to a subshift of $(\R^\N)^\Gamma$.
We can thus consider the standard Borel space $\Sh(\Gamma, \R^\N)$
as the universal space of continuous $\Gamma$-actions.

In particular taking $\Gamma = \F_\infty$,
the free group with a countably infinite set of generators,
we see that every CBER is Borel isomorphic to the equivalence relation $E_F$ induced on some subshift $F$ of $(\R^\N)^{\F_\infty}$ and so we can view $\Sh(\F_\infty, \R^\N)$ also as the universal space of CBER and study the complexity of various classes of CBER (like, e.g., smooth, aperiodic, hyperfinite, etc.) as subsets of this universal space. Similarly we can view $\Sh(\F_\infty, \I^\N)$ as the universal space of CBER that admit a compact action realization. In this case we can also consider complexity questions as well as generic questions of various classes. 

\subsection{\nopunct}\label{subshifts-space-props}
Let $\Phi$ be a property of continuous $\Gamma$-actions on Polish spaces
which is invariant under (topological) isomorphism.
Let 
\[
    \Sh_\Phi(\Gamma, X) =
    \{F \in \Sh(\Gamma, X) : F \models \Phi\},
\]
where we write $F\models\Phi$ to mean that $F$ has the property $\Phi$.

Denote by $\Prob(\Gamma)$ the closed $\Gamma$-subspace of $[0, 1]^\Gamma$
consisting of probability measures on $\Gamma$.

A Borel action $\Gamma\car X$ on a standard Borel space is \textbf{Borel amenable}
if there is a sequence of Borel maps $p_n \colon X \to \Prob(\Gamma)$
such that, letting $p^x_n = p_n(x)$, $\|p_n^{\gamma\cdot x} - \gamma\cdot p_n^x\|_1\to 0$,
for every $\gamma\in\Gamma$ and $x\in X$.
If $\mu$ is a Borel probability measure on $X$,
then $\Gamma\car X$ is \textbf{$\boldsymbol{\mu}$-amenable}
if there is a $\Gamma$-invariant $\mu$-conull Borel subset of $X$
where the action is Borel amenable.
The action $\Gamma\car X$ is \textbf{measure-amenable}
if it is $\mu$-amenable for every $\mu$.

Let $\Gamma$ be a countable group,
and let $X$ be a Polish space.
A continuous action $\Gamma\car X$ is
\textbf{topologically amenable}
if for every finite $S\subseteq \Gamma$,
every compact $K\subseteq X$,
and every $\epsilon > 0$,
there is some continuous $p \colon X \to \Prob(\Gamma)$,
such that, letting $p^x = p(x)$, we have
\[
    \max_{\substack{\gamma \in S \\ x \in K}}
    \|p^{\gamma\cdot x} - \gamma\cdot p^x\|_1
    < \epsilon.
\]

By \cref{topAmenChar},
measure-amenability is equivalent to topological amenability.

A countable discrete group $\Gamma$ is \textbf{exact}
if it admits a topologically amenable
action on a compact Polish space (see \cite[Chapter 5]{BO08}),
in which case there exists such an action on the Cantor space $\cantor$,
since every compact $\Gamma$-flow is a factor of a $\Gamma$-flow on $\cantor$,
see \cite{GH97}.
Examples of exact groups include all linear groups,
and all groups of finite asymptotic dimension.
There exist groups which are not exact,
but all such groups have been constructed explicitly
for this purpose via small cancellation.

We will consider below the following $\Phi$,
where we recall our terminology convention from
\Cref{intro-subshifts}:

\summarySubshiftProps

The following two problems are open:
\newcommand{\probhfcomeager}{
    Let $\Gamma$ be an infinite exact group.
    Is $\Sh_\hyp(\Gamma, \I^\N)$ comeager in $\Sh(\Gamma, \I^\N)$?
}
\begin{prob}\label{prob-hfcomeager}
    \probhfcomeager
\end{prob}
Note that the hypothesis of exactness is necessary.
In fact,
if $\Gamma$ is non-exact,
then every free continuous action on a compact space is non-measure-hyperfinite,
so in particular,
since the generic action is free,
the generic action is non-measure-hyperfinite.

As it was mentioned before it was shown in \cite{IS25} that this holds if $\Gamma$ is a free group (in fact any group of finite asymptotic dimension, e.g., any hyperbolic group).

\newcommand{\probhfcomplexity}{
    Let $\Gamma$ be an infinite group.
    What is the exact descriptive complexity of
    $\Sh_\hyp(\Gamma, \I^\N)$ in $\Sh(\Gamma, \I^\N)$?
}
\begin{prob}\label{prob-hfcomplexity}
    \probhfcomplexity
\end{prob}

Note that from the results in the 5th row of the above table,
it follows that a countable group $\Gamma$ is amenable
iff the generic subshift of $(\I^\N)^\Gamma$
admits an invariant Borel probability measure.

We will now prove the results in the table in a series of propositions. 
A property $\Phi$ of {\bf Polish} $\boldsymbol{\Gamma -}${\bf spaces} (i.e.,  continuous actions of $\Gamma$ on Polish spaces),
invariant under topological isomorphism, is: 
\begin{itemize}
    \item \textbf{satisfiable} if some Polish $\Gamma$-space satisfies $\Phi$;
    \item \textbf{compactly satisfiable} if some (compact) $\Gamma$-flow satisfies $\Phi$;
    \item \textbf{product-stable}
        if for any (compact) $\Gamma$-flows $\mathbf a$ and $\mathbf b$,
        if $\mathbf a$ satisfies $\Phi$,
        then $\mathbf a\times\mathbf b$ satisfies $\Phi$.
\end{itemize}

\begin{prop}\label{stableDense}
    Let $\Phi$ be a compactly satisfiable, product-stable property.
    Then the set
    \[
        \{K\in \Sh(\Gamma, \I^\N) : \textup{$K$ satisfies $\Phi$}\}
    \]
    is dense in $\Sh(\Gamma, \I^\N)$.
\end{prop}
\begin{proof}
    Since $\I^\N$ is the inverse limit of the spaces $\I^n$,
    we have that $\Sh(\Gamma, \I^\N)$ is the inverse limit of
    $(\Sh(\Gamma, \I^n))_n$.
    Thus it suffices to show, for every $n \in \N$
    and every nonempty open $U\subseteq \Sh(\Gamma, \I^n)$,
    that some subshift in $\pi_n^{-1}(U)$ satisfies $\Phi$,
    where $\pi_n : \Sh(\Gamma, \I^\N)\to \Sh(\Gamma, \I^n)$
    is the projection map.
    Fix $K \in U$,
    and fix $L \in \Sh(\Gamma, \I^{\N\setminus n})$
    satisfying $\Phi$.
    Then $K\times L$ satisfies $\Phi$ by product stability,
    and is contained in $\pi_n^{-1}(U)$,
    so we are done.
\end{proof}

For compact Polish $X$, a subset $\mathcal I\subseteq\Sh(\Gamma, X)$
is a \textbf{$\boldsymbol\sigma$-ideal}
if the following hold:
\begin{enumerate}[label=(\roman*)]
    \item
        if $K \in \mathcal I$, $L \in \Sh(\Gamma, X)$   
        and $L\subseteq K$, then $L \in \mathcal I$;
    \item
        if $K \in \Sh(\Gamma, X)$ and $K = \bigcup_n K_n$
        for some countable sequence $K_n\in \mathcal I$,
        then $K \in \mathcal I$.
\end{enumerate}
Every $\Sh_\Phi(\Gamma, X)$ in the above table is a $\sigma$-ideal.
We will need the following to show $\PI^1_1$-hardness.
It is an analog of \cite[Section 1.4, Theorem 7]{KLW87}
and can be proved by the same argument
which we repeat here for the convenience of the reader.
\begin{prop}\label{idealHard}
    Let $X$ be a compact Polish space
    and let $\mathcal I$ be a $\sigma$-ideal in $\Sh(\Gamma, X)$.
    If $\mathcal I$ is $F_\sigma$-hard,
    then $\mathcal I$ is $\PI^1_1$-hard.
\end{prop}
\begin{proof}
    There is a continuous map $\cantor\to\Sh(\Gamma, X)$
    reducing $2^{<\N}\subseteq \cantor$ to $\mathcal I$,
    which we will denote by $\alpha\mapsto K_\alpha$.
    Consider the composition
    \[
        K(\cantor)
        \to K(\Sh(\Gamma, X))
        \hra K(K(X^\Gamma))
        \to K(X^\Gamma)
    \]
    where the first map is the one induced by
    the map $\cantor \to \Sh(\Gamma, X)$,
    the second map is the inclusion map,
    and the third map is the union map $L \mapsto \bigcup L$.
    The image of this map is contained in $\Sh(\Gamma, X)$,
    so we obtain a continuous map
    $K(\cantor) \to \Sh(\Gamma, X)$
    defined by
    $A\mapsto \bigcup_{\alpha\in A} K_\alpha$.
    This map reduces
    $K(2^{<\N}) = \{K \in K(\cantor) : K \subseteq 2^{<\N}\}$
    to $\mathcal I$,
    since for every $A\in K(\cantor)$,
    we have
    \begin{align*}
        A \subseteq 2^{<\N}
        & \implies \text{$K_\alpha \in \mathcal I$ for all $\alpha\in A$, and $A$ is countable} \\
        & \implies \bigcup_{\alpha \in A} K_\alpha\in \mathcal I \\
        & \implies \text{$K_\alpha \in \mathcal I$ for all $\alpha\in A$} \\
        & \implies A \subseteq 2^{<\N}.
    \end{align*}
    So the result follows,
    since $K(2^{<\N})$ is $\PI^1_1$-hard
    (see \cite[27.4(ii)]{Kec95}).
\end{proof}

For a subset $(F_s)_{s\in\N^{<\N}}$ of $\Sh(\Gamma, \R^\N)$,
there is a closed $\Gamma$-invariant subspace of
the product $\Gamma$-space $\baire\times((\R^\N)^\Gamma)^\N$
(where $\Gamma$ acts trivially on $\baire$),
given by
\[
    \coprod_{\alpha \in \baire} \prod_n F_{\alpha\uhr n}
    = \{(\alpha, (x_n)_n) \in \baire \times ((\R^\N)^\Gamma)^\N :
    \forall n \, [x_n \in F_{\alpha\uhr n}]\}.
\]
Fixing a closed embedding $\baire \times ((\R^\N)^\Gamma)^\N \hra \R^\N$,
we obtain an element of $\Sh(\Gamma, \R^\N)$,
which we denote by $\mathcal A_s F_s$.

\begin{prop}\label{souslinHard}
    Let $\Phi$ and $\Psi$ be disjoint satisfiable properties
    of Polish $\Gamma$-spaces such that
    \begin{enumerate}[label=(\roman*)]
        \item there is some $F_\Phi \in \Sh(\Gamma, \R^\N)$ such that
            for every subset $(F_s)_{s\in \N^{<\N}}$ of $\Sh(\Gamma, \R^\N)$
            such that $\{s\in \N^{<\N} : F_s \neq F_\Phi\}$ is well-founded,
            we have that $\mathcal A_s F_s$ satisfies $\Phi$;
        \item there is some $F_\Psi \in \Sh(\Gamma, \R^\N)$ such that
            for every subset $(F_s)_{s\in {\N^<\N}}$ of $\Sh(\Gamma, \R^\N)$
            such that $\{s\in \N^{<\N} : F_s = F_\Psi\}$ is ill-founded,
            we have that $\mathcal A_s F_s$ satisfies $\Psi$.
    \end{enumerate}
    Then $\Sh_\Phi(\Gamma, \R^\N)$ is $\PI^1_1$-hard.
\end{prop}
\begin{proof}
    Let $\tree\subseteq 2^{\N^{<\N}}$ denote the space of trees,
    and let $\WF\subseteq \tree$ be the subset of well-founded trees,
    which is $\PI^1_1$-complete;
    see \cite[33.A]{Kec95}.
    
    For every $T\in\tree$ and $s\in\N^{<\N}$,
    define $F^T_s \in \Sh_\Phi(\Gamma, \R^\N)$ by
    \[
        F^T_s :=
        \begin{cases}
            F_\Phi & s \notin T \\
            F_\Psi & s \in T
        \end{cases}
    \]
    Then
    \begin{align*}
        T\in\WF & \implies \A_s F^T_s \models\Phi  ,\\
        T\notin\WF & \implies \A_s F^T_s \models\Psi ,
    \end{align*}
    so the Borel map $T\mapsto \A_s F^T_s$ is a reduction
    from $\WF$ to $\Sh_\Phi(\Gamma, \R^\N)$,
    whence the latter is $\PI^1_1$-hard.
\end{proof}

\begin{prop}\label{freeAperShift}
    Let $\Gamma$ be a countably infinite group,
    and let $\Phi \in \{\free, \aper\}$.
    Then $\Sh_\Phi(\Gamma, \I^\N)$ is dense $G_\delta$,
    and $\Sh_\Phi(\Gamma,\R^\N)$ is $\PI^1_1$-complete.
\end{prop}
\begin{proof}
    For every $\gamma\in \Gamma$,
    the set of fixed points of $\gamma$ in $(\I^\N)^\Gamma$
    (resp., $(\R^\N)^\Gamma$)
    is closed
    (resp., Borel).
    Thus the set of points with trivial stabilizer is $G_\delta$
    (resp., Borel),
    whence $\Sh_{\free}(\Gamma, \I^\N)$ is $G_\delta$
    (resp., $\Sh_\free(\Gamma, \R^\N)$ is $\PI^1_1$).
    Similarly,
    the set of points with infinite orbit in $(\I^\N)^\Gamma$
    (resp., $(\R^\N)^\Gamma$)
    is $G_\delta$
    (resp., Borel),
    so $\Sh_\aper(\Gamma, I^\N)$ is $G_\delta$
    (resp., $\Sh_\free(\Gamma, \R^\N)$ is $\PI^1_1$).
    
    The property $\Phi$ is compactly satisfiable
    (see, e.g. \cite[1(B)]{KPT05})
    and product-stable,
    so density of $\Sh_\Phi(\Gamma, \I^\N)$
    follows from \cref{stableDense}.
    
    Finally,
    $\PI^1_1$-completeness follows from \cref{souslinHard}
    by taking $\Psi$ to be ``has a fixed point''.
\end{proof}

\begin{prop}\label{387}
    Let $\Gamma$ be a countably infinite group.
    Then $\Sh_\comp(\Gamma, \I^\N)$ is open,
    $\Sh_\comp(\Gamma, \R^\N)$ is $\PI^1_1$-complete,
    and if $\Gamma$ is non-amenable,
    then the former is dense
    (in $\Sh(\Gamma, \I^\N)$).
\end{prop}
\begin{proof}
    By Nadkarni's theorem (\Cref{nadkarni}),
    $F\in \Sh (\Gamma, \I^\N)$ is non-compressible iff
    \[
        \exists\mu \in P(F)\,
        \forall\gamma\,
        [\gamma\cdot\mu = \mu],
    \]
    where $P(F)$ is the set of Borel probability measures on $F$,
    which is a compact Polish space for $\I^\N$
    and a standard Borel space for $\R^\N$.
    Thus the set of compressible subshifts is open for $\I^\N$,
    and $\PI^1_1$ for $\R^\N$.
    Moreover,
    $\PI^1_1$-completeness follows from \cref{souslinHard}
    by taking $\Phi$ to be ``compressible''
    and $\Psi$ to be ``has a fixed point'',
    and taking any $F_\Phi$ satisfying $\Phi$
    and any $F_\Psi$ satisfying $\Psi$.
    
    Now suppose $\Gamma$ is non-amenable.
    Then compressibility is compactly satisfiable by non-amenability,
    and it is product-stable,
    so density follows from \cref{stableDense}.
\end{proof}

\begin{prop}
    Let $\Gamma$ be a countably infinite group,
    let $X$ be a Polish space,
    and let $\Phi \in \{\fin, \sm\}$.
    Then $\Sh_\Phi(\Gamma, X)$ is $\PI^1_1$,
    and if $X =\I^\N$,
    then it is meager.
\end{prop}
\begin{proof}
    The set of points with finite orbit in $X^\Gamma$ is Borel.
    Also,
    a subshift is smooth iff every orbit is discrete
    (see, e.g., \cite[Corollary 22.3]{Kec10}).
    The set of points with discrete orbit is Borel.
    So in either case,
    $\Sh_\Phi(\Gamma, X)$ is $\PI^1_1$.
    
    If $X = \I^\N$,
    then meagerness follows since $\Sh_\Phi(\Gamma, \I^\N)$
    is disjoint from $\Sh_\aper(\Gamma, \I^\N)$ (by \cref{3.6} for $\Phi =$ sm),
    which is comeager by \cref{freeAperShift}.
\end{proof}

We will see in \cref{resFinShift} that $\Sh _{\Phi }(\Gamma , X)$ is $\boldsymbol{\Pi^1_1}$-hard for $X\in \{ \I^\N , \R^\N \}$, when $\Gamma$ is residually finite. This will complete the proofs of the statements in the first five rows of the table.

We now turn to the various notions of hyperfiniteness and amenability.
\begin{prop}
    Let $\Gamma$ be a countably infinite group
    and let $X$ be a Polish space.
    Then for $\Phi$ in $\{\hyp, \amen, \freeHyp\}$, resp.,    
     $\{\measHyp, \measAmen,    
    \freeMeasHyp \}$,
    $\Sh_\Phi(\Gamma, X)$ is $\SIGMA^1_2$, resp., $\PI^1_1$.
    If moreover $\Gamma$ is non-amenable,
    then $\Sh_\Phi(\Gamma, \R^\N)$ is $\PI^1_1$-hard, for any $\Phi\in \{ \hyp, \amen, \freeHyp, \measHyp, \measAmen, \freeMeasHyp \}$.
\end{prop}
\begin{proof}
    First,
    $\Sh_\hyp(\Gamma, X)$ is $\SIGMA^1_2$, since $F\in \Sh (\Gamma, X)$ is hyperfinite iff 
    \begin{align*}
        & \text{$\exists$ sequence $(E_n)_n$ of Borel subsets of $(X^\Gamma)^2$} \\
        & \qquad [\forall n\,
        [\text{$E_n$ is a finite equivalence relation}
        \text{ and } E_n\subseteq E_{n+1}] \\
        & \qquad \text{ and } \forall x\in F\, \forall \gamma \,\exists n\,
        [(\gamma\cdot x, x)\in E_n]].
    \end{align*}
    
    Next,
    $\Sh_\amen(\Gamma, X)$ is $\SIGMA^1_2$,
    since $F\in \Sh (\Gamma, X)$ is amenable iff (letting $R$ be the shift equivalence relation on $F$)
        \begin{align*}
        & \text{$\exists$ sequence $(f_n)_n$ of Borel functions
        $f_n \colon F^2\to[0,1]$} \\
        & \qquad \forall x\in F\, \qty[\forall n\,
        \sum_{y\in [x]_{R}} f_n^x(y) = 1
        \text{ and }\forall y \in [x]_{R}\,
        \|f_n^x - f_n^y\|_1\to 0],
    \end{align*}
    where $\|{\cdot}\|_1$ is on $\ell^1([x]_{R})$.
    
    $\Sh_\measHyp(\Gamma, X)$ is $\PI^1_1$
    by Miri Segal's effective witness to measure-hyperfiniteness
    (see \Cref{segal}).
    
    Now $\Sh_\freeMeasHyp(\Gamma, X)$ is $\PI^1_1$,
    since $\Sh_\free(\Gamma, X)$ and $\Sh_\measHyp(\Gamma, X)$ are $\PI^1_1$.
    
    Similarly,
    the set of points with amenable stabilizer is $G_\delta$,
    since $x$ has amenable stabilizer iff
    \[
        \forall S\in \Fin(\Gamma)\,
        \qty[S\subseteq \Gamma_x
        \implies
        \forall n \in \N\,
        \exists F\subseteq \ev S\,
        \frac{|SF\btu F|}{|F|} < \frac{1}{n}].
    \]
    where $\Gamma_x = \{\gamma : \gamma\cdot x = x\}$ is the stabilizer of $x$.
    
    Thus the set
    \[
        \{F \in \Sh(\Gamma, X)
        : \forall x\in F\, [\text{$\Gamma_x$ is amenable}]\},
    \]
    is $\PI^1_1$
    (in fact $G_\delta$ when $X$ is compact),
    and thus $\Sh_\measAmen(\Gamma, X)$ is $\PI^1_1$ by \cref{amenStabChar}.
    
    If $\Gamma$ is non-amenable,
    then $\PI^1_1$-hardness follows from \cref{souslinHard}
    by taking $F_\Phi$ to be the $\Gamma$-action by left multiplication on $\Gamma$,
    by taking $\Psi$ to be
    ``has a free non-compressible $\Gamma$-invariant closed subspace'',
    and by taking any $F_\Psi$ satisfying $\Psi$.
\end{proof}

We next show $\PI^1_1$-hardness of
$\Sh_\Phi(\Gamma, \I^\N)$ for various $\Gamma, \Phi$.
\begin{prop}\label{resFinShift}
    Let $\Gamma$ be an infinite residually finite group,
    and let $X$ be $\I^\N$ or $\R^\N$.
    Then $\Sh_\Phi(\Gamma, X)$ is $\PI^1_1$-hard,
    where $\Phi \in \{\fin, \sm\}$.
    If moreover $\Gamma$ is non-amenable,
    then $\Sh_\Phi(\Gamma, \I^\N)$ is $\PI^1_1$-hard,
    where $\Phi \in \{\hyp, \amen, \measHyp\}$.
\end{prop}
\begin{proof}
    Since $\Sh_\Phi(\Gamma, \I^\N)$
    reduces to $\Sh_\Phi(\Gamma, \R^\N)$
    via the inclusion map,
    it suffices to consider the case where $X = \I^\N$.
    By \cref{idealHard},
    it suffices to show $F_\sigma$-hardness.
    We will define a continuous map $\cantor \to \Sh(\Gamma, \I^\N)$
    which simultaneously reduces $2^{<\N}$
    to $\Sh_{\fin}(\Gamma, \I^\N)$ and to $\Sh_{\sm}(\Gamma, \I^\N)$,
    and if moreover $\Gamma$ is also non-amenable
    also to $\Sh_{\measHyp}(\Gamma, \I^\N)$.
    Fix a descending sequence $(\Lambda_n)_n$
    of finite index subgroups of $\Gamma$
    such that $\bigcap_n \Lambda_n = \{1\}$.
    
    Let $K^n_i\in \Sh(\Gamma,\I^\N),n\in \N,  i\in\{0,1\}$,
    be defined as follows:
    $K^n_0$ is an invariant singleton and
    $K^n_1$ is a subshift isomorphic to the action of $\Gamma$ on $\Gamma/\Lambda_n$.
    Consider now the space $\prod_n (\I^\N)^\Gamma$ on which $\Gamma$ acts diagonally
    and let $F \colon \prod_n (\I^\N)^\Gamma \to (\I^\N)^\Gamma$
    be a $\Gamma$-equivariant continuous embedding.
    Finally for each $\alpha\in \cantor$, let
    \[
        f(\alpha) = F\qty(\prod_n K_{\alpha_n}^n).
    \]
    Then $f \colon \cantor \to \Sh(\Gamma,\I^\N)$ is continuous.
    If $\alpha \in 2^{<\N}$, clearly $f(\alpha)$ is finite.
    If $\alpha \notin 2^{<\N}$,
    then $f(\alpha)$ is a free subshift admitting
    an invariant Borel probability measure,
    so it is not smooth.
    If moreover $\Gamma$ is non-amenable,
    it is also not measure-hyperfinite.
\end{proof}

Surprisingly, for certain $\Gamma$, 
the free measure-hyperfinite subshifts
of $(\I^\N)^{\Gamma}$ form a $G_\delta$ set:
\begin{prop}\label{exactComeager}
    \sloppy Let $\Gamma$ be a countably infinite exact group.
    Then $\Sh_\measAmen(\Gamma, \I^\N)$ and
    $\Sh_\freeMeasHyp(\Gamma, \I^\N)$ are dense $G_\delta$.
    Moreover
    $\Sh_\Phi(\Gamma, \I^\N)$ is comeager
    for $\Phi \in \{\amen, \measHyp\}$.
\end{prop}
\begin{proof}
    Measure-amenability is compactly satisfiable
    (by exactness)
    and product-stable,
    so it is dense by \cref{stableDense}.
    To show that $\Sh_\measAmen(\Gamma, \I^\N)$ is $G_\delta$,
    by \cref{idealHard},
    it suffices to show that it is $\mathbf\Sigma^1_1$.
    
    We use the characterization of measure-amenability
    as topological amenability,
    see \cref{topAmenChar}.
    By \cite[12.13]{Kec95},
    there is Borel function $D \colon \Sh (\Gamma, \I^\N) \to (\I^\N)^\N$
    such that $D(K)$ is a dense subset of $K$
    for every $K\in \Sh (\Gamma, \I^\N)$,
    and we can assume that $D(K)$ is $\Gamma$-invariant.
    Fix a compatible metric $d$ on $\I^\N$.
    Then a subshift $K$ is topologically amenable
    iff for every rational $\epsilon > 0$
    and any finite $S\subseteq\Gamma$,
    there is a function $p \colon \N \to \Prob(\Gamma)$
    such that
    \begin{enumerate}[label=(\roman*)]
        \item (uniform continuity)
            for every $\epsilon_1 > 0$,
            there is a $\epsilon_2 >0$ such that
            for every $n, m\in \N$,
            if $d(D(K)_n, D(K)_m) < \epsilon_2$,
            then $\|p^n - p^m\|_1 < \epsilon_1$;
        \item (invariance)
            for every $\gamma\in S$ and every $n, m\in \N$,
            if $D(K)_n = \gamma\cdot D(K)_m$,
            then $\|p^n - \gamma\cdot p^m\|_1 < \epsilon$.
    \end{enumerate}
    So it is $\mathbf\Sigma^1_1$.
    
    Now $\Sh_\freeMeasHyp(\Gamma, \I^\N)$ is dense $G_\delta$,
    since by \cref{amenStabChar},
    it is the intersection of $\Sh_\free(\Gamma, \I^\N)$
    and $\Sh_\measAmen(\Gamma, \I^\N)$,
    which are both dense $G_\delta$
    (the former by \cref{freeAperShift}).
    
    Finally, by the diagram of implications in the beginning of \cref{appendix-amenable-actions}, we have that 
    $\Sh_\freeMeasHyp(\Gamma, \I^\N) \subseteq \Sh_\amen(\Gamma, \I^\N) \subseteq  \Sh_\measHyp(\Gamma, \I^\N)$, so the last two classes are also comeager.    
\end{proof}

\begin{remark}
    Petr Naryshkin brought out recently to our attention the paper of Suzuki \cite{Suz17},
    in which the author shows that in the space $\Act(\Gamma, \cantor)$
    of continuous actions of $\Gamma$ on $\cantor$
    (see \ref{subshifts-cantor-basics} of \Cref{subshifts-cantor}),
    the set of topologically amenable actions is dense $G_\delta$,
    when the group $\Gamma$ is exact.
    In view of \cref{correspondence},
    this gives a different proof of the comeagerness of $\Sh_\measAmen(\Gamma, \I^\N)$. 
\end{remark}

In connection to \Cref{prob-hfcomeager},
we can show density of hyperfinite subshifts for certain groups.
A countable group $\Gamma$ is \textbf{Borel exact}
if there is a free hyperfinite $\Gamma$-flow
(note that replacing ``hyperfinite'' by ``amenable''
gives the definition of an exact group). For example, all subgroups of hyperbolic groups are Borel exact. We give the proof of that fact below, after the next proposition,

An immediate consequence of \Cref{stableDense}
is that there are densely many hyperfinite subshifts for such groups:

\begin{prop}\label{hypDense}
    Let $\Gamma$ be Borel exact.
    Then $\Sh_\hyp(\Gamma, \I^\N)$ is dense.
\end{prop}

We have a criterion for Borel exactness.
A countable group is \textbf{hyperfinite}
if all of its Borel actions
have hyperfinite orbit equivalence relations.
\begin{prop}\label{borelExactCriterion}
    Let $\Gamma$ be a countable group with a $\Gamma$-flow $X$ such that
    \begin{enumerate}[label=(\roman*)]
        \item $X$ is hyperfinite;
        \item the set $S = \{\Gamma_x : x \in X\}$ is countable;
        \item for every infinite $\Delta \in S$,
            the normalizer $N_\Gamma(\Delta)$
            of $\Delta$ in $\Gamma$ is hyperfinite.
    \end{enumerate}
    Then $\Gamma$ is Borel exact.
\end{prop}
\begin{proof}
    We show that for any Borel $\Gamma$-space $Y$
    that $E = E_\Gamma^{X\times Y}$ is hyperfinite;
    the result follows from taking $Y$ to be any free $\Gamma$-flow.
    By partitioning $X$,
    it suffices to show this in the case where every stabilizer of $X$ is finite,
    and the case where every stabilizer of $X$ is infinite.
    
    First assume that every stabilizer of $X$ is finite.
    We first show that $X \times Y$ is hyperfinite.
    Since $X$ is hyperfinite,
    we can write $E_\Gamma^X = \bigcup_n F_n$,
    where each $F_n$ is finite.
    Then each $E\cap (F_n\times I_Y)$ is a finite CBER
    (since each stabilizer is finite),
    so their increasing union $E$ is hyperfinite.
    
    Now assume that every stabilizer of $X$ is infinite.
    Fix a transversal $T$ for the conjugation action of $\Gamma$ on $S$,
    and for every $x \in X$,
    choose $\gamma_x \in \Gamma$ such that
    $\Stab_\Gamma(\gamma_x \cdot x)$ is in $T$.
    For every $\Delta$ in $T$,
    let $Z_\Delta = \{(x, y)\in X\times Y : \Stab_\Gamma(x) = \Delta\}$.
    The map $(x, y)\mapsto \gamma_x \cdot (x, y)$ maps each $E$-class into itself,
    with image contained in $\bigsqcup_{\Delta \in T} Z_\Delta$,
    so it suffices to show that each $E\uhr Z_\Delta$ is hyperfinite,
    but $E\uhr Z_\Delta = E_{N_\Gamma(\Delta)}^{Z_\Delta}$,
    which is hyperfinite since $N_\Gamma(\Delta)$ is hyperfinite.
\end{proof}

We use this to show Borel exactness of hyperbolic groups.
\begin{prop}
    Every hyperbolic group $\Gamma$ is Borel exact.
\end{prop}
\begin{proof}
    The Gromov boundary $\partial\Gamma$ is a $\Gamma$-flow.
    We show that it satisfies the conditions of \Cref{borelExactCriterion}.
    Hyperfiniteness follows from \cite[Theorem A]{MS20}.
    By \cite[Chapitre 8, 30]{GH90},
    every stabilizer is virtually cyclic,
    so there are only countably many stabilizers.
    Finally,
    let $\Delta$ be an infinite stabilizer.
    Since $\Delta$ is virtually cyclic,
    by \cite[III.$\Gamma$, 3.10]{BH99},
    it is a quasiconvex subgroup of $\Gamma$.
    So since $\Delta$ is infinite,
    by \cite[Chapter III.$\Gamma$, 3.16]{BH99},
    it has finite index in its normalizer,
    so its normalizer is also virtually cyclic,
    and thus hyperfinite.
\end{proof}
Since Borel exactness is closed under subgroups,
every subgroup of a hyperbolic group is Borel exact.
G\'abor Elek has pointed out that this can also be deduced from recent results as follows.
Every hyperbolic group,
and hence every subgroup of a hyperbolic group,
has finite asymptotic dimension
(this is due to Gromov;
see \cite{Roe05} for a full proof).
So by \cite[Theorem 1.3(iv)]{GWY17},
there is a free Cantor action with finite dynamic asymptotic dimension,
and such actions are hyperfinite by \cite[Theorem 1.7]{CJMST23};
alternatively,
the generic Cantor action,
which is free,
is hyperfinite by \cite[Theorem 3]{IS25}.

The preceding constructions have some relevance to the question of
whether every non-smooth, aperiodic CBER admits a compact action realization. 

\begin{prop}
    Suppose that $\Gamma$ has an infinite Borel exact quotient.
    Then the set
    \[
        \{F \in \Sh(\Gamma, \R^\N)
        : \text{$E_F$ has a compact action realization}\}
    \]
    is $\SIGMA^1_1$-hard.
\end{prop}
\begin{proof}
    For any quotient $\Gamma'$ of $\Gamma$,
    the natural map $\Sh(\Gamma', \R^\N) \to \Sh(\Gamma, \R^\N)$
    reduces shifts with a compact action realization
    to shifts with a compact action realization.
    So without loss of generality,
    we assume that $\Gamma$ itself is infinite and Borel exact.
    
    Let $K, L \in \Sh(\Gamma, \R^\N)$,
    where $K$ is free compact hyperfinite,
    and $L$ is free smooth.
    For every $T \in \Tr$ and $s \in \N^{<\N}$,
    define $F^T_s \in \Sh(\Gamma, \R^\N)$
    by
    \[
        F^T_s :=
        \begin{cases}
            K & s \in T \\
            L & s \notin T
        \end{cases}
    \]
    Let $f \colon \Tr \to \Sh(\Gamma, \R^\N)$
    be defined by $f(T) = \A_s F^T_s$.
    We claim that $T$ is ill-founded iff $E_{f(T)}$
    has a compact action realization.

    Note that for every $T \in \Tr$,
    every branch of $(F^T_s)_s$ consists of free hyperfinite actions,
    so the product along every branch is free hyperfinite,
    and thus $f(T)$ is (free) hyperfinite.

    If $T$ has an infinite branch,
    then there is a branch of $(F^T_s)_s$
    consisting only of $K$ which is compact free,
    so the product along this branch is compact free,
    and thus non-smooth.
    Thus $f(T)$ is non-smooth hyperfinite,
    and thus has a compact action realization.

    On the other hand,
    if $T$ has no infinite branches,
    then every branch of $(F^T_s)_s$ contains $L$
    which is free smooth,
    so the product along every branch is free smooth,
    and thus $f(T)$ is free smooth,
    and thus does not have a compact action realization.
\end{proof}

\begin{prop}
    For every $x\in \cantor$,
    there is some $F\in \Sh(\F_\infty, \R^\N)$
    such that $E_F$ has a compact action realization,
    but there is no $\Delta^1_1(F, x)$
    isomorphism of $E_F$ with some $E_K$,
    $K\in \Sh(\F_\infty, \I^\N)$. 
\end{prop}
\begin{proof}
    Assume this fails toward a contradiction.
    Then, by $\Delta^1_1$ bounded quantification, there is a $\PI^1_1$ definition of the class of all
    $F\in \Sh(\F_\infty, \R^\N)$
    for which $E_F$ admits a compact action realization,
    contradicting $\SIGMA^1_1$-hardness.
\end{proof}

Informally this implies that there is no ``uniform Borel method''
that will construct a compact action realization
for each CBER with a compact action realization,
even if it is given as a subshift of $(\R^\N)^{\F_\infty}$.

Another version is the following:
For every Borel function
$f \colon \Sh(\F_\infty, \R^\N) \to \Sh(\F_\infty, \I^\N)$,
there is some $F\in \Sh(\F_\infty, \R^\N)$
such that $E_F$ has a compact action realization
but $E_F \not\cong_B E_{f(F)}$.

\subsection{\nopunct}
Let 
\[
    \Sh_\kappa(\Gamma, X)
    = \{F \in \Sh(\Gamma, X) : \text{$F$ is $\kappa$-ergodic}\}
\]
and 
\[
    \Sh_{\free\kappa}(\Gamma, X)
    = \{F \in \Sh(\Gamma, X) : \text{$F$ is free and $\kappa$-ergodic}\}.
\]

We will classify next the complexity of
$\Sh_\kappa(\Gamma, \I^\N)$,
$\Sh_\kappa(\Gamma, \baire)$
as well as their free versions
$\Sh_{\free \kappa}(\Gamma, \I^\N)$,
$\Sh_{\free \kappa}(\Gamma, \baire)$,
for $\kappa \ge 1$ and various $\Gamma$,
in comparison with the results in \ref{subshifts-space-props} above for $\kappa = 0$.

We will first need some results concerning trees in descriptive set theory.
As before we denote by $\Tr$ the space of all subtrees of $\N^{<\N}$
and by $\Tr_2$ the space of all subtrees of $2^{<\N}$.
For each such tree $T$,
we denote the cardinality of its set of infinite branches by $\kappa(T)$.
We now have:

\begin{prop}\label{prop-kappabranch}
    \leavevmode
    \begin{enumerate}[label=(\arabic*)]
        \item \label{item-kappabranch-binary}
            $\{T \in \Tr_2 : \kappa(T) = \aleph_0\}$
            is $\PI^1_1$-complete and
            $\{T \in \Tr_2 : \kappa(T) = 2^{\aleph_0}\}$
            is $\SIGMA^1_1$-complete.
        \item \label{item-kappabranch-infinitary}
            $\{T \in \Tr : \kappa(T) = \kappa\}$,
            where $1 \le \kappa \le \aleph_0$,
            is $\PI^1_1$-complete and
            $\{T \in \Tr : \kappa(T) = 2^{\aleph_0}\}$
            is $\SIGMA^1_1$-complete. 
    \end{enumerate}
\end{prop}

\begin{proof}
    \leavevmode
    \begin{enumerate}[label=(\arabic*)]
        \item
            It is well known that $\{T \in \Tr_2 : \kappa(T) \le \aleph_0\}$ is $\PI^1_1$-complete,
            so $\{T \in \Tr_2 : \kappa(T) = 2^{\aleph_0}\}$ is $\SIGMA^1_1$-complete.
            Fix now a tree $T_0$ in $\Tr_2$ with $\aleph_0$ infinite branches and note that
            $T \mapsto T \cup T_0$ is a Borel reduction of
            $\{T \in \Tr_2 : \kappa(T) \le \aleph_0\}$ to $\{T \in \Tr_2 : \kappa(T) = \aleph_0\}$,
            so the latter is also $\PI^1_1$-complete. 
        \item
            Let $A \subseteq \baire$ be $\PI^1_1$.
            Then there is a closed set $F \subseteq \baire \times (\baire)_0$,
            where $(\baire)_0 = \{x \in \baire : x(0) = 0 \}$,
            such that 
            \[
                x \notin A \iff \exists y \; (x, y) \in F.
            \]
            Let now for $1 \le \kappa < \aleph_0$,
            $C_\kappa$ be a closed subset of $(\baire)_1$,
            where $(\baire)_1 = \{x \in \baire : x(0) = 1\}$,
            such that $C_\kappa$ has cardinality $\kappa$.
            Finally let $F_\kappa = F \cup (\baire \times C_\kappa)$,
            which is a closed subset of $\baire \times \baire$ and
            let $T_\kappa$ be a subtree of $\N^{<\N} \times \N^{<\N}$ such that $F_\kappa = [T_\kappa]$,
            the set of infinite branches of $T_\kappa$.
            Then if $T_n(x) = \bigcup_k \{t \in \N^k : (x|k, t) \in T_n\}$,
            clearly
            \[
                x \in A \iff \kappa(T_\kappa(x)) = \kappa,
            \]
            so the Borel map $x \mapsto T_\kappa(x)$ reduces $A$ to $\{T \in \Tr : \kappa(T) = \kappa\}$,
            so the latter is $\PI^1_1$-complete.
            The results about $\kappa = \aleph_0, 2^{\aleph_0}$
            follow from the corresponding ones in \ref{item-kappabranch-binary} above.
    \end{enumerate}
\end{proof}

Consider the following property \eqref{property-good}
of a countable infinite group $\Gamma$:
\[\label{property-good}
    \parbox{0.8\linewidth}{
        There is a sequence $(\ac a_n)_n$ of free, uniquely ergodic,
        continuous $\Gamma$-actions on compact Polish spaces such that
        the product action $\prod_n \ac a_n$ is uniquely ergodic.}
    \tag{$\star$}
\]
One can see,
using rationally independent rotations of $\T$ and Kronecker’s Theorem,
that $\Z$ satisfies \eqref{property-good}.
Tsankov showed that every amenable group satisfies \eqref{property-good}
and Glasner pointed out that $\SL_n(\Z)$ and free groups satisfy \eqref{property-good}.
It is not clear to what extent the property \eqref{property-good}
is satisfied among countable groups.

Using these results and the method of proof of \cref{souslinHard},
we can now show the following.
\begin{thm}\label{thm-kappaerg}
    \leavevmode
    \begin{enumerate}[label=(\arabic*)]
        \item \label{item-kappaerg}
            For any infinite countable group $\Gamma$, 
            \begin{enumerate}[label=(\alph*)]
                \item
                    $\Sh_1(\Gamma, \I^\N)$ and $\Sh_{\free 1}(\Gamma, \I^\N)$ are $G_\delta$
                    and  $\Sh_\kappa(\Gamma, \I^\N)$ and $\Sh_{\free \kappa}(\Gamma, \I^\N)$,
                    are differences of two $G_\delta$ sets,
                    for $1 < \kappa < \aleph_0$, 
                    while $\Sh_\kappa(\Gamma, \R^\N)$ is $\PI^1_1$-complete,
                    for $1 \le \kappa < \aleph_0$.
                \item \label{item-kappaerg-ctble}
                    $\Sh_{\aleph_0}(\Gamma, \I^\N)$ and $\Sh_{\aleph_0}(\Gamma, \R^\N)$
                    are $\PI^1_1$-complete.
                \item \label{item-kappaerg-cntnm}
                    $\Sh_{2^{\aleph_0}}(\Gamma, \I^\N)$ and $\Sh_{2^{\aleph_0}}(\Gamma, \R^\N)$
                    are $\SIGMA^1_1$-complete.
            \end{enumerate}
        \item \label{item-gooderg}
            For any infinite countable group $\Gamma$ satisfying \eqref{property-good},
            \begin{enumerate}[label=(\alph*)]
                \item \label{item-gooderg-fin}
                     $\Sh_{\free \kappa}(\Gamma, \R^\N)$ is $\PI^1_1$-complete,
                     for $1 \le \kappa < \aleph_0$.
                 \item \label{item-gooderg-ctble}
                     $\Sh_{\free \aleph_0}(\Gamma, \I^\N)$ and $\Sh_{\free \aleph_0}(\Gamma, \R^\N)$
                     are $\PI^1_1$-complete.
                 \item \label{item-gooderg-cntnm}
                     $\Sh_{\free 2^{\aleph_0}}(\Gamma, \I^\N)$ and $\Sh_{\free 2^{\aleph_0}}(\Gamma, \R^\N)$
                     are $\SIGMA^1_1$-complete. 
            \end{enumerate}
    \end{enumerate}
\end{thm}

\begin{proof}
    \leavevmode
    \begin{enumerate}[label=(\arabic*)]
        \item
            \begin{enumerate}[label=(\alph*)]
                \item
                    Consider first $\Sh_\kappa(\Gamma, \I^\N)$.
                    It is enough to show that for $1 \le \kappa < \aleph_0$,
                    $\Sh_{\le \kappa}(\Gamma, \I^\N) = \bigcup_{\lambda \le \kappa}\Sh_\lambda(\Gamma, \I^\N)$
                    is $G_\delta$.
                    Denoting by $P$ the compact space of
                    shift-invariant Borel probability measures on $(\I^\N)^\Gamma$,
                    we note that,
                    using the ergodic decomposition,
                    we have $F \in \Sh_{\le\kappa}(\Gamma, \I^\N) \iff
                        \forall \mu_1, \ldots , \mu_{\kappa + 1} \in P \;
                            [(\forall i \le (\kappa +1) \; (\mu_i(F) = 1))
                            \implies \exists a_1, \dots , a_{\kappa + 1} \in [-1,1] \;
                                (\sum_{i \le \kappa+1} |a_i| = 1  \  \&  \  \sum_{i \le \kappa + 1} a_i \mu_i = 0)]$.

                    Since freeness is a $G_\delta$ condition,
                    the results about $\Sh_{\free \kappa}(\Gamma, \I^\N)$ follow. 

                    Finally the proof of the result about $\Sh_\kappa(\Gamma, \R^\N)$
                    is similar to that in \ref{item-kappaerg}\ref{item-kappaerg-ctble} below,
                    using \cref{prop-kappabranch}\ref{item-kappabranch-infinitary},
                    instead of \cref{prop-kappabranch}\ref{item-kappabranch-binary}.

                \item
                    It is enough to consider $\Sh_{\aleph_0}(\Gamma, \I^\N)$.
                    For any $T \in \Tr_2, s \in 2^{<\N}$,
                    let $F_s^T$ be a 1-point subshift of $(\I^\N)^\Gamma$, if $s \in T$,
                    and a compressible subshift of $(\I^\N)^\Gamma$, if $s\notin T$.
                    Then, as is in the proof of \Cref{souslinHard},
                    but defining $A_s F^T_s$ using the space $\cantor \times ((\I^\N)^\Gamma)^\N$
                    instead of $\baire \times ((\R^\N)^\Gamma)^\N$
                    in the paragraph preceding \cref{souslinHard},
                    the map $T \in \Tr_2 \mapsto A_s F^T_s$ is a Borel reduction of
                    $\{T \in \Tr_2 : \kappa(T) = \aleph_0\}$ to $\Sh_{\aleph_0}(\Gamma, \I^\N)$,
                    so we are done by \cref{prop-kappabranch}\ref{item-kappabranch-binary}.
                \item 
                    The proof is similar to that in \ref{item-kappaerg}\ref{item-kappaerg-ctble}.
            \end{enumerate}
        \item
            Since $\Gamma$ satisfies \eqref{property-good},
            there is a family $(F_s)_{s\in \N^{< \N}}$ of free subshifts of $(\I^\N)^\Gamma$
            which are uniquely ergodic and whose product is also uniquely ergodic.
            Then, again as in the proof of \cref{souslinHard},
            define for any $T \in \Tr, s\in \N^{<\N}$,
            $F_s^T$ to be $F_s$, if $s\in T$,
            and a free, compressible subshift of $(\I^\N)^\Gamma$, if $s\notin T$.
            The the Borel map $T \in \Tr \mapsto A_s F^T_s$ is a Borel reduction of
            $\{T \in \Tr : \kappa(T) = \kappa\}$ to $\Sh_\kappa(\Gamma, \R^\N)$,
            so we are done for \ref{item-gooderg}\ref{item-gooderg-fin}
            by \cref{prop-kappabranch}\ref{item-kappabranch-infinitary}.
            The case of $\I^\N$ in \ref{item-gooderg}\ref{item-gooderg-ctble}\ref{item-gooderg-cntnm},
            is analogous to that in \ref{item-kappaerg}\ref{item-kappaerg-ctble}\ref{item-kappaerg-cntnm},
            and the case of $\R^\N$ clearly follows from this.
    \end{enumerate}
\end{proof}

\begin{remark}\label{3818}
    If $\Gamma$ has an infinite factor satisfying \eqref{property-good},
    then it is clear that all the results in \cref{thm-kappaerg}\ref{item-gooderg}
    hold if ``free'' is replaced by ``aperiodic''.
\end{remark}

In combination with results in \cref{subshifts-cantor},
we also have the following dichotomy result,
which also gives another characterization of amenability for countable groups
(at least in the finitely generated case).

\begin{thm}
    Let $\Gamma$ be an infinite countable group.
    \begin{enumerate}[label=(\arabic*)]
        \item
            If $\Gamma$ is not amenable,
            then the generic subshift of $(\I^N)^\Gamma$ is compressible,
            i.e., has no ergodic invariant measures.
        \item
            If $\Gamma$ is finitely generated and amenable,
            then the generic subshift of $(\I^N)^\Gamma$
            has $2^{\aleph_0}$ ergodic invariant measures.
    \end{enumerate}
\end{thm} 
\begin{proof}
    \leavevmode
    \begin{enumerate}[label=(\arabic*)]
        \item
            This follows from \cref{387}.
        \item
            This follows from the results in \cref{subshifts-cantor},
            see \cref{cor-fgnottrans}\ref{item-fgnottrans-eqrel}
            and \cref{correspondence}.
    \end{enumerate}
\end{proof}

\section{The space of actions on the Cantor space}\label{subshifts-cantor}
\subsection{\nopunct}\label{subshifts-cantor-basics}
Let $\Gamma$ be a countable group,
and let $X$ be a compact zero-dimensional Polish space.
Denote by $\Act(\Gamma, X)$ the space
of group homomorphisms $\Gamma\to\Homeo(X)$, i.e., $\Gamma$-flows on the space $X$.
The group $\Homeo(X)$ is a Polish group with the uniform convergence topology and $\Act(\Gamma, X)$ is a closed subspace of $\Homeo(X)^\Gamma$ with the product topology, so it is a Polish space in this topology. The Polish group $\Homeo(X)$ acts continuously by conjugation on $\Act(\Gamma, X)$.

For $\ac a \in \Act(\Gamma, X)$, let
$\mathcal A_\ac a$ denote the Boolean algebra
of clopen $\ac a$-invariant subsets of $X$,
and let $\St(\mathcal A_\ac a)$ denote its Stone space.
There is a continuous $\ac a$-equivariant surjection
$X \thra \St(\mathcal A_\ac a)$
defined by sending $x$ to the ultrafilter
$\{A \in \mathcal A_\ac a : x \in A\}$.
For every $\mathcal U\in\St(\mathcal A_\ac a)$,
the fiber $C^{\ac a}_\mathcal U$ above $\mathcal U$
is a closed $\ac a$-invariant subset of $X$,
giving the decomposition
\[
    X = \bigsqcup_{\mathcal U \in \St(\mathcal A_\ac a)}
    C^{\ac a}_\mathcal U.
\]

Let $\CEINV(\ac a)$ denote the subset of the space $\INV(\ac a)$
of invariant Borel probability measures for $\ac a$,
consisting of \textbf{clopen-ergodic} measures,
that is,
measures $\mu \in \INV(\ac a)$ for which
every $A\in \mathcal A_\ac a$ is $\mu$-null or $\mu$-conull.
Note that $\CEINV(\ac a)$ is closed by the Portmanteau Theorem
\cite[17.20(v)]{Kec95},
and we have
\[
    \EINV(\ac a) \subseteq \CEINV(\ac a) \subseteq \INV(\ac a).
\]
There is a map $\CEINV(\ac a)\thra\St(\mathcal A_\ac a)$
sending $\mu$ to the ultrafilter $\{A\in \mathcal A_\ac a : \mu(A) = 1\}$, which is surjective if $\Gamma$ is amenable,
and the fiber above $\mathcal U$
can be identified with $\INV(\ac a \uhr C^{\ac a}_\mathcal U)$,
giving a decomposition
\[
    \CEINV(\ac a) = \bigsqcup_{\mathcal U\in\St(\mathcal A_\ac a)}
    \INV(\ac a \uhr C^{\ac a}_\mathcal U).
\]
\begin{prop}
    Suppose $\Gamma$ is amenable.
    Let $\ac a \in \Act(\Gamma, X)$.
    If $|\mathcal A_\ac a| > 2$,
    then $\CEINV(\ac a)$ is a proper subset of $\INV(\ac a)$,
    so in particular, $\INV(\ac a)$ is not Poulsen.
    If $\mathcal A_\ac a$ is atomless,
    then $\EINV(\ac a)$ has size continuum.
\end{prop}
\begin{proof}
    If $|\mathcal A_\ac a| > 2$,
    then $|\St(\mathcal A_\ac a)| \ge 2$,
    so we see from the decomposition that
    $\CEINV(\ac a)$ is not closed under convex combinations,
    and is thus a strict subset of $\INV(\ac a)$.
    If $\mathcal A_\ac a$ is atomless,
    then $\St(\mathcal A_\ac a)$ has size continuum,
    so $\EINV(\ac a)$ has size continuum,
    since each $\INV(\ac a\uhr C^{\ac a}_\mathcal U)$
    is nonempty by amenability of $\Gamma$.
\end{proof}
The following fact appears in
\cite[Remark 5]{Ele19}, as was pointed out to us by J. Melleray:
\begin{prop}\label{392}
Consider the action of $\Homeo(\cantor)$ by conjugation on $\Act(\Gamma, \cantor)$. Then there is a dense conjugacy class.
\end{prop}
\begin{proof} 
Let $(\ac a_n)$ be a dense sequence in $\Act(\Gamma, \cantor)$ and consider the product action $\prod_n \ac a_n$. Then an isomorphic copy of this action in $\Act(\Gamma, \cantor)$ has dense conjugacy class.
\end{proof}

\begin{prop}\label{393a}
    Suppose $\Gamma$ is finitely generated.
    Then for comeagerly many $\ac a \in \Act(\Gamma, \cantor)$,
    $\mathcal A_\ac a$ is atomless,
    so in particular if $\Gamma$ is amenable,
    then $\EINV(\ac a)$ has size continuum
    and $\INV(\ac a)$ is not Poulsen.
\end{prop}
\begin{proof}
    Let $\mathcal A$ be the set of nonempty clopen subsets of $\cantor$.
    Then $\mathcal A_\ac a$ is atomless iff
    for every $A\in \mathcal A$,
    if $A$ is $\ac a$-invariant,
    then there is a partition $A = A_0\sqcup A_1$
    into $\ac a$-invariant $A_0, A_1\in\mathcal A$.
    So it suffices to fix $A\in\mathcal A$,
    and show comeagerness of the set of $\ac a$
    such that if $A$ is $\ac a$-invariant,
    then there is a partition $A = A_0\sqcup A_1$
    into $\ac a$-invariant $A_0, A_1\in\mathcal A$.
    This set is open, since $\Gamma$ is finitely generated,
    so it suffices to show that it is dense.
    Let $V$ be a nonempty open subset of $\Act(\Gamma, \cantor)$.
    We can assume that $A$ is $\ac a$-invariant
    for every $\ac a\in V$,
    otherwise we are done.
    Then $V$ gives by restriction an open subset of $\Act(\Gamma, A)$,
    so since the set of $\ac a\in\Act(\Gamma, A)$,
    with a partition $A = A_0\sqcup A_1$
    in $\mathcal A_\ac a$, is closed under conjugation,
    we are done,
    since $\Act(\Gamma, A)$ has a dense conjugacy class
    (because $A\cong \cantor$).
\end{proof}

\begin{remark}
    \cref{393a} may fail if $\Gamma$ is not finitely generated. If $\Gamma = \F_\infty$, then the set of all minimal actions $\ac a \in \Act(\Gamma, \cantor)$ is  dense $G_\delta$ but
$|\mathcal A_\ac a| =2$.
\end{remark}

\newcommand{\probbauer}{
    If $\Gamma$ is finitely generated and amenable,
    is it true that for comeager many $\ac a \in \Act(\Gamma, \cantor)$,
    $\INV(\ac a)$ is a Bauer simplex,
    i.e., $\EINV(\ac a)$ is closed in $\INV(\ac a)$?
}
\begin{prob}\label{prob-bauer}
    \probbauer
\end{prob}

\subsection{\nopunct}
Recently work of Doucha \cite{Dou24} and Doucha-Melleray-Tsankov \cite{DMT25}
made important progress in understanding the structure of generic actions in $\Act(\Gamma, \cantor)$.
We state below some of their results as they relate to \ref{subshifts-cantor-basics} above.

\begin{thm}[\cite{DMT25}]
    Let $\Gamma$ be not finitely generated.
    Then for comeager many $\ac a \in \Act(\Gamma, \cantor)$,
    $\ac a$ is topologically transitive.
\end{thm}

In combination with (the proof of) \cref{393a},
this implies the following dynamical characterization of finitely generated groups:

\begin{cor}\label{cor-fgnottrans}
    The following are equivalent for a countable group $\Gamma$:
    \begin{enumerate}[label=(\roman*)]
        \item
            $\Gamma$ is finitely generated;
        \item
            For comeager many $\ac a \in \Act(\Gamma, \cantor)$,
            $\ac a$ is not topologically transitive;
        \item \label{item-fgnottrans-eqrel}
            For comeager many $\ac a \in \Act(\Gamma, \cantor)$,
            $\ac a$ admits an invariant closed equivalence relation
            on $\cantor$ with $2^{\aleph_0}$ classes.
    \end{enumerate}
\end{cor}

It has been known since the work of Kechris-Rosendal \cite{KR07}, for $\Z$, and then Kwiatkowska \cite{Kwi12} for every finitely generated free group $\Gamma$, that there is $\ac a \in \Act(\Gamma, \cantor)$ with comeager conjugacy class. This was further extended in \cite{Dou24}, Theorem B, to all finite free products of finite or cyclic groups. On the other hand it was shown in \cite{KR07} that all conjugacy classes of $\ac a \in \Act(\Gamma, \cantor)$, for $\Gamma$ the infinitely generated free group, are meager. Remarkably it turns out that this is true for all non finitely generated groups.

\begin{thm}[{\cite[6.1]{Dou24} and \cite{DMT25}}]
    If $\Gamma$ is not finitely generated, then all conjugacy classes of $\ac a \in \Act(\Gamma, \cantor)$ are meager.
\end{thm}
Finally we have the following results:

\begin{thm}[\cite{DMT25}]
    \leavevmode
    \begin{enumerate}[label=(\arabic*)]
        \item
            If $\Gamma$ is amenable and not locally finite,
            then for comeager many $\ac a \in \Act(\Gamma, \cantor)$,
            $\ac a$ is not minimal;
        \item
            If $\Gamma$ is infinite, locally finite,
            then for comeager many $\ac a \in \Act(\Gamma, \cantor)$,
            $\ac a$ is minimal and uniquely ergodic.
    \end{enumerate}
\end{thm}

The following are open (see \cite{DMT25}):

\begin{prob}
    \leavevmode
    \begin{enumerate}[label=(\arabic*)]
        \item
            If $\Gamma$ is non-amenable and not finitely generated,
            is it true that for comeager many $\ac a \in \Act(\Gamma, \cantor)$,
            $\ac a$ is minimal?
        \item
            If $\Gamma$ is amenable, not locally finite,
            what is the cardinality of $\EINV(\ac a)$
            for a comeager set of $\ac a \in \Act(\Gamma, \cantor)$?
    \end{enumerate}
\end{prob}

\subsection{\nopunct}
In \cite{Ele19} Elek studies a notion of weak containment in the space $\Act(\Gamma, \cantor)$.
An equivalent form of his definition is the following:
Given $\ac a, \ac b \in \Act(\Gamma, \cantor)$,
we say that $\ac a$ is \textbf{weakly contained} in $\ac b$,
in symbols
\[
    \ac a \preceq \ac b,
\]
if $\ac a$ is in the closure of the conjugacy class of $\ac b$.
Weak containment is a quasi-order on $\Act(\Gamma, \cantor)$,
which has a maximum element by \cref{392}.
Note that the set of maximum elements,
i.e., those with dense conjugacy class,
is a dense $G_\delta$ in $\Act(\Gamma, \cantor)$.
It is shown in \cite[Section 3.3]{Ele19}
that there is also a minimum element among the free actions,
for finitely generated $\Gamma$.

Certain properties of actions are upwards monotone in the quasi-order $\preceq$:

\begin{thm}\label{395}
    \leavevmode
    \begin{enumerate}[label=(\roman*)]
        \item
            If $\ac a, {\ac b} \in \Act(\Gamma, \cantor)$,
            $\ac a\preceq {\ac b}$ and $\ac a$ is free,
            so is ${\ac b}$.
        \item
            \textup{\cite[Section 3.1]{Ele19}}
            If $\ac a, {\ac b} \in \Act(\Gamma, \cantor)$,
            $\ac a\preceq {\ac b}$ and $\ac a$ is compressible,
            so is ${\ac b}$.
        \item
            \textup{\cite[Section 3.1]{Ele19}}
            If $\ac a, {\ac b} \in \Act(\Gamma, \cantor)$,
            $\ac a\preceq {\ac b}$ and $\ac a$ is topologically amenable,
            so is ${\ac b}$.
    \end{enumerate}
\end{thm}
We include for the convenience of the reader proofs of these results, in the framework of the above definition of weak containment, in \Cref{appendix-weak-containment}.

\subsection{\nopunct}
The Correspondence Theorem of Hochman \cite[Theorem 1.3 and Section 10]{Hoc08},
shows the equivalence of genericity properties in the spaces
$\Act(\Gamma, \cantor)$ and $\Sh(\Gamma, \I^\N)$.
Let $\Phi$ be a property of continuous $\Gamma$-actions on Polish spaces
which is invariant under (topological) isomorphism.
Recall that $\Sh_\Phi(\Gamma, \I^\N)$ is the set of
all subshifts of $(\I^\N)^\Gamma$ that have the property $\Phi$.
Let also 
\[
    \Act_\Phi(\Gamma, \cantor)
\]
be the set of all actions in $\Act(\Gamma, \cantor)$ that have the property $\Phi$.
Then we have:

\begin{thm}[Hochman {\cite[Theorem 1.3 and Section 10]{Hoc08}}]\label{correspondence}
    Let $\Phi$ be a property of continuous $\Gamma$-actions on Polish spaces
    which is invariant under (topological) isomorphism.
    Then the following are equivalent:
    \begin{enumerate}[label=(\roman*)]
        \item
            $\Act_\Phi(\Gamma, \cantor)$ is comeager;
        \item
            $\Sh_\Phi(\Gamma, \I^\N)$ is comeager.
    \end{enumerate}
\end{thm}

We include a somewhat simplified proof of this theorem in
\cref{appendix-correspondence-theorem}.

In particular \cref{395} and the Correspondence Theorem gives another proof that for $\Phi \in \textrm{\{free, compressible, topologically amenable\}}$, the set $\Sh_\Phi(\Gamma, \I^\N)$ is comeager, for infinite $\Gamma$ which are also non-amenable for the compressible case and exact in the topologically amenable case.

\chapter{$K_\sigma$ Countable Borel Equivalence Relations}\label{ksigma}

\section{$K_\sigma$ and $F_\sigma$ realizations}

Clinton Conley raised the following question:
Does every aperiodic CBER have a realization as
a $K_\sigma$ equivalence relation in a Polish space?
We answer this question in the affirmative:
\begin{thm}\label{transitiveKSigma}
    Every $E\in\ap$ has a transitive $K_\sigma$ realization
    in the Cantor space $\cantor$.
\end{thm}
\begin{proof}
    Let  $Q = 2^{<\N} \subseteq \cantor$,
    and let $N = \cantor\setminus Q$.
    Then $N$ is homeomorphic to the Baire space,
    so by \cref{BaireActionRealization},
    we can assume that $E = E_\Gamma^N$,
    where $\Gamma\car N$ is a continuous action
    of a countable group $\Gamma$ on $N$.
    For each $\gamma\in\Gamma$,
    let $R_\gamma$ be the relation on $N$ defined by
    $x\mr{R_\gamma} y \iff y = \gamma\cdot x$.
    Let $\ol{R_\gamma}$ denote the closure of $R_\gamma$ in $(\cantor)^2$.
    We claim that $\ol{R_\gamma}\subseteq R_\gamma \oplus I_Q$.
    Let $(x, y) \in \ol{R_\gamma}$,
    and suppose that $x\in N$
    (the case $y\in N$ is identical).
    Then there is a sequence $(x_n, y_n)_n$ in $R_\gamma$
    converging to $(x, y)$.
    Since $x\in N$,
    we have that $y_n = \gamma\cdot x_n \to \gamma\cdot x$,
    so $y = \gamma\cdot x$,
    and thus $(x, y)\in R_\gamma$,
    proving the claim.
    Thus the relation $E\oplus I_Q$ on $\cantor$
    which is isomorphic to $E$
    (see \Cref{prelim-classes}).
    is equal to $I_Q \cup \bigcup_\gamma \ol{R_\gamma}$,
    so it is $K_\sigma$,
    and it has the dense class $Q$.
\end{proof}

We can also ask about $K_\sigma$ and $F_\sigma$ realizations which are minimal. The answer for $F_\sigma$ minimal realizations turns out to be positive. 

\begin{thm}\label{t:fsminimal}
    Every $E \in \mathcal{AE}$ admits
    a minimal $F_\sigma$ realization. 
\end{thm}
For the proof of \cref{t:fsminimal},
we will need the following lemma:
\begin{lem}
    Every aperiodic {\CBER} is Borel isomorphic to a {\CBER} $E$ on $\N \times \baire$ such that
    \begin{enumerate}[label=(\roman*)]
        \item for every $n > 0$,
            the set $\{n\} \times \baire$ is a complete section for $E$.
        \item $E = \Delta_{\N \times \baire} \cup \bigcup_{f \in \mc F} \grph(f)$,
            where $\mc F$ is a countable family of continuous maps
            whose domain is a closed subset of some $\{n\} \times \baire$ with $n > 0$,
            and whose codomain is some $\{m\} \times \baire$;
    \end{enumerate}
\end{lem}
\begin{proof}
    Split the underlying space of $E$ into a disjoint union
    of Borel sets $H$ and $(B_n)_{n \ge 1}$,
    such that every $B_n$ is a complete $E$-section,
    and $H$ is an uncountable partial $E$-transversal.
    Now do the same as in the proof of \Cref{BaireActionRealization},
    after making $H$ and every $B_n$ clopen.
\end{proof}

\begin{proof}[Proof of \Cref{t:fsminimal}]
    Fix a countable partition $\baire = H \sqcup \bigsqcup_{C \in \mc C} C$ such that
    \begin{enumerate}[label=(\roman*)]
        \item $H$ and every $C \in \mc C$ is homeomorphic to $\baire$;
        \item every $C \in \mc C$ is closed;
        \item every sequence $(x_C)_{C \in \mc C}$ in $\baire$
            with every $x_C \in C$ is dense.
    \end{enumerate}
    For instance,
    fix an injection $s \mapsto n_s$ from $\N^{<\N}$ to $\N$,
    then for every $s \in \N^{<\N}$,
    let $C_s$ be the set of $x \in \baire$ with $x \succ s$
    such that $x_{2i} = n_s$ for every $i > |s|$.
    Then take $\mc C = (C_s)_{s \in \N^{<\N}}$,
    and set $H = \baire \setminus \bigsqcup_n C_n$.
    
    By the lemma and by fixing homeomorphisms
    between $\{n\} \times \baire$ and $C_n$,
    and between $\{0\} \times \baire$ and $H$,
    we can assume that every $C \in \mc C$ is a complete $E$-section,
    and that $E = \Delta_{\baire} \cup \bigcup_{f \in \mc F} \grph(f)$,
    where $\mc F$ is a countable family of continuous maps
    whose domain is a closed subset of some $C \in \mc C$,
    and whose codomain is either $H$ or some $C \in \mc C$.
    Since every $C \in \mc C$ is a complete $E$-section,
    every $E$-class contains a sequence $(x_C)_{C \in \mc C}$ with $x_C \in \mc C$,
    and hence every $E$-class is dense.
    Every $f \in \mc F$ is a continuous map from a closed subset of $\baire$ to $\baire$,
    so its graph is closed as a subset of $(\baire)^2$,
    and hence $E$ is $F_\sigma$.
\end{proof}

For $K_\sigma$ minimal realizations,
there is a known obstruction, due to Solecki:
\begin{thm}[Solecki, {\cite[Corollary 3.2]{Sol02}}]
    Every minimal $K_\sigma$ equivalence relation on a Polish space
    with at least two classes is not smooth.
\end{thm}
It is open whether this is the only obstruction:
\newcommand{\probksigma}{
    Let $E$ be an aperiodic {\CBER}.
    If $E$ is non-smooth,
    does $E$ have a minimal $K_\sigma$ realization?
}
\begin{prob}\label{prob-ksigma}
    \probksigma
\end{prob}

\cref{3.9} shows that all non-smooth relations in $\aph$
have minimal $K_\sigma$ realizations.

In contrast to Solecki's result,
we can show the following:
\begin{prop}
    Every aperiodic smooth {\CBER} can be realized
    as a minimal equivalence relation
    which is a Boolean combination of $K_\sigma$ relations
    in a compact Polish space.
\end{prop}

\begin{proof}
    Here are two such realizations:
    \begin{enumerate}[label=(\arabic*)]
        \item
            Consider the equivalence relation $E_0$ in $\cantor$.
            Let $A$ be a Cantor set in $\cantor$ which is a partial transversal for $E_0$.
            Let $B$ be the $E_0$-saturation of $A$ and put $Y= \cantor\setminus B$.
            Then $Y$ is $G_\delta$,
            so a zero-dimensional Polish space
            (in the relative topology).
            Every compact subset of $Y$ has empty interior in $Y$,
            so $Y$ is homeomorphic to the Baire space $\baire$
            (see \cite[7.7]{Kec95}).
            Therefore there is a continuous bijection $f \colon Y\to A$
            (see \cite[7.15]{Kec95}).
            Let $F$ be the equivalence relation on $\cantor$ obtained by adding to each $E_0$ class $[a]_{E_0}$,
            with $a\in A$,
            the point $f^{-1} (a)$.
            Then $F$ is smooth with all classes dense.
            Put 
            \[
            S(x,y) \iff x\in B \ \& \  y\in Y \ \& \ \exists z\in A (xE_0z \ \& f(y) =z)
            \]
            and 
            \[
            T(x,y) \iff S(y,x).
            \]
            Then each of $S,T$ is the intersection of two $K_\sigma$ relations with a $G_\delta$ relation and 
            \[
            xFy \iff (x, y\in B \ \&  \ x E_0 y)\vee S(x,y) \vee T(x,y),
            \]
            so $F$ is a Boolean combination of $K_\sigma$ relations as well.
            
        \item
            Let $X = \prod_{n \ge 1} 2^n$,
            where $2^n$ is the set of binary sequences of length $n$.
            Let $Y = \{(x_n) \in X : \exists m \forall n \ge m (x_n\prec x_{n+1})\}$,
            and define $f \colon X \to \cantor$ as follows:
            \[
                f(x) =
                \begin{cases}
                    \lim_n x_n & x \in Y \\
                    x_1 \concat x_2 \concat x_3\concat \cdots & x\notin Y
                \end{cases}
            \]
            Let $x E y \iff f(x) = f(y)$.
            Then $E$ is a smooth CBER with all classes dense,
            and it is easy to check that $E = F_1 \cup F_2 \cup F_3 \cup F_4$,
            where $F_1$ is $K_\sigma$,
            $F_2$ and $F_3$ are intersections of a $K_\sigma$ and a $G_\delta$ relation
            and $F_4$ is the equality relation on $X$.
    \end{enumerate}
\end{proof}
We next discuss a sharper notion of $K_\sigma$ realization. Let $X$ be a compact Polish space and $E$ a CBER in $X$. Recall that we say that $E$ is {\bf compactly graphable} if there is a compact graphing of $E$, i.e., a compact graph (irreflexive, symmetric relation) $K\subseteq E$ so that the $E$-classes are the connected components of $ K$. Note then that $E$ is $K_\sigma$.
A CBER $E$ has a {\bf compactly graphable realization} if it is Borel isomorphic to a compactly graphable CBER. Clearly every CBER that has a compact action realization implemented by a free continuous action of a \textit{finitely generated} group has a compactly graphable realization. Also clearly a CBER that has a a compactly graphable realization admits a $K_\sigma$ realization.

We now have the following result:

\begin{thm}\label{3.9.5}
    \leavevmode
    \begin{enumerate}[label=(\arabic*)]
        \item
            Every aperiodic hyperfinite {\CBER} has a compactly graphable realization.
        \item \label{compressible-compactly-graphable}
            Every compressible {\CBER} has a compactly graphable realization.
    \end{enumerate}
\end{thm}
\begin{proof}
    \leavevmode
    \begin{enumerate}[label=(\arabic*)]
        \item This follows from \cref{3.9} for non-compressible hyperfinite CBER.
            The compressible case is covered in \ref{compressible-compactly-graphable}.
        \item The proof is a modification of the proof of \cref{transitiveKSigma}.
            Let $E$ be a compressible CBER.
            Then by \cite[Proposition 1.8]{DJK94}, \cite[Proposition 3.27]{Kec25}
            and the arguments in the proof of \cref{BaireActionRealization},
            we can assume that $E$ is of the form $E=E^N_{\F_2}$,
            where $N$ is as in the proof of \cref{transitiveKSigma}.
            Let $\alpha, \beta$ be free generators of $\F_2$
            and let $S$ consist of these generators and their inverses.
            Finally, as in the proof of \cref{transitiveKSigma},
            let $K= \bigcup_{\gamma \in S} \ol{R_\gamma}$,
            and note that if $F$ is the equivalence relation generated by $K$
            (i.e., the smallest equivalence relation containing $K$),
            then $F$ is of the form $E\oplus R$,
            where $R$ is an equivalence relation on the space $Q$.
            Thus $E$ is Borel bireducible to $F$.
            Now let $Y= \{1, 1/2, 1/3, \dots , 0 \}$ and define on $Y$
            the compact, connected graph $R$, where $y R y'$ iff 
            \[
                \qty(y = 1 ,  y' \le \frac{1}{2}) \textrm{ or }
                \qty(y' = 1, y \le \frac{1}{2}) \textrm{ or }
                \qty(y = \frac{1}{2}, y' \le \frac{1}{3}) \textrm{ or }
                \qty(y' = \frac{1}{2}, y \le \frac{1}{3}).
            \]
            
            Consider now the equivalence relation $G = F\times I_Y$ on $\cantor \times Y$,
            where as usual $I_Y = Y^2$.
            Thus $(x,y) G (x', y') \iff x F x'$.
            Then the compact relation $\tilde{K}$ on $\cantor \times Y$ given by 
            \[
                (x,y)\tilde{K} (x',y') \iff x Kx' \textrm{ and } yRy',
            \]
            is a compact graphing of $G$.
            But $G$ is Borel bireducible to $F$ and thus to $E$.
            Since both $E$ and $G$ are compressible,
            they are Borel isomorphic by \cite[Proposition 3.27]{Kec25}
            and the proof is complete.
    \end{enumerate}
\end{proof}

The following is an open problem:
\newcommand{\probgraphable}{
    Does every aperiodic {\CBER} admit a compactly graphable realization?
}
\begin{prob}\label{prob-graphable}
    \probgraphable
\end{prob}

Recently,
Allison Wang \cite{Wan25}
obtained a positive answer in the following situation:
let the countable group $\Gamma$ act continuously on $\cantor$
so that the orbit equivalence relation $E$ is minimal and aperiodic,
and such that some element in $\Gamma$ has no fixed points.
Then $E$ is compactly graphable.

\section[A $\sigma$-ideal associated to a $K_\sigma$ CBER]{A $\sigma$-ideal associated to a $K_\sigma$ countable Borel equivalence relation}\label{ksigma-ideal}

Suppose that $X$ is an (uncountable) Polish space and $E$ a {\CBER} on $X$.
Denote by $K(X)$ the space of compact subsets of $X$
with the usual Vietoris topology (see \cite[4.F]{Kec95}).
Let 
\[
    I_E = \{K \in K(X) : [K]_E \neq X\}.
\]
Recall that a $\sigma$-ideal of compact sets is a nonempty subset
$I\subseteq K(X)$ such that $K\subseteq L \in I\implies K\in I$ (i.e., it is hereditary) and
$K\in K(X), K = \bigcup_n K_n, K_n \in I, \forall n \implies K\in I$
(i.e., it is closed under countable unions which are compact).

\begin{prop}\label{323}
    Let $X$ be a Polish space and $E$ a $K_\sigma$ {\CBER} on $X$ with all $E$-classes dense.
    Then for every sequence $(K_n)_n$ in $I_E$,
    we have $\bigcup_n [K_n]_E \neq X$,
    and $I_E$ is a $G_\delta$ $\sigma$-ideal of compact sets.
\end{prop}
\begin{proof}
    Here and in the sequel,
    notice that since $E$ is $K_\sigma$,
    $X = \{x\in X : (x,x)\in E\}$ (and $X^2$)
    is also $K_\sigma$ and $F_\sigma = K_\sigma$ on $X$ (and $X^2$).

    Note that if $K \in I_E$,
    then since $E$ is $K_\sigma$,
    the set $[K]_E$ is also $K_\sigma$,
    so since it has dense complement,
    it is meager.
    Thus given a sequence $(K_n)_n$ in $I_E$,
    each $[K_n]_E$ is meager,
    and hence $\bigcup_n [K_n]_n$ is meager,
    and hence a proper subset of $X$.
    
    Clearly $I_E$ is hereditary,
    so $I_E$ is a $\sigma$-ideal.
    Finally,
    since
    \[ 
        K\in I_E \iff \exists x \forall y (y\in K \implies \neg xEy),
    \]
    clearly $I_E$ is $\boldsymbol{\Sigma^1_1}$,
    thus by \cite[Theorem 11]{KLW87}
    (see also \cite[Theorem 1.4]{MZ07})
    it is $G_\delta$.
\end{proof}

\begin{cor}\label{324}
    If $X,E$ are as in \cref{323} and moreover $E$ admits a meager complete section,
    then $E$ admits a nowhere dense, compact complete section.
\end{cor}
\begin{proof}
    We have a sequence $K_n$ of nowhere dense compact sets with $[\bigcup_n K_n]_E = \bigcup_n [K_n]_E = X$. Thus for some $n$, $K_n \notin I_E$, so $K_n$ is a nowhere dense, compact complete section.
\end{proof}
Below denote by $K_{\aleph_0}(X)$ the $\sigma$-ideal of countable compact subsets of $X$
and by $\MGR(X)$ the $\sigma$-ideal of nowhere dense (i.e., meager) compact subsets of $X$.

\begin{cor}
    If $X,E$ are as in \cref{324}, then 
    \[
        K_{\aleph_0}(X)\subsetneq I_E \subsetneq \MGR(X).
    \]
\end{cor}
\begin{cor}
    If $X, E$ are as in \cref{323},
    then $E$ does not admit a $K_\sigma$ transversal.
\end{cor}
\begin{proof}
    Let $F \subseteq X$ be a $K_\sigma$ partial transversal of $E$.
    Since $F$ is $K_\sigma$ and has at least two elements,
    we can write $F = F_0 \sqcup F_1$,
    where each $F_i$ is also nonempty $K_\sigma$
    (for instance,
    fix some $x \in X$ and set $F_0 = \{x\}$ and $F_1 = F\setminus\{x\}$).
    Then each $F_i$ is the union of countably many sets in $I_E$,
    so $F$ is as well,
    and hence $[F]_E$ is a proper subset of $X$.
\end{proof}
We say that a $\sigma$-ideal of compact sets $I$ satisfies {\bf Solecki's Property (*)} if for any sequence $K_n\in I, \forall n$ , there is a $G_\delta$ set $G$ such that $\bigcup_n K_n\subseteq G$ and $K(G) = \{ K\in K(X) : K\subseteq G\}\subseteq I$; see \cite{Sol11}.

\begin{prop}
If $X,E$ are as in \cref{323}, then $I_E$ satisfies Solecki's Property (*).
\end{prop}
\begin{proof}
Let $K_n\in I_E, \forall n$. Then there is $x\in X$ such that $[x]_E \cap [\bigcup_n K_n]_E =\emptyset$ and thus if $G=X\setminus [x]_E$, $G$ is $G_\delta$ and $K(G)\subseteq I_E$.
\end{proof}
In particular $I_E$ admits a representation as in \cite[Theorem 3.1]{Sol11}.

A $\sigma$-ideal $I$ of compact sets is {\bf ccc} if there is no uncountable collection of pairwise disjoint compact sets which are not in $I$. Since for any CBER $E$ every $K\notin I_E$ is a complete section, it follows that $I_E$ is ccc.

On the other hand,
let $I^*_E$ be the $\sigma$-ideal of subsets of $X$ generated by $I_E$,
i.e, for $A\subset X$,
$A\in I^*_E \iff \exists (K_n) (K_n \in I_E, \forall n\in \N, \ \textrm{and} \ A\subseteq \bigcup_n K_n)$.
Then $I^*_E$ need not be ccc,
in fact we have the following:

\begin{prop}\label{328}
    Let $X,E$ be as in \cref{323} and moreover for every nonempty open set $U\subseteq X$ there is a meager complete section contained in $U$. Then there is a homeomorphic embedding $f \colon \cantor \times \baire \to {X}$ such that for every $\alpha\in \cantor$, we have $f(\{\alpha\}\times \baire) \notin I^*_E.$
\end{prop}
\begin{proof}
    By \cite[Section 3, Lemma 9]{KS95}, it is enough to show that for every nonempty open $U\subseteq X$, there is a nowhere dense compact set $K\subseteq U$ with $K\notin I_E$. This follows as in the proof of \cref{324}.
\end{proof}

A $\sigma$-ideal $I$ of compact sets has the {\bf covering property} if for every $\boldsymbol{\Sigma^1_1}$ set $A\subseteq X$, either $A\subseteq \bigcup_n K_n$, where $K_n\in I, \forall n$, or else $K(A)\nsubseteq I$. It is {\bf calibrated} if whenever $K\in K(X)$ and $K_n\subseteq K$ are such that $K_n\in I, \forall n$, and $K(K\setminus \bigcup_n K_n)\subseteq I$, then $K\in I$.

\begin{prop}
    Let $X,E$ be as in \cref{323}. Then $I_E$ does not have the covering property and is not calibrated.
\end{prop}
\begin{proof}
    Fix $x\in X$ and let $G=X\setminus [x]_E$. This provides a counterexample to the covering property.

    Let now $K$ be a compact complete section, which exists by the proof of \cref{323}. Then $K = (K\cap [x]_E ) \cup (K\setminus [x]_E)$ gives a counterexample to calibration.
\end{proof}

We next provide an example of a pair $X,E$ satisfying all the properties of \cref{328}, and which therefore satisfies all the preceding propositions. We take $X$ to be the collection of all subsets $A$ of $\N$ such that $0\in A, 1\notin A$, with the usual topology. We let then $E$ be the restriction of many-one equivalence to $X$. It is easy to see that $E$ is a $K_\sigma$ CBER and every $E$-class is dense. Finally if $U$ is an open subset of $X$, which we can assume that it has the form $U = \{A\in X : F_1\subseteq A, F_2\cap A =\emptyset \}$, for two disjoint finite subsets $F_1, F_2$ of $\N$, then for a large enough number $n$ the set $K = \{ A\in U : A  \  \textrm{contains only even numbers} >n \}$ is a meager complete section contained in $U$.

\chapter{Generators and 2-adequate Groups}\label{generators}
For each infinite countable group $\Gamma$ and standard Borel space $X$,
consider the shift action of $\Gamma$ on $X^\Gamma$
and let $E(\Gamma, X)$ be the associated equivalence relation
and $E^{\mathrm{ap}}(\Gamma, X)$ be its aperiodic part,
i.e., the restriction of $E(\Gamma, X)$ to the set of points with infinite orbits.
Consider now a Borel action of $\Gamma$ on an uncountable standard Borel space,
which we can assume is equal to $\R$.
Then the map $f \colon X\to \R^\Gamma$ given by $x\mapsto p_x$,
where $p_x(\gamma) = \gamma^{-1}\cdot x$,
is an equivariant Borel embedding of this action to the shift action on $\R^\Gamma$.
In particular,
for every aperiodic CBER $E$ induced by a Borel action of $\Gamma$,
we have that $E \sqsubseteq^i_B E(\Gamma, \R)$,
i.e. $E$ can be realized as
(i.e., is Borel isomorphic to)
the restriction of $E^{\mathrm{ap}}(\Gamma, \R)$ to an invariant Borel set. 

Now recall that for a Borel action of $\Gamma$ on a standard Borel space $X$
and $n\in \{2,3,\dots ,  \dots, \N\}$
an $\boldsymbol{n}${\bf -generator} is a Borel partition $X = \bigsqcup_{i < n} X_i$
such that $\{\gamma\cdot X_i : \gamma \in \Gamma , i< n\}$ separates points,
i.e., generates the Borel sets in $X$.
This is equivalent to having a equivariant Borel embedding
of the action to the shift action on $n^\Gamma$.

It is shown in \cite[5.4]{JKL02} that for every such action with infinite orbits,
there exists an $\N$-generator.
It follows that every aperiodic equivalence relation $E$ induced by a Borel action of $\Gamma$
can be realized as the restriction of $E^{\mathrm{ap}}(\Gamma, \N)$ to an invariant Borel set.
In particular $E^{\mathrm{ap}}(\Gamma, \R)\cong_B E^{\mathrm{ap}}(\Gamma, \N)$.
However because of entropy considerations,
even for the group $\Gamma = \Z$,
it is not the case that every such action with invariant measure has a finite generator. 

Weiss \cite{Wei89} asked whether for $\Gamma = \Z$
every Borel action without invariant measure admits a finite generator.
Tserunyan \cite[5.7]{Tse15} showed that answer is affirmative for \textit{any} infinite countable group $\Gamma$ if the action is Borel isomorphic to a continuous action on a $\sigma$-compact Polish space.
Then Hochman \cite{Hoc19} provided a positive answer to Weiss' question (for $\Z$).
Finally this work culminated in the following complete answer:

\begin{thm}[Hochman-Seward]\label{compressible-2-generator}
    Every Borel action of a countable group on a standard Borel space
    without invariant measure admits a $2$-generator.
\end{thm}

This however leaves open the possibility that every aperiodic CBER $E$
induced by a Borel action of $\Gamma$ can be realized as
the restriction of $E^{\mathrm{ap}}(\Gamma, 2)$ to an invariant Borel set.
This is clearly equivalent to the statement that
$E^{\mathrm{ap}}(\Gamma, \R)\cong_B E^{\mathrm{ap}}(\Gamma, 2)$
and it also equivalent to the statement that
there is a Borel action of $\Gamma$ that generates $E$ and has a $2$-generator.
This leads to the following concept.

\begin{defi}
    An infinite countable group $\Gamma$ is called $\boldsymbol{2}${\bf -adequate} if 
    \[
        E^{\mathrm{ap}}(\Gamma, \R) \cong_B E^{\mathrm{ap}}(\Gamma, 2).
    \]
\end{defi}

\begin{remark}
    Thomas \cite{Tho12} studies the question of when $E(\Gamma, \R) \sim_B E(\Gamma, 2)$.
\end{remark}

The first result here is the following:
\begin{thm}\label{amenable-2-adequate}
    Every countably infinite amenable group is $2$-adequate.
\end{thm}
\begin{proof}
    Let $\Gamma$ be a countably infinite amenable group,
    let $E$ be an orbit equivalence relation of $\Gamma$,
    and let $F = E^{\mathrm{ap}}(\Gamma, 2)$.
    By \Cref{corollary-hf-compressible},
    we have $E = E_{\text{hf}} \oplus E_{\text{comp}}$
    for some hyperfinite $E_{\text{hf}}$
    and compressible $E_{\text{comp}}$.
    Since $E_{\text{comp}}$ is compressible,
    by Hochman-Seward (\Cref{compressible-2-generator}),
    we can fix an invariant Borel embedding $E_{\text{comp}} \sqsubseteq_B^i F$.
    Again by \Cref{corollary-hf-compressible},
    we have $F = F_{\text{hf}} \oplus F_{\text{comp}}$
    for some hyperfinite $F_{\text{hf}}$
    and compressible $F_{\text{comp}}$.
    By enlarging $F_{\text{comp}}$,
    we can assume that the image of the invariant Borel embedding
    $E_{\text{comp}} \sqsubseteq_B^i F$ is contained in $F_{\text{comp}}$,
    so that in particular we have $E_{\text{comp}} \sqsubseteq^i_B F_{\text{comp}}$.
    Now since $|\EINV_F| = 2^{\aleph_0}$ and $F_{\text{comp}}$ is compressible,
    we have $|\EINV_{F_\text{hf}}| = 2^{\aleph_0}$,
    and hence by the classification (\Cref{thm-hf-classif}),
    we have $E_{\text{hf}} \sqsubseteq_B^i F_{\text{hf}}$.
    Hence
    $E \cong_B E_{\text{hf}} \oplus E_{\text{comp}}
    \sqsubseteq^i_B F_{\text{hf}} \oplus F_{\text{comp}}
    \cong_B F$.
    
\end{proof}

Thomas \cite[Page 391]{Tho12} asked the question of whether
there is an infinite amenable $\Gamma$ such that
$E(\Gamma, \R) \not\le_B E(\Gamma, 2)$.
\Cref{amenable-2-adequate} provides a negative answer in a strong form.

To discuss other examples of $2$-adequate groups, we will need the following strengthening of \cref{2.7}.

\begin{prop}\label{4.5}
Let $E\in\ap$ and let $R \subseteq E$ be hyperfinite. Then there is an $R \subseteq F \subseteq E$ with $F\in \aph$.
\end{prop}

\begin{proof}
    Suppose $E$ lives on the standard Borel space $X$ and let 
    \[
        Y = \{x : [x]_E  \ \textrm{contains a finite nonempty set of finite} \ R \textrm{-classes}\}.
    \]
    Then $Y$ is $E$-invariant and $E \uhr Y$ is smooth,
    thus we can let $F=E$ on $Y$.
    Let $W = \{x : [x]_E  \ \textrm{contains no finite} \ R \textrm{-classes}  \}$.
    Then we can take $F=R$ on $W$.
    
    So we can assume that each $E$-class contains infinitely many finite $R$-classes.
    Let $Z =\{ x : [x]_R \ \textrm{is finite}\}$.
    Then $R \uhr Z$ is $R$-invariant and smooth, so let $S$ be a Borel selector and $T$ the associated Borel transversal $T= \{x : S(x) =x\}$.
    Then, by \cref{2.7}, let $F'$ be a hyperfinite aperiodic Borel equivalence relation on $T$ such that $F'\subseteq E \uhr T$.
    Let then $F''$ be the equivalence relation on $Z$ defined by $x F'' y \iff S(x) F' S(y)$.
    It is clearly aperiodic, hyperfinite, and $R \uhr Z \subseteq F'' \subseteq E \uhr Z$.
    Finally put $F = F'' \cup R \uhr (X\setminus Z)$.
\end{proof}

We also consider the following class of countable groups.

\begin{defi}\label{4.6}
    A countable group $\Gamma$ is {\bf hyperfinite generating} if for every $E\in \aph$ there is a Borel action of $\Gamma$ that generates $E$.
\end{defi}
We will see in \cref{5.2} that every countable group
with an infinite amenable factor is hyperfinite generating. 

We now have the next result that generalizes \cref{4.5} from $\Z$ to any hyperfinite generating group.
The proof is similar,
noting that any smooth aperiodic CBER can be generated by a Borel action of any infinite countable group.

\begin{prop}\label{4.7}
    Let $E\in\ap$ and let $R \subseteq E$ be generated by a Borel action of $\Gamma$,
    where $\Gamma$ is a hyperfinite generating group.
    Then there is $R \subseteq F \subseteq E$ with $F\in \ap$ generated by a Borel action of $\Gamma$.
\end{prop}

\begin{prop}\label{4.8}
    Let $\Gamma$ be any countable group and $\Delta$ a hyperfinite generating, $2$-adequate group.
    Then $\Gamma * \Delta$ is $2$-adequate.
\end{prop}
\begin{proof}
    Fix a Borel action $\ac a$ of $\Gamma * \Delta$ on an uncountable standard Borel space $X$ generating an aperiodic equivalence relation that we denote by $E_{\ac a}$. Let ${\ac b}= \ac a\uhr\Delta, {\ac c}= \ac a\uhr\Gamma$ and denote by $E_{\ac b}, E_{\ac c}$ the associated equivalence relations, so that $E_{\ac a} = E_{\ac b} \vee E_{\ac c}$. By \cref{4.7} find a Borel action ${\ac b}'$ of $\Delta$ such that $E_{{\ac b}'}$ is aperiodic and $E_{\ac b} \subseteq E_{{\ac b}'} \subseteq E_{\ac a}$, so that $E_{\ac a} = E_{{\ac b}'} \vee E_{\ac c}$. Let now $\ac a'$ be the action of $\Gamma * \Delta$ such that $\ac a'\uhr\Delta = {\ac b}', \ac a'\uhr\Gamma ={\ac  c}$, so that $E_{\ac a'} = E_{\ac a}$. Since ${\ac b}'$ has a $2$-generator, so does $\ac a'$ and the proof is complete.
\end{proof}

It will be shown in \cref{5.2} that all groups that have an infinite amenable factor are hyperfinite generating. Thus we have:

\begin{cor}\label{4.9}
    The free product of any countable group with a 2-adequate group that has an infinite amenable factor and thus, in particular, the free groups $\F_n, 1 \le n \le \infty$, are $2$-adequate.
\end{cor}

The following is immediate:

\begin{prop}\label{4.10}
If $\Gamma, \Delta$ are countable groups, every aperiodic equivalence relation induced by a Borel action of $\Gamma$ can be also induced by a Borel action of $\Delta$, $\Delta$ is a factor of $\Gamma$ and $\Delta$ is $2$-adequate, so is $\Gamma$. In particular, for any $1 \le n \le \infty$,  every $n$-generated countable group that factors onto $\F_n$ is $2$-adequate.
\end{prop}

The next two results owe a lot to some crucial observations by Brandon Seward.

\begin{prop}\label{4.11}
    Let $\Gamma$ be $n$-generated, $1 \le n \le \infty$.
    Then $\Gamma\times \F_n$ is $2$-adequate.
    In particular, all products $\F_m\times \F_n$, $1 \le m,n \le \infty$, are $2$-adequate.
\end{prop}
\begin{proof}
    Let $\{\gamma_i\}_{i < n}$ be generators for $\Gamma$ and let $\{\alpha_i\}_{i < n}$ be free generators for $\F_n$. Consider a Borel action $\ac a$ of $\Gamma\times \F_n$ with $E_a$ aperiodic. Then the equivalence relation $E_i$ generated by $\ac a \uhr \ev{\gamma_i, \alpha_i}$ is generated by a Borel action of $\Z^2$ thus is hyperfinite,
    by a theorem of Weiss
    see, e.g., \cite[1.20]{JKL02},
    and so is given by a Borel action $\ac a_i$ of $\Z$. Let ${\ac b}$ be the Borel action of $\F_n$ in which the generator $\alpha_i$ acts like $\ac a_i$. Then $E_{\ac b} = \bigvee_i E_i = E$ and the proof is complete by \cref{4.10}.
\end{proof}
Finally not every infinite countable group is $2$-adequate. The argument below follows the pattern of the proofs in \cite[Section 6]{Tho12}.

\begin{thm}\label{4.12}
The group $\SL_3(\Z)$ is not $2$-adequate.
\end{thm}
\begin{proof}
Assume that $\Gamma = \SL_3(\Z)$ is $2$-adequate, towards a contradiction. Then in particular $E^{\mathrm{ap}} (\Gamma, 3) \cong_B E^{\mathrm{ap}} (\Gamma, 2)$, say via the Borel isomorphism $f$. Let $\mu$ be the usual product of the uniform measure on $3^\Gamma$. Then $\nu = f_* \mu$ is an ergodic, invariant measure for the shift action of $\Gamma$ on $2^\Gamma$, thus by Stuck-Zimmer \cite[2.1]{SZ94} this shift action is free $\nu$-a.e. This gives a contradiction by the arguments in \cite[Section 6]{Tho12}.
\end{proof}
We conclude this section with the following problem:

\newcommand{\probadequate}{
    Characterize the $2$-adequate groups.
}
\begin{prob}\label{prob-adequate}
    \probadequate
\end{prob}

\chapter{Additional Results}\label{additional}

\section{Hyperfinite generating groups}
We introduced in \Cref{generators} the concept of hyperfinite generating groups.
We will establish here some equivalent formulations of this concept
and in particular prove the fact mentioned in the paragraph after \cref{4.8}.
Below we let $\mu$ be the unique $E_0$-invariant measure on $2^\N$,
and we denote by $[E_0]$ the measure-theoretic full group of $E_0$ with respect to $\mu$
(see \Cref{prelim-measure}).
For a countable subgroup $\Delta \le [E_0]$,
we denote by $E_\Delta$ the subequivalence relation of $E_0$
induced by the action of $\Delta$ on $\cantor$.
This is again understood to be defined only $\mu$-a.e.

Below an {\bf IRS} on a countable group $\Gamma$ is a measure
on the space of subgroups of $\Gamma$ invariant under conjugation.
We say that an IRS $\mu$ has some property P
if $\mu$-almost all $\Delta \le \Gamma$ have property P.
Finally a subgroup $\Delta \le \Gamma$ is {\bf co-amenable}
if the action of $\Gamma$ on $\Gamma/\Delta$ is amenable,
i.e., admits a finitely additive probability measure.

\begin{prop}\label{5.1}
    Let $\Gamma$ be an infinite countable group.
    Then the following are equivalent:
    \begin{enumerate}[label=(\roman*)]
        \item \label{hfgen-equiv-hfgen}
            $\Gamma$ is hyperfinite generating;
        \item \label{hfgen-equiv-E0}
            There is a Borel action of $\Gamma$ that generates $E_0$;
        \item \label{hfgen-equiv-hf}
            $\Gamma$ admits a Borel action which generates
            a non-compressible, aperiodic hyperfinite equivalence relation;
        \item \label{hfgen-equiv-fullgroup}
            $\Gamma$ admits a factor $\Delta \le [E_0]$ such that
            $E_\Delta$ has a $\mu$-positive set of infinite orbits.
    \end{enumerate}
    Moreover, if $\Gamma$ is hyperfinite generating,
    $\Gamma$ admits a co-amenable \textup{IRS} with infinite index.
\end{prop}
\begin{proof}
    Clearly \ref{hfgen-equiv-hfgen}
    $\implies$ \ref{hfgen-equiv-E0}
    $\implies$ \ref{hfgen-equiv-fullgroup}
    $\implies$ \ref{hfgen-equiv-hf}.
    We show \ref{hfgen-equiv-hf}
    $\implies$ \ref{hfgen-equiv-E0}
    $\implies$ \ref{hfgen-equiv-hfgen}.
    
    To see that \ref{hfgen-equiv-hf} $\implies$ \ref{hfgen-equiv-E0},
    note that by classification of hyperfinite CBER
    (\Cref{thm-hf-classif}),
    $\Gamma$ generates $SE_0$ for some $S$,
    so by passing to an invariant subset,
    we see that $\Gamma$ generates $E_0$.

    To see that \ref{hfgen-equiv-E0} $\implies$ \ref{hfgen-equiv-hfgen},
    note that $\Gamma$ generates $E_t$ by \cite[11.2]{DJK94},
    and if $\Gamma$ generates $E_0$,
    then $\Gamma$ generates $SE_0$ for every $S$,
    so we are done by classification of hyperfinite CBER
    (\Cref{thm-hf-classif}).
    
    Finally, the last statement follows as in the proof of (vii) $\implies$ (x)
    in the last paragraph of \cite[Appendix D]{BK20}
    (finite generation is not required there).
\end{proof}

\begin{cor}\label{5.2}
    Every countable group that has an infinite amenable factor is hyperfinite generating.
\end{cor}
\begin{proof}
    If $\Gamma$ is infinite amenable, consider its shift action on $2^\Gamma$, equipped with the product of the uniform measure, with associated equivalence relation $E= E(\Gamma, 2)$. Then $E$ and $E_0$ are measure-theoretically isomorphic, so the measure-theoretic full group of $E$ is isomorphic to $[E_0]$. Since $\Gamma \le [E]$ we have an embedding $\pi \colon \Gamma\to [E_0]$ such that if $\Delta = \pi (\Gamma)$, then $E_\Delta = E_0$, which completes the proof.
\end{proof}

It also immediately follows from \cite[Theorem 13]{Mil06} that every countable group that has a factor of the form $\Gamma * \Delta$, where $\Gamma, \Delta$ are non-trivial subgroups of $[E_0]$, is hyperfinite generating.

On the other hand, not every infinite countable group is hyperfinite generating.

\begin{prop}\label{5.3}
    No infinite countable group with property \textup{(T)} is hyperfinite generating.
\end{prop}
\begin{proof}
    See, for example, the proof of \cite[Proposition 4.14]{Kec10}.
\end{proof}
It is also shown in \cite[page 29]{Kec10}
that there are groups without property (T) which are not hyperfinite generating;
explicitly,
every group of the form $\Z/2 * \Gamma$ satisfies these conditions,
where $\Gamma$ is a simple, non-residually finite, property (T) group.

The following is an open problem.

\newcommand{\probhypgen}{
    Characterize the hyperfinite generating groups.
}
\begin{prob}\label{prob-hypgen}
    \probhypgen
\end{prob}

\section{Dynamically compressible groups}
In the course of the previous investigations the following property of countable groups came up.

\begin{defi}
    An infinite countable group $\Gamma$ is called {\bf dynamically compressible}
    if for every aperiodic $E_\Gamma^X$,
    there is a compressible $E_\Gamma^Y$ with $E_\Gamma^X \le_B E_\Gamma^Y$.
\end{defi}

Here is an equivalent formulation of this notion.
\begin{prop}
    A countable group $\Gamma$ is dynamically compressible iff for every aperiodic $E_\Gamma^X$,
    $E_\Gamma^X\times I_\N$ is induced by a Borel action of $\Gamma$.
\end{prop}

\begin{proof}
    Since $E_\Gamma^X\times I_\N \le_B E_\Gamma^X$, if $E_\Gamma^X \le_B E_\Gamma^Y$, with $E_\Gamma^Y$ compressible, then $E_\Gamma^X\times I_\N \le_B E_\Gamma^Y$, therefore $E_\Gamma^X\times I_\N\sqsubseteq^i_B E_\Gamma^Y$ by \cite[Proposition 2.27]{Kec25}.
\end{proof}

We now have:
\begin{prop}\label{5.7}
    Every infinite countable amenable group is dynamically compressible.
\end{prop}

\begin{proof}
    Consider any aperiodic $E=E_\Gamma^X$, which we can clearly assume is not compressible,
    so admits an invariant measure.
    Then let $(X_e)_{e \in \EINV_E}$ be its ergodic decomposition.
    Then there is a Borel set $Y_e\subseteq X_e$ with $e(Y_e) =1$ such that $E \uhr Y_e$ is hyperfinite,
    thus $E \uhr Y_e \le_B E_t$.
    As usual $Y = \bigcup_e Y_e$ is Borel and $E \uhr Y \le_B \R E_t \le_B E_t$.
    Now $E \uhr (X\setminus Y)$ is compressible and $E_t$ is induced by a Borel action of $\Gamma$ by \cite[11.2]{DJK94},
    so the proof is complete.
\end{proof}

\begin{prop}\label{5.6}
    If $\F_2 \le \Gamma$,
    then $\Gamma$ is dynamically compressible.
\end{prop}

\begin{proof}
    Let $E_\Gamma^X$ be aperiodic. Then $E_\Gamma^X = E_{\F_\infty} \le_B  E_{\F_\infty} \times I_\N = E_{\F_\infty}^Y$, for $Y = X\times \N$. Now $\F_\infty \le \Gamma$, so by using the inducing construction from the action of $\F_\infty$ on $Y$, see \cite[2.3.5]{BK96}, we have $E_{\F_\infty}^Y \le_B E^Z_\Gamma$ for some compressible $E^Z_\Gamma$.
\end{proof}
Therefore only the groups that are not amenable but do not contain $\F_2$ can possibly fail to be dynamically compressible. But even among those there exist dynamically compressible groups.

\begin{prop}\label{5.9}
Let $\Gamma$ be a countable group for which there is an infinite group $\Delta$ such that $\Gamma\times \Delta \le \Gamma$. Then $\Gamma$ is dynamically compressible.
\end{prop}

\begin{proof}
    Let $E_\Gamma^X$ be aperiodic.
    Then for $Y = X\times \N$,
    $E_\Gamma^X \le_B E_\Gamma^X\times I_\N = E^Y_{\Gamma\times \Delta} \le_B E^Z_\Gamma$,
    where $E^Z_\Gamma$ is obtained by inducing from the action of $\Gamma\times\Delta$ on $Y$.
\end{proof}

As a result, any countable group of the form $\Gamma\times \Delta^{<\N}$, for an infinite $\Delta$, is dynamically compressible. Take now $\Gamma$ to be any group that is not amenable and does not contain $\F_2$ and consider $G = \Gamma\times \Z^{<\N}$. Then $G$ is dynamically compressible and clearly is not amenable. Moreover it does not contain $\F_2$ because of the following standard fact.

\begin{prop}
    Let $G, H$ be two groups such that $\F_2 \le G \times H$.
    Then $\F_2 \le G$ or $\F_2 \le H$.
\end{prop}
\begin{proof}
    Let $\pi \colon \F_2 \to H$ be the second projection.
    If it has trivial kernel, then $\F_2 \le H$.
    Else either $\F_2 \le \ker(\pi) \le G$ or $\ker(\pi) \cong \Z$.
    In the latter case, by \cite[3.110]{LS01},
    $[\F_2 : \ker(\pi)]$ is finite, so by \cite[3.9]{LS01}, 
    \[
        [\F_2 : \ker(\pi)]
        = \frac{\rank(\ker(\pi)) - 1}{\rank(\F_2) - 1}
        = 0,
    \]
    a contradiction.
\end{proof}
We now have the following open problem:
\newcommand{\probdyncomp}{
    Is every infinite countable group dynamically compressible?
}
\begin{prob}\label{prob-dyncomp}
    \probdyncomp
\end{prob}

We note that $\Gamma$ fails to be dynamically compressible iff
there is some aperiodic $E_\Gamma^X$ such that every $E_\Gamma^X \le_B E_\Gamma^Y$ admits an invariant measure.

We conclude with the following interesting consequence of \cref{5.6}.
Let $\Gamma = \SL_3(\Z)$ and consider the shift action of $\Gamma$ on $\R^\Gamma$ and denote by $E= F(\Gamma, \R)$ the restriction of $E(\Gamma, \R)$ to the free part of the action. Then, by \cref{5.6}, $E\times I_\N$ is induced by a Borel action of $\Gamma$. On the other hand, $E\times I_\N$ cannot be induced by a \textit{free} Borel action of $\Gamma$, since if that was the case then $E\times I_\N\sqsubseteq^i_B E$, contradicting the Addendum following \cite[5.28]{CK18}.

\chapter{Open Problems}\label{problems}
For the convenience of the reader, we collect here some of the main open problems discussed in this paper.

\begin{prob}[\cref{prob-treeable}]
    \probtreeable
\end{prob}

\begin{prob}[\cref{prob-realization}]
    \probrealization
\end{prob}

\begin{prob}[\cref{prob-cantor}]
    \probcantor
\end{prob}

\begin{prob}[\cref{prob-bired}]
    \probbired
\end{prob}
Note that by \cref{322},
every non-smooth $E \in \ap$ is Borel bireducible to some $F \in \ap$
which has a compact action realization
iff every non-smooth compressible $E \in \ap$ has a compact action realization.

\begin{prob}[\cref{prob-aperiodiccompact}]
    \probaperiodiccompact
\end{prob}

\begin{prob}[\cref{prob-conull}]
    \probconull
\end{prob}

\begin{prob}[\cref{prob-cbercpt}]
    \probcbercpt
\end{prob}

\begin{prob}[\cref{prob-actionparadox}]
    \probactionparadox
\end{prob}

\begin{prob}[\cref{prob-turingcompact}]
    \probturingcompact
\end{prob}

\begin{prob}[\cref{prob-turingbc}]
    \probturingbc
\end{prob}

\begin{prob}[\cref{prob-turingjump}]
    \probturingjump
\end{prob}

\begin{prob}[\cref{prob-twogenuniv}]
    \probtwogenuniv
\end{prob}

\begin{prob}[\cref{prob-cbershift}]
    \probcbershift
\end{prob}

\begin{prob}[\cref{prob-minuniv}]
    \probminuniv
\end{prob}

\begin{prob}[\cref{prob-compuniv}]
    \probcompuniv
\end{prob}

\begin{prob}[\cref{prob-hfcomeager}]
    \probhfcomeager
\end{prob}

\begin{prob}[\cref{prob-hfcomplexity}]
    \probhfcomplexity
\end{prob}

\begin{prob}[\cref{prob-bauer}]
    \probbauer
\end{prob}

\begin{prob}[\cref{prob-ksigma}]
    \probksigma
\end{prob}

\begin{prob}[\cref{prob-graphable}]
    \probgraphable
\end{prob}

\begin{prob}[\cref{prob-adequate}]
    \probadequate
\end{prob}

\begin{prob}[\cref{prob-hypgen}]
    \probhypgen
\end{prob}

\begin{prob}[\cref{prob-dyncomp}]
    \probdyncomp
\end{prob}

\appendix

\chapter{Amenable Actions}\label{appendix-amenable-actions}
The purpose of this appendix is to explain the following implications
for a continuous action $\Gamma\car X$
of a countable group on a Polish space.
Recall that $E_\Gamma^X$ is the induced orbit equivalence relation
and all the concepts in the diagram below are defined in
\ref{subshifts-space-props} of \cref{subshifts-space}.
\[
    \adjustbox{width={\textwidth},center}{
        \begin{tikzcd}
            \mqty{\text{$E_\Gamma^X$ hyperfinite} \\
                \text{+ amenable stabilizers}} \drar[Rightarrow] & & \\
            & \text{$\Gamma\car X$ Borel amenable}
                \dlar[Rightarrow] & \\
            \mqty{\text{$E_\Gamma^X$ amenable} \\
                \text{+ amenable stabilizers}} \dar[Rightarrow] & &\\
            \mqty{\text{$E_\Gamma^X$ measure-amenable} \\
                \text{+ amenable stabilizers}} \rar[Leftrightarrow]
            & \text{$\Gamma\car X$ measure-amenable} \rar[Leftrightarrow]
            & \text{$\Gamma\car X$ topologically amenable} \ar[uul, "\sigma-compact", Rightarrow, magenta]
        \end{tikzcd}
    }
\]

Denote by $\Prob(\Gamma)$ the closed $\Gamma$-subspace of
$[0, 1]^\Gamma$ consisting of probability measures on $\Gamma$.

\section{Borel amenability}
We first have the following result:

\begin{thm}\label{borAmenImplication}
    Let $\Gamma\car X$ be a Borel action
    of a countable group on a standard Borel space,
    and consider the following statements:
    \begin{enumerate}[label=(\roman*)]
        \item \label{hfStab}
            $E_\Gamma^X$ is hyperfinite
            and every stabilizer is amenable.
        \item \label{borAmenAct}
            $\Gamma\car X$ is Borel amenable.
        \item \label{amenStab}
            $E_\Gamma^X$ is amenable
            and every stabilizer is amenable.
    \end{enumerate}
    Then \ref{hfStab}$\implies$\ref{borAmenAct}$\implies$\ref{amenStab}.
\end{thm}
\begin{proof}
    Let $E := E_\Gamma^X$.
    \begin{itemize}[leftmargin=\oddsidemargin,align=left]
        \item[\ref{hfStab}$\implies$\ref{borAmenAct}:]
            Since $E$ is hyperfinite,
            it is amenable in a strong sense:
            there is a sequence $p_n \colon E\to [0,1]$ of Borel functions,
            such that $p_n^x$ is a probability measure supported on $[x]_E$,
            for every $(x, y)\in E$,
            we have $\|p_n^x - p_n^y\|_1 \to 0$,
            and additionally,
            for every $y$,
            there are only finitely many $x$ with $p_n^x(y) > 0$.
            
            Let $\alpha \colon E \to \Gamma$ be a Borel function
            such that for every $(x, y) \in E$,
            we have $y = \alpha^y_x\cdot x$
            and $\alpha^x_y \alpha^y_x = 1$.
            Write $\Gamma = \bigcup_n S_n$ as an increasing union of finite subsets.
            
            We claim that there is a sequence
            $q_n \colon X\to\Prob(\Gamma)$ of Borel functions
            with $q_n^x$ supported on $\Gamma_x$,
            such that for every $(x, y) \in E$ with $p_n^x(y) > 0$
            and every $\gamma \in S_n$,
            we have
            $\|q_n^y - \alpha^y_{\gamma\cdot x}\gamma\alpha^x_y \cdot q_n^y\|_1
            < \frac{1}{n}$.
            To see this,
            for every $y \in X$,
            by amenability of $\Gamma_y$,
            let $A_n^y$ be the least (in some enumeration)
            finite subset of $\Gamma$ such that $A_n^y \subseteq \Gamma_y$
            and
            \[
                \frac{|A_n^y \btu \alpha^y_{\gamma\cdot x}\gamma\alpha^x_y A_n^y |}
                    {|A_n^y|}
                < \frac{1}{n}
            \]
            for every $x \in [y]_E$ with $p_n^x(y) > 0$ and every $\gamma \in S_n$.
            Then let $q_n^y := \frac{1}{|A_n^y|}\mathbf 1_{A_n^y}$
            be the uniform distribution on $A_n^y$.
            Then
            \[
                \|q_n^y - \alpha^y_{\gamma\cdot x}\gamma\alpha^x_y \cdot q_n^y\|_1
                = \frac{\|\mathbf 1_{A_n^y} - \mathbf 1_{\alpha^y_{\gamma\cdot x}\gamma\alpha^x_y A_n^y}\|_1}{|A_n^y|}
                = \frac{|A_n^y \btu \alpha^y_{\gamma\cdot x}\gamma\alpha^x_y A_n^y|}{|A_n^y|}
                < \frac{1}{n}.
            \]
            
            Let $r_n \colon X\to\Prob(\Gamma)$ be defined by
            \[
                r_n^x(\gamma)
                = p_n^x(\gamma\cdot x)
                q_n^{\gamma\cdot x}(\gamma\alpha^x_{\gamma\cdot x}).
            \]
            Let $x\in X$ and $\gamma\in\Gamma$.
            Then
            \begin{align*}
                & \|r_n^{\gamma\cdot x} - \gamma\cdot r_n^x\|_1 \\
                = \; & \sum_{\delta\in\Gamma}
                    |r_n^{\gamma\cdot x}(\delta) - r_n^x(\gamma^{-1}\delta)| \\
                = \; & \sum_{\delta\in\Gamma}
                    \qty|p_n^{\gamma\cdot x}(\delta^{-1}\gamma\cdot x)
                        q_n^{\delta^{-1}\gamma\cdot x}
                        (\alpha^{\delta^{-1}\gamma\cdot x}_{\gamma\cdot x}\delta)
                         - p_n^x(\delta^{-1}\gamma\cdot x)
                        q_n^{\delta^{-1}\gamma\cdot x}
                        (\alpha^{\delta^{-1}\gamma\cdot x}_x \gamma^{-1}\delta)| \\
                = \; & \sum_{y \in [x]_E}
                    \sum_{\lambda \in \Gamma_y}
                        \qty|p_n^{\gamma\cdot x}(y)
                            q_n^y(\lambda)
                            - p_n^x(y)
                            q_n^y(\alpha^y_x \gamma^{-1} \alpha^{\gamma\cdot x}_y \lambda)| \\
                \le \; & \sum_{y \in [x]_E}
                         |p_n^{\gamma\cdot x}(y) - p_n^x(y)|
                         \sum_{\lambda \in \Gamma_y} q_n^y(\lambda)
                         + \sum_{y \in [x]_E} p_n^x(y)
                           \sum_{\lambda \in \Gamma_y}
                           |q_n^y(\lambda) - q_n^y(\alpha^y_x \gamma^{-1} \alpha^{\gamma\cdot x}_y \lambda)| \\
                = \; & \|p_n^{\gamma\cdot x} - p_n^x\|_1
                    + \sum_{y \in [x]_E} p_n^x(y)
                        \|q_n^y - \alpha^y_{\gamma\cdot x}\gamma\alpha^x_y \cdot q_n^y\|_1.
            \end{align*}
            If $\gamma \in S_n$,
            then the second term is bounded by
            \[
                \sum_{\substack{y \in [x]_E \\ p_n^x(y) > 0}} p_n^x(y)
                    \frac{1}{n}
                = \frac{1}{n}
                \to 0
            \]
            so the whole expression converges to $0$.
        \item[\ref{borAmenAct}$\implies$\ref{amenStab}:]
            Let $p_n \colon X \to \Prob(\Gamma)$ witness
            the Borel amenability of the action $\Gamma\car X$.
            
            To show that $E$ is amenable,
            define $q_n \colon E \to [0, 1]$ by
            \[
                q_n^x(y) :=
                \sum_{\substack{\gamma \in \Gamma \\ \gamma\cdot y = x}}
                p_n^x(\gamma).
            \]
            Now if $x \in X$ and $\gamma \in \Gamma$,
            then we have
            \begin{align*}
                \|q_n^{\gamma\cdot x} - q_n^x\|_1
                & = \sum_{y \in [x]_E}
                    \qty|\sum_{\substack{\delta \in \Gamma \\ \delta\cdot y = \gamma\cdot x}}
                    p_n^{\gamma\cdot x}(\delta)
                    - \sum_{\substack{\lambda \in \Gamma \\ \lambda\cdot y = x}}
                    p_n^x(\lambda)| \\
                & = \sum_{y \in [x]_E}
                    \qty|\sum_{\substack{\delta \in \Gamma \\ \delta\cdot y = \gamma\cdot x}}
                    p_n^{\gamma\cdot x}(\delta)
                    - \sum_{\substack{\delta \in \Gamma \\ \delta\cdot y = \gamma x}}
                    p_n^x(\gamma^{-1}\delta)| \\
                & \le \sum_{y \in [x]_E}
                    \sum_{\substack{\delta \in \Gamma \\ \delta\cdot y = \gamma\cdot x}}
                    |p_n^{\gamma\cdot x}(\delta) - p_n^x(\gamma^{-1} \delta)| \\
                & = \sum_{\delta \in \Gamma}
                    |p_n^{\gamma\cdot x}(\delta) - p_n^x(\gamma^{-1} \delta)| \\
                & = \|p_n^{\gamma\cdot x} - \gamma\cdot p_n^x\|_1 \\
                & \to 0.
            \end{align*}
            Thus $E$ is amenable.
            
            Now let $x \in X$.
            To see that $\Gamma_x$ is amenable,
            let $T$ be a transversal for the right cosets of $\Gamma_x$ in $\Gamma$,
            and define $q_n \in \Prob(\Gamma_x)$ by
            \[
                q_n(\gamma)
                := \sum_{t\in T} p_n^x(\gamma t).
            \]
            Then for every $\gamma \in \Gamma_x$,
            we have
            \begin{align*}
                \|q_n - \gamma\cdot q_n\|_1
                & = \sum_{\delta \in \Gamma_x}
                    |q_n(\delta) - q_n(\gamma^{-1} \delta)| \\
                & = \sum_{\delta \in \Gamma_x}
                    \qty|\sum_{t\in T} p_n^x(\delta t)
                    - \sum_{t \in T} p_n^x(\gamma^{-1} \delta t)| \\
                & \le \sum_{\delta \in \Gamma_x} \sum_{t\in T}
                    |p_n^x(\delta t) - p_n^x(\gamma^{-1} \delta t)| \\
                & = \sum_{\lambda \in \Gamma}
                    |p_n^x(\lambda) - p_n^x(\gamma^{-1} \lambda)| \\
                & = \|p_n^x - \gamma\cdot p_n^x\|_1 \\
                & = \|p_n^{\gamma\cdot x} - \gamma\cdot p_n^x\|_1 \\
                & \to 0.
            \end{align*}
            Thus $\Gamma_x$ is amenable.
    \end{itemize}
\end{proof}

\section{Measure amenability}

By \cref{borAmenImplication}
and the Connes-Feldman-Weiss theorem
(see \Cref{prelim-classes}),
we have the following analogue of \cite[Theorem A]{AEG94}
(see also \cite[Corollary 5.3.33]{AR00}):
\begin{thm}
    Let $\Gamma\car X$ be a Borel action
    of a countable group on a standard Borel space,
    and let $\mu$ be a Borel probability measure on $X$.
    Then the following are equivalent:
    \begin{enumerate}[label=(\roman*)]
        \item \label{muAmen}
            $\Gamma\car X$ is $\mu$-amenable.
        \item \label{muAmenStab}
            $E_\Gamma^X$ is $\mu$-amenable
            and $\mu$-a.e. stabilizer is amenable.
    \end{enumerate}
\end{thm}

\begin{cor}\label{amenStabChar}
    Let $\Gamma\car X$ be a Borel action
    of a countable group on a standard Borel space.
    Then the following are equivalent:
    \begin{enumerate}[label=(\roman*)]
        \item
            $\Gamma\car X$ is measure-amenable.
        \item
            $E_\Gamma^X$ is measure-amenable
            and every stabilizer is amenable.
    \end{enumerate}
\end{cor}

\section{Topological amenability}\label{topamen}

Topological amenability is equivalent to measure-amenability
(see \cite[Theorem 3.3.7]{AR00} for the locally compact case,
also the proof of \cite[Proposition 5.2.1]{BO08}):
\begin{thm}\label{topAmenChar}
    Let $\Gamma\car X$ be a continuous action
    of a countable group on a Polish space.
    Then the following are equivalent:
    \begin{enumerate}[label=(\roman*)]
        \item \label{topAmen}
            $\Gamma\car X$ is topologically amenable.
        \item \label{measAmen}
            $\Gamma\car X$ is measure-amenable.
    \end{enumerate}
    Moreover,
    if $X$ is $\sigma$-compact,
    then these are also equivalent to
    \begin{enumerate}[resume, label=(\roman*)]
        \item \label{borAmenKsigma}
            $\Gamma\car X$ is Borel amenable.
    \end{enumerate}
\end{thm}

The following lemma says that in the definition of $\mu$-amenability,
we can upgrade the Borel functions to continuous ones:

\begin{lem}\label{muAmenCtsLem}
    Let $\Gamma\car X$ be a continuous action
    of a countable group on a Polish space,
    and let $\mu$ be a Borel probability measure on $X$.
    Then the following are equivalent:
    \begin{enumerate}[label=(\roman*)]
        \item
            $\Gamma\car X$ is $\mu$-amenable.
        \item
            For every finite $S\subseteq\Gamma$ and every $\epsilon > 0$,
            there is continuous $p \colon X \to \Prob(\Gamma)$
            such that for every $\gamma \in S$,
            we have
            \[
                \int_X \|p^{\gamma\cdot x} - \gamma\cdot p^x\|_1 \,\dd\mu(x)
                < \epsilon.
            \]
    \end{enumerate}
\end{lem}
\begin{proof}[Proof of \cref{muAmenCtsLem}]
    It suffices to show that for every Borel
    $p \colon X\to \Prob(\Gamma)$ and every $\epsilon > 0$,
    there is some continuous $q \colon X \to \Prob(\Gamma)$
    such that
    \[
        \int_X \|p - q\|_1 \, \dd\mu < \epsilon.
    \]
    By Lusin's theorem \cite[17.12]{Kec95},
    there is a closed $F\subseteq X$
    with $\mu(F) > 1 - \frac{\epsilon}{2}$
    such that $p\uhr F$ is continuous.
    By Dugundji's extension theorem \cite[4.1]{Dug51},
    there is some continuous extension
    $q \colon X \to \Prob(\Gamma)$ of $p\uhr F$.
    Then $p$ and $q$ agree on $F$,
    so we are done.
\end{proof}

\begin{proof}[Proof of \cref{topAmenChar}]
    \leavevmode
    \begin{itemize}[leftmargin=\oddsidemargin,align=left]
        \item[\ref{topAmen}$\implies$\ref{measAmen}:]
            This follows from tightness of Borel probability measures,
            see \cite[17.11]{Kec95}.
        \item[\ref{measAmen}$\implies$\ref{topAmen}:]
            Let $S \subseteq \Gamma$ be finite
            and let $K \subseteq X$ be compact.
            Denote below  by $C(X, \Prob(\Gamma))$
            the set of continuous functions $X\to\Prob(\Gamma)$,
            and define $\Psi \colon C(X, \Prob(\Gamma)) \to C(K)$ by
            \[
                \Psi_p(x)
                = \sum_{\gamma\in S}
                \|p^{\gamma\cdot x} - \gamma\cdot p^x\|_1.
            \]

            By measure-amenability and \cref{muAmenCtsLem},
            for every Borel probability measure $\mu$ on $K$,
            we have
            \[
                \inf_{f\in \im\Psi} \int_K f \, \dd\mu
                = 0.
            \]
            So by the Riesz representation theorem for $C(K)$,
            for every functional $\phi \in C(K)^*$,
            we have
            \[
                \inf_{f\in \im\Psi} |\phi(f)| = 0.
            \]
            Thus by the Hahn-Banach separation theorem,
            we have
            \[
                \inf_{f \in \Conv(\im\Psi)} \|f\|_\infty
                = 0,
            \]
            where $\Conv(\im\Psi)$ denotes the convex hull of $\im\Psi$.
            Since
            \[
                \Psi_{\sum_{i < k} \alpha_i p_i}
                \le \sum_{i < k}\alpha_i \Psi_{p_i},
            \]
            we have
            \[
                \inf_{f \in \im\Psi} \|f\|_\infty
                = 0,
            \]
            so we are done.
    \end{itemize}
    
    Now suppose that $X$ is $\sigma$-compact.
    It suffices to show \ref{topAmen}$\implies$\ref{borAmenKsigma}.
    Write $\Gamma = \bigcup_n S_n$ as an increasing union of finite subsets,
    and write $X = \bigcup_n K_n$ as an increasing union of compact subsets.
    Then for each $n$,
    by topological amenability,
    there is some continuous $p_n \colon X \to \Prob(\Gamma)$
    such that
    \[
        \max_{\substack{\gamma \in S_n \\ x \in K_n}}
        \|p_n^{\gamma\cdot x} - \gamma\cdot p_n^x\|_1
        < \frac{1}{n}.
    \]
    Then $(p_n)_n$ witnesses Borel amenability of $\Gamma\car X$.
\end{proof}

\chapter{Weak Containment}\label{appendix-weak-containment}
We give here the proof of \cref{395}. 

\begin{enumerate}[label=(\arabic*)]
    \item
        Note that if $h \in \Homeo(\cantor)$, then $h$ is free (i.e., $h(x) \neq x,\forall x$) iff there is a finite sequence $(K_i)_{i=1}^n$ of clopen sets such that $\cantor = \bigcup_{i=1}^nK_i$ and $h(K_i)\cap K_i =\emptyset, \forall i$. From this is follows easily that if $\ac a \preceq {\ac b}$ and $\ac a$ is free, then so is ${\ac b}$, recalling that if $h_n, h \in \Homeo(\cantor)$ and $h_n\to h$, then for any clopen set $K$, for all large enough $n$, $h_n (K) = h(K)$.
    \item
        By Nadkarni's theorem (\Cref{nadkarni}), it is enough to show that if $\ac a \preceq {\ac b}$ and ${\ac b}$ admits an invariant (Borel probability) measure, then so does $\ac a$. We have that $\ac a$ is the limit of a sequence of actions $\ac a_n$ that admit invariant measures $\mu_n$, which by compactness in the space of measures we can assume that they converge to a measure $\mu$.  We will check that $\mu$ is invariant for $\ac a$. Let for $\gamma \in \Gamma$, $\gamma^{\ac a} (x) = \ac a(\gamma, x)$, and similarly for $\ac a_n$, so that $\gamma^{\ac a_n} \to \gamma^{\ac a}$. Fix a clopen set $K$. Then for all large enough $n$, we have that $\gamma^{\ac a} (K) =  \gamma^{\ac a_n} (K)$. Then $\mu (\gamma^{\ac a} (K) )= \lim_n \mu_n (\gamma^{\ac a} (K)) = \lim_n \mu_n (\gamma^{\ac a_n} (K)) ) = \lim_n \mu_n (K) = \mu (K)$.
    \item
        We will use the following lemma (see \cite[Lemma 3.3]{Ele19}).
        \begin{lem}\label{B01}
            Let $p \colon \cantor \to \Prob(\Gamma)$ be continuous. Then for each $\epsilon >0$, there are finite $F\subseteq \Gamma$ and continuous $q \colon \cantor \to \Prob(\Gamma)$ such that for all $x\in\cantor$ and $\gamma\notin F$, we have that $q^x (\gamma) =0$, and moreover $||p^x - q^x||_1 <\epsilon$.
        \end{lem}
        \begin{proof}
            Let $(F_n)$ be an increasing sequence of finite subsets of $\Gamma$ with $1\in F_0$ that covers $\Gamma$. Then there must exist $n$ such that for all $x\in \cantor$, $p^x (F_n) > 1-\frac{\epsilon}{3}$. Indeed otherwise there is a sequence of points $x_n$ such that $p^{x_n}(F_n) \le 1-\frac{\epsilon}{3}$. Then there is a subsequence of $x_n$ converging to some $x$ for which $p^x (\Gamma) \le 1-\frac{\epsilon}{3}$, a contradiction. Take $F=F_n$.
        
            Define now $q$ as follows: If $\gamma \in F\setminus \{1\}$, then $q^x (\gamma) = p^x (\gamma)$; $q^x(\gamma) = 0$, if $x\notin F$; $q^x (1) = p^x (1) + p^x( \Gamma\setminus F)$.
        \end{proof}
\end{enumerate}
Let now $\ac a\preceq {\ac b}$ and assume that $\ac a$ is topologically amenable. Fix a finite subset $S\subseteq \Gamma$ and $\epsilon >0$. Then there is a continuous $p \colon \cantor \to \Prob (\Gamma)$ such that 
\[
    \max_{\substack{\gamma \in S \\ x \in \cantor}}
    \|p^{\ac a(\gamma, x)} - \gamma\cdot p^x\|_1
    < \epsilon.
\]
By \cref{B01} we can also assume that there is finite $F\subseteq \Gamma$ such that $p^x(\gamma) =0, \forall \gamma \notin F$. Since $\ac a$ is a limit of conjugates of ${\ac b}$, it follows that there is a conjugate ${\ac c}$ of ${\ac b}$, say via the homeomorphism $h$ of $\cantor$, such that

\[
    \max_{\substack{\gamma \in S \\ x \in \cantor}}
    \|p^{{\ac c}(\gamma, x)} - p^{\ac a(\gamma, x)}    
    \|_1
\]
can be made as small as we want, so that 
\[
    \max_{\substack{\gamma \in S \\ x \in \cantor}}
    \|p^{{\ac c}(\gamma, x)} - \gamma\cdot p^x\|_1
    < \epsilon.
\]
Then if $q^x= p^{h(x)}$, we have that $q$ is continuous and 
\[
    \max_{\substack{\gamma \in S \\ x \in \cantor}}
    \|q^{{\ac b}(\gamma, x)} - \gamma\cdot q^x\|_1
    < \epsilon.
\]
and the proof is complete.

\chapter{The Correspondence Theorem of Hochman}\label{appendix-correspondence-theorem}
We give here the proof of \cref{correspondence}.
Below for each Polish space $X$,
we denote by $K(X)$ the Polish space of
all compact subsets of $X$ with the Vietoris topology,
which is compact if $X$ is compact.
If $d \le 1$ is a compatible metric on $X$,
let $d_H$ be the corresponding Hausdorff metric for $K(X)$,
which induces the Vietoris topology (see \cite[4.F]{Kec95}).
Fix a compatible metric $d \le 1$ for $\I^\N$
with corresponding Hausdorff metric $d_H$ for $K(\I^\N)$,
the product metric $d_1$ for $(\I^\N)^\Gamma$ given by
$d_1((x_\gamma)_\gamma, (y_\gamma)_\gamma)
= \sum^\infty_{i=0} 2^{-i} d(x_{\gamma_i}, y_{\gamma_i})$,
where $\Gamma = \{\gamma_i : i\in \N \}$,
with $\gamma_0 =1$,
and the corresponding Hausdorff metric $d_{1,H}$ on
$(K(\I^\N))^\Gamma$.

The idea is to construct a Polish space $E$ and a topological embedding 
\[
    \pi \colon \Act(\Gamma, \cantor) \times E \to \Sh(\Gamma, \I^\N),
\]
so that for all $\ac a \in \Act(\Gamma, \cantor)$ and $e\in E$, the subshift $\pi (\ac a, e)$ is topologically isomorphic to $\ac a$ and moreover the range of $\pi$ is dense (and therefore dense $G_\delta$) in $\Sh(\Gamma, \I^\N)$. The Correspondence Theorem follows immediately.

To construct $\pi$ we will need a series of lemmas. Below denote by $D$ the set of topological embeddings of $\cantor$ to $\I^\N$. This is a $G_\delta$ subset of the Polish space of all continuous maps from $\cantor$ to $\I^\N$  with the sup metric, so it is a Polish space.

\begin{lem}\label{C01}
There is a $G_\delta$ subset $E \subseteq D$ so that the map $e\in E\mapsto e(\cantor) \in K(\I^\N)$ is a topological embedding with dense (and therefore dense $G_\delta$) image in $K(\I^\N)$.
\end{lem}
\begin{proof}
The set ${\mathcal C}$ of all Cantor sets in $K(\I^\N)$ is a dense $G_\delta$ subset of $K(\I^\N)$ (see \cite[8.8]{Kec95}). Moreover by the Jankov-von Neumann Uniformization Theorem (see, \cite[18.1]{Kec95}) there is a Baire measurable map $f \colon {\mathcal C}\to D$ such that if $f(C) = e$, then $e(\cantor) = C$. Then by \cite[8.38]{Kec95}, there is a dense $G_\delta$ subset ${\mathcal C}'$ of ${\mathcal C}$ such that $f$ is continuous on ${\mathcal C}'$. Clearly $f|{\mathcal C}'$ is a topological embedding and we can take $E = f({\mathcal C}')$. 
\end{proof}
For each action $\ac a \in \Act(\Gamma, C)$, where $C$ is a Cantor space, and topological embedding $e$ of $C$ into $\I^\N$,  there is a canonical subshift in $\Sh(\Gamma, \I^\N)$, denoted by $e^*_{\ac a}$, defined as follows: Let $f_{\ac a} \colon C \to (\I^\N)^\Gamma$ be defined by:
\[
f_{\ac a} (x) _\gamma = e(\ac a(\gamma^{-1}, x)).
\]
We put $e^*_{\ac a } = f_{\ac a}(C)$. This is a subshift in $\Sh(\Gamma, \I^\N)$ topologically isomorphic to $\ac a$. Finally we define for $\ac a \in \Act(\Gamma, \cantor), e\in E$,
\[
\pi(\ac a, e) = e^*_{\ac a }.
\]
Clearly $\pi(\ac a, e)$ is isomorphic to $\ac a$, so it only remains to show that $\pi$ is a topological embedding with dense range. 

\begin{lem}
    The map $\pi$ is a topological embedding.
\end{lem}
\begin{proof}
    It is clear from \cref{C01} that $\pi$ is injective and continuous.
    We will verify next that its inverse is continuous.
    
    Assume that $\pi(\ac a_n, e_n) \to \pi(\ac a, e)$,
    in order to show that $\ac a_n \to \ac a$ and $e_n\ \to e$. 

    \begin{enumerate}[label=(\arabic*)]
        \item
            $e_n\to e$:
            Let $\pi_1$ be the projection map $\pi_1 \colon (\I^\N)^\Gamma \to \I^\N$ given by $\pi_1 (x) = x(1)$. Then $\pi_1 (e^*_\ac a) = e(\cantor)$. Since $\pi(\ac a_n, e_n) \to \pi(\ac a, e)$, by applying $\pi_1$ we have that  $e_n (\cantor)\to e(\cantor)$, therefore $e_n\to e$ by \cref{C01}.
    
        \item
            $\ac a_n \to \ac a$:
            Since $e_n\to e$ uniformly, it is clear that $\pi(\ac a_n, e) \to \pi(\ac a, e) $. Replacing $\cantor$ with $C = e(\cantor)$ and each action on $\cantor$ by its copy on $C$ via $e$, we can assume that actually $\ac a_n, \ac a\in \Act(\Gamma, C)$ and if $i_C$ is the identity on $C$, then we have that  $(i_C)^*_{\ac a_n} \to (i_C)^*_{\ac a}$. To show that  $\ac a_n \to \ac a$ we need to show that for each $\gamma \in \Gamma$, $\ac a_n(\gamma, c) \to  \ac a(\gamma, c)$, uniformly for $c\in C$. 
        
            Fix $\epsilon >0$ and then find $\delta < \frac{\epsilon}{k_0}$, where $k_0$ is some positive constant, defined below, depending only on $\gamma$, such that 
            \[
            d(c,d) < \delta \implies d(\ac a(\gamma, c), \ac a(\gamma, d)) < \epsilon.
            \]
            Next fix $N$ such that for all $n \ge N$, we have that $d_{1,H}((i_C)^*_{\ac a_n}, (i_C)^*_{\ac a}) < \delta$. Let now $c\in C$. Then there is $c'\in C$ such that 
            \[
            d_1((\ac a_n(\gamma, c))_\gamma, (\ac a(\gamma, c'))_\gamma) < \delta
            \]
            and therefore $d(c,c') < \delta$ and $d( \ac a_n (\gamma, c), d( \ac a (\gamma, c') )< k_0\delta$, for some constant $k_0 >0$ depending only on $\gamma$. Thus $d(\ac a(\gamma, c), \ac a(\gamma, c')) < \epsilon$ and so $d(\ac a_n(\gamma, c), \ac a(\gamma, c)) < \epsilon + k_0 \delta < 2\epsilon$ and the proof is complete.
    \end{enumerate}
\end{proof}

We next need the following general topological lemma. Below $C$ will be a Cantor space with compatible metric $d_C$, $X$ a perfect Polish space with metric $d$ and ${\mathcal K}'$ is a dense subset of $K(X)$ consisting of Cantor sets.

\begin{lem}\label{C03}
    If $f \colon C \to X$ is such that 
    \[
        d(f(c), f(d)) \le d_C(c, d),
    \]
    for all $c, d \in C$,
    then for every $\epsilon > 0$,
    there is $K \in \mathcal K'$ and a homeomorphism $g \colon C \to K$ such that
    \[
        d(f(c), g(c)) < \epsilon, \forall c \in C.
    \]
\end{lem}
\begin{proof}
    Let $C =\bigsqcup^n_{i=1} C_i$ be a partition of $C$ into nonempty clopen sets with $d_C (C_i) < \frac{\epsilon}{3}, 1 \le i \le n$. Let $c_i \in C_i, 1 \le i \le n$. Recursively define distinct $x_i\in X, 1 \le i \le n$, such that $d(f(c_i), x_i)< \frac{\epsilon}{3}$. Let $\delta$ be the minimum of the distances $d (x_i, x_j)$, for $1 \le i \neq j \le n$. Then find $K\in {\mathcal K}'$ such that 
    \[
    d_H (K, \{x_1, \dots , x_n\}) < \rho \le \min\{\frac{\delta}{2}, \frac{\epsilon}{3}\}.
    \]
    Then for each $k\in K$, there is \textit{unique} $1 \le i \le n$ such that $d (k, x_i) < \rho$. Put $K_i = K \cap B_{x_i} (\rho)$, where $B_{x_i}(\rho) $ is the open ball of radius $\rho$ with center $x_i$ in the metric $d$. Then 
    
    \begin{enumerate}[label=(\arabic*)]
        \item
            $K =\bigsqcup^n_{i=1} K_i$,
        \item
            For each $1 \le i \le n$,
            $K_i \neq \emptyset$,
            since there is some $k\in K$ with $d (k, x_i ) < \rho$,
        \item
            For each $1 \le i \le n$, $K_i$ is clopen in $K$.
    \end{enumerate}
    
    Thus $K =\bigsqcup^n_i K_i$ is a clopen partition of $K$ into Cantor sets. Therefore there is a homeomorphism $g_i \colon C_i \to K_i$, $1 \le i \le n$. Put $g =\bigsqcup^n_{i=1} g_i$ so that $g \colon C \to K$ is a homeomorphism. 
    
    Let now $c\in C$. Then for some (unique)  $1 \le i \le n$, $c\in C_i$, so $d_C (c,c_i) < \frac{\epsilon}{3}$. Thus $d ( f(c), f(c_i)) < \frac{\epsilon}{3}$. Also $d ( f(c_i), x_i) < \frac{\epsilon}{3}$ and $g(c) = g_i (c) \in K_i$, so $d_X (g(c), x_i) < \rho < \frac{\epsilon}{3}$. Thus $d ( f(c), g(c) ) < \epsilon$.
\end{proof}

\begin{lem}\label{C04}
Let ${\mathcal K}_0$ be a dense subset of $K(\I^\N)$ consisting of Cantor sets. Then for each Cantor subshift $K$ in $\Sh(\Gamma, \I^\N)$, $K$ is the limit of a sequence  $K_n =(i_{C_n})^*_{\ac a_n}$, with $C_n \in {\mathcal K}_0$, $\ac a_n \in \Act(\Gamma, C_n)$.
\end{lem}
\begin{proof}
Let $\epsilon >0$. Let $K_1 = \pi_1 (K)$. Then by \cref{C03} (with $C= K, X=\I^\N, f =\pi_1, {\mathcal K}' = {\mathcal K}_0$),  there is a homeomorphism $g \colon K \to C_0$, where $C_0\in {\mathcal K}_0$ and $d(x(1), g(x))< \epsilon$, for all $x\in K$. Push the shift action on $K$ via $g$ to the action $\ac a\in \Act(\Gamma, C_0)$. Let $L= (i_{C_0})_{\ac a}(C_0)$ and put $ j = (i_{C_0})_{\ac a}\circ g \colon K \to L$. Then for all $x\in  K$,
\[
d(x(1), j(x)(1)) = d(x(1), g(x)) <\epsilon.
\]
and so for every $\gamma \in \Gamma, x\in K$,
\[
d(x(\gamma), j(x)(\gamma)) = d({\ac s} (\gamma^{-1}, x)(1), (i_{C_0})_{\ac a} (g(x))(\gamma)) = d({\ac s} (\gamma^{-1}, x)(1), g({\ac s} ( \gamma^{-1}, x)) <\epsilon,
\]
where ${\ac s}$ is the shift action. Thus $d_{1,H} (K, L)  <\epsilon$ and the proof is complete.
\end{proof}

Next note the following:

\begin{lem}
The Cantor subshifts are dense in $\Sh(\Gamma, \I^\N)$.
\end{lem}
\begin{proof}
Let $K\in \Sh(\Gamma, \I^\N)$ and let $D= \pi_1 (K)$, which is a compact subset of $\I^\N$. Let then $C$ be a Cantor set such that the Hausdorff distance of $C$ from $D$ is less than $\frac{\epsilon}{2}$. Put 
\[
K'_0 = \{x'\in C^\Gamma : \exists x\in K(d(x(\gamma), x'(\gamma)) < \frac{\epsilon}{2}, \forall \gamma \in \Gamma )\}.
\]
and let $K'$ be the closure of $K'_0$. Then $K'$ is a Cantor subshift with $d_{1,H}(K, K') < \epsilon$.
\end{proof}

We finally complete the proof of the Correspondence Theorem by showing that $\pi$ has dense range.

Let $K$ be a Cantor subshift. Let ${\mathcal K}_0 = \{ e(\cantor) : e\in E\}$, which is dense in $K(\I^\N)$ by \cref{C01}. Then by \cref{C04}, we can find $K_n \to K$ with  $K_n =(i_{C_n})^*_{\ac a_n}$, where $C_n \in {\mathcal K}_0$ and $\ac a_n \in \Act(\Gamma, C_n)$. Let $e_n\in E$ be a homeomorphism of $\cantor$ with $C_n$ and pull back the action $ {\ac a_n}$ via $e_n$ to the action ${\ac b}_n \in \Act(\Gamma, \cantor)$. Then clearly $\pi({\ac b_n}, e_n) = K_n$ and we are done.

\begin{remark}
    Let $\Phi$ be a property of continuous $\Gamma$-actions on Polish spaces, which is invariant under (topological) isomorphism. Let $\mathcal{C}$ be a class of sets in Polish spaces, which is closed under continuous preimages. Then the preceding results in this Appendix show that if $\Sh_\Phi(\Gamma, \I^\N)$ is in the class $\mathcal{C}$, then $\Act_\Phi(\Gamma, \cantor)$ is also in the class $\mathcal{C}$. It is not clear if the converse holds.
\end{remark}

\backmatter

\bibliographystyle{amsalpha}
\bibliography{realizations_cber}

\end{document}